\documentclass{imsart}
\usepackage[english]{babel}
\usepackage{natbib}
\usepackage{amsthm}
\usepackage{graphicx}
\usepackage{amsmath}
\usepackage{amsthm}
\usepackage{amssymb}
\usepackage{bm}
\usepackage{algorithmic}
\usepackage[ruled,vlined]{algorithm2e}
\usepackage{braket}
\usepackage{comment}
\usepackage{mathtools}
\usepackage[utf8]{inputenc}
\usepackage[mathscr]{eucal}
\usepackage{fancyhdr}
\usepackage{fancybox}
\usepackage{xcolor}
\usepackage{longtable}
\usepackage{dsfont}
\usepackage{makecell}
\usepackage{float}
\usepackage{multirow}
\usepackage{svg}
\usepackage{physics}
\definecolor{darkgreen}{rgb}{0.0, 0.4, 0.0}

\usepackage{pifont}

\usepackage{caption}
\usepackage{subcaption}

\usepackage{xr-hyper}
\usepackage[colorlinks=true, allcolors=blue]{hyperref}
\makeatletter
\newcommand*{\addFileDependency}[1]{
\typeout{(#1)}
\@addtofilelist{#1}
\IfFileExists{#1}{}{\typeout{No file #1.}}
}\makeatother

\newcommand*{\myexternaldocument}[1]{%
\externaldocument{#1}%
\addFileDependency{#1.tex}%
\addFileDependency{#1.aux}%
}

\myexternaldocument{AOS-Supplementary-REV}

\newcommand{\Argmin}[1]{\underset{#1}{\mathrm{Argmin}}}
\def\HK{\HH_K}
\def\HEK{\HH_{\cE,K}}
\def\Hk{\HH_k}

\def\N{\mathbb{N}}
\def\bE{\mathbb{E}}
\def\bP{\mathbb{P}}

\def\R{\mathbb{R}}

\def\HH{\mathcal{H}} 
\def\L2mu{L^(\mu)} 

\def\cN{\mathcal{N}}
\def\SW{\mathcal{SW}}
\def\cPO{\mathcal{P}(\Omega)}

\def\cS{\mathcal{S}}
\def\cU{\mathcal{U}}
\def\cC{\mathcal{C}}
\def\cA{\mathcal{A}}
\def\cE{\mathcal{E}}
\def\cB{\mathcal{B}}
\def\cA{\mathcal{A}}
\def\cL{\mathcal{L}}
\def\cK{\mathcal{K}}
\def\cO{\mathcal{O}}

\def\cte{\mathrm{cst}}

\def\tell{\tilde{\ell}}

\newtheorem{theorem}{Theorem}[section]

\newtheorem{condition}[theorem]{Condition}

\newtheorem{lemma}[theorem]{Lemma}
\newtheorem{corollary}[theorem]{Corollary}
\newtheorem{remark}[theorem]{Remark}

\numberwithin{equation}{section}

\begin{document}

\begin{frontmatter}
\title{Improved learning theory for kernel distribution regression with two-stage sampling}

\runtitle{Distribution regression}
\runauthor{Bachoc, B\'ethune, González-Sanz and Loubes}
\begin{aug}

\author[A]{\fnms{Fran\c cois} \snm{Bachoc}\ead[label=e1]{Francois.Bachoc@math.univ-toulouse.fr}},
\author[B]{\fnms{Louis} \snm{B\'ethune}\ead[label=e2]{l\_bethune@apple.com}},
\author[C]{\fnms{Alberto} \snm{Gonz\'alez-Sanz}\ead[label=e3]{ag4855@columbia.edu}}
\and
\author[D]{\fnms{Jean-Michel} \snm{Loubes}\ead[label=e4]{loubes@math.univ-toulouse.fr}}
\address[A]{IMT, Université Paul-Sabatier \& Institut universitaire de France (IUF), Toulouse, France}
\address[B]{Apple, Paris, France}
\address[C]{Department of Statistics, Columbia University, New York, United States}
\address[D]{Regalia Team, INRIA \& IMT, France}
~ \\
\printead{e1}
\printead{e2}
\printead{e3}
\printead{e4}
\end{aug}

\begin{abstract}
The distribution regression problem encompasses many important statistics and machine learning tasks,
and arises in a large range of applications.
Among various existing approaches to tackle this problem, kernel methods have become a method of choice.
Indeed, kernel distribution regression is both computationally favorable, and supported by a recent learning theory. 
This theory also tackles the two-stage sampling setting, where only samples from the input distributions are available.
In this paper, we improve the learning theory of kernel distribution regression. 
We address kernels based on Hilbertian embeddings, that encompass most, if not all, of the existing approaches.
We introduce the novel near-unbiased condition on the Hilbertian embeddings, that enables us to provide new error bounds on the effect of the two-stage sampling, thanks to a new analysis.
We show that this near-unbiased condition holds for three important classes of kernels, based on optimal transport and mean embedding. 
As a consequence, we strictly improve the existing convergence rates for these kernels.
Our setting and results are illustrated by numerical experiments.
\end{abstract}

\begin{keyword}
\kwd{Distribution regression, kernel learning, optimal transport}
\end{keyword}

\end{frontmatter}

\section{Introduction} \label{section:introduction}

\subsection{Hilbertian embeddings for distribution regression} 
\label{subsection:intro:distribution:regression}

In this work, our objective is to address the regression problem where the inputs belong to probability distribution spaces and the outputs are real-valued observations,
\begin{equation}\label{eq:regression:problem}
    Y_i = f^{\star}(\mu_i) + \epsilon_i,
\end{equation}
for $i=1,\ldots,n$, where $(\mu_i)_{i=1}^n$ represent probability distributions on a generic space $\Omega \subset \R^d$, while $(Y_i)_{i=1}^n$ denote real numbers.
The pairs $(\mu_i, Y_i)$ are i.i.d. and $f^{\star}(\mu_i)$ is the conditional expectation of $Y_i$ given $\mu_i$, or equivalently, in the above display, $\bE[ \epsilon_i | \mu_i ] = 0$. The goal is to learn the unknown real-valued function $f^{\star}$ based on the observations $\left( \mu_i, Y_i\right)_{i=1}^n$.
 
This problem of learning functions over spaces of probability measures, known as  {\it distribution regression}, has received much attention over the last years. 
Distribution regression enables to handle more data variability as standard regression and  has proved its capacity to model complex problems, for instance in image analysis, physical science, meteorology, sociology or econometry.
We refer for instance to \citet{pmlr-vR5-hein05a,muandet2012learning,poczos2013distribution,oliva2014fast,szabo2015two,szabo2016learning,thi2021distribution,meunier2022distribution} and references therein.

Kernel ridge regression, see for instance \citet{kimeldorf1971some}, \citet[Eq. (4.6)]{scholkopf2002learning} and \citet[Sect. 5.8.2]{hastie2009elements}, is attractive for distribution regression, provided a suitable kernel operating on distributions is available. 
\textcolor{black}{There is a rich literature on the construction of such a kernel, see in particular 
\cite{Gartner2002,pmlr-v108-buathong20a} (on related kernels on finite sets), 
\cite{pmlr-vR5-hein05a} and \cite{ziegel2024characteristic}.
Here we shall focus on the concept of {\it Hilbertian embedding}, 
exploited by recent contributions on distribution regression \citep{smola2007hilbert,
szabo2015two,
szabo2016learning,muandet2017kernel,meunier2022distribution}, and
inspired by classical works on functional data regression, see for instance \cite{ramsay2007applied}. With Hilbertian embedding,} distributions are embedded into a Hilbert space, on which standard kernels are available, thus extending statistical learning theory to distributional data.

Constructing the Hilbertian embedding is a major challenge, which amounts to finding a suitable representation capturing all relevant properties of the underlying distributions.
The historical and most common 
approach is provided by kernel mean embeddings, where choosing a kernel operating on the input space $\Omega$ enables to associate, to each distribution, an element of the corresponding reproducing kernel Hilbert space (RKHS). 
For further insights into the theoretical properties of distribution regression with mean embedding, we refer to \cite{muandet2017kernel}. 

Another recent line of approach for Hilbertian embedding is based on optimal transport theory (see for instance \citet{Villani2003,panaretos2020invitation}), using mostly the Wasserstein distance. 
For univariate distributions, standard functions such as the squared exponential $t \mapsto e^{-t^2}$ can be applied to the Wasserstein distance and yield a kernel (a non-negative definite function). This is because the Wasserstein distance can be associated to the Hilbertian embedding obtained by taking the quantile functions \citep{bachoc2017gaussian}, and is thus specific to the univariate case. 
A popular extension to the multidimensional case is given by sliced Wasserstein kernels, associated to sliced Wasserstein distances \citep{kolouri2016sliced,peyre2016gromov}. This provides a Hilbertian embedding based on a family of quantile functions indexed by directions in $\R^d$ (see \cite{meunier2022distribution} and Section \ref{subsection:application:sliced:wasserstein}). 

An alternative extension of Wasserstein kernels from the univariate to the multivariate case is based on optimal transport maps. In \cite{bachoc2020gaussian}, a reference distribution is selected, and each distribution is associated to the optimal transport map from the reference distribution to itself. Hence, this constitutes a Hilbertian embedding where the Hilbert space consists of squared-summable functions with respect to the reference distribution. Last, the very recent reference \cite{bachoc2023gaussian} extends this approach by replacing the standard optimal transport problem by the regularized one, corresponding to the Sinkhorn distance \citep{cuturi2013sinkhorn}.
This brings strong computational benefits.
Note that the kernels obtained by \cite{bachoc2023gaussian}, as well as many others from the previous references, are universal kernels in the sense described in \cite{christmann2010universal} and are thus suitable to address wide classes of regression functions $f^{\star}$ in \eqref{eq:regression:problem}. 

\subsection{Two-stage sampling and existing convergence rates}
\label{subsection:two:stage}

Thanks to Hilbertian embedding, distribution regression can be tackled by kernel ridge regression, with kernels operating on Hilbert spaces.
This yields an estimated regression function $\hat{f}_n$ based on \eqref{eq:regression:problem}.
A general theory encompassing kernel ridge regression on Hilbert spaces is developed in \cite{caponnetto2007optimal} and yields minimax convergence rates on $\hat{f}_n - f^{\star}$ as $n \to \infty$.
These rates apply to the distribution regression methods discussed in Section \ref{subsection:intro:distribution:regression}.

Nevertheless, a limitation of \cite{caponnetto2007optimal} is that the measures $\mu_1,\ldots,\mu_n$ should be observed exactly for computing $\hat{f}_n$.
However, in many practical situations there is a {\it two-stage sampling} setting, where for $i  = 1 , \ldots , n$, only an i.i.d. sample $(X_{i,j})_{j=1}^N$ following the distribution $\mu_i$ is observed. Thus the first-stage sample is $\mu_1,\ldots,\mu_n$ (i.i.d. and unobserved) and the second-stage sample is $(X_{i,j})_{i=1,\ldots,n,j=1,\ldots,N}$. 
The data $(X_{i,j})$ and $(Y_{i})$ are sufficient to construct a second estimated regression function $\hat{f}_{n,N}$.

For Hilbertian embeddings based on mean embeddings, \citet{szabo2015two,szabo2016learning} provide upper bounds on $\hat{f}_{n,N} - f^{\star}$, as $n,N 
\to \infty$, building on the analysis of \cite{caponnetto2007optimal}. \cite{,meunier2022distribution} proceed similarly for Hilbertian embeddings based on the sliced Wasserstein distance. 

\subsection{Contributions and outline}

 In this work, we provide a general learning theory of distribution regression with two-stage sampling. 
 First, we consider a general kernel ridge regression setting with inputs $(x_i)_{i=1}^n$ belonging to a Hilbert space. 
 These inputs are not observed, but noisy versions of them are, $(x_{N,i})_{i=1}^n$, where the accuracy of $x_{N,i}$ increases with $N$.
 The exact (respectively noisy) inputs yield the estimated regression function $\hat{f}_n$ (respectively $\hat{f}_{n,N}$).  
 We provide upper bounds on $\| \hat{f}_{n,N} - \hat{f}_n \|_{\HH_K}$ as $n , N \to \infty$, where $\HH_K$ is the RKHS defined by the kernel $K$ operating on the Hilbert space. 
 These upper bounds are based on a new analysis, that improves that made in \citet{szabo2015two,szabo2016learning,meunier2022distribution}. Indeed, these references address specific distribution regression settings, and are in fine aiming at studying $\hat{f}_{n,N} - f^{\star}$, but, in intermediate steps, they bound $\hat{f}_{n,N} - \hat{f}_n$ with arguments that are not restricted to their specific settings.
 Hence, the bounds from \citet{szabo2015two,szabo2016learning,meunier2022distribution} are available on $\hat{f}_{n,N} - \hat{f}_n$ for a general kernel ridge regression on a Hilbert space, and the bounds we provide improve them in many situations (see in particular Remark \ref{remark:comment:regarding:existing}).

 Our new analysis is based on the assumption that $x_{N,i}$ is near unbiased for $x_i$, which we call the {\it near-unbiased condition}. 
 This condition enables us to exhibit sums of independent centered real-valued random variables, that did not appear in \citet{szabo2015two,szabo2016learning,meunier2022distribution}.
These sums are obtained thanks to coupling arguments, and the fact that they are real-valued (not Hilbert-valued) is permitted thanks to a new line of approach. More precisely, \citet{szabo2015two,szabo2016learning,meunier2022distribution} rely on the explicit expressions of $\hat{f}_n$ and $\hat{f}_{n,N}$, which seems attractive but necessitates to study random elements
in Hilbert spaces, with the RKHS norm $\HH_K$. Instead, we rely on studying the ridge regression empirical risk, and exploiting convexity, which enables us to study real-valued variables, but still obtaining conclusions on $\| \hat{f}_{n,N} - \hat{f}_n \|_{\HH_K}$.
Remark \ref{remark:comment:regarding:existing} explains in more details these innovations of our analysis compared to  \citet{szabo2015two,szabo2016learning,meunier2022distribution}.

 Then, still for inputs in a general Hilbert space, we show that asymptotic bounds on $\| \hat{f}_{n,N} - \hat{f}_n \|_{\HH_K}$ imply asymptotic bounds on $\| \hat{f}_{n,N} - \hat{f}_n \|_{\cE,\infty}$ of a strictly better order, where $\| \cdot \|_{\cE,\infty}$ is the supremum norm and when the functions  $\hat{f}_{n,N}$ and  $\hat{f}_n$ are restricted to a compact set $\cE$.  
Also, we then combine the bounds on $ \hat{f}_{n,N} - \hat{f}_n $ with the bounds provided by \citet{caponnetto2007optimal} on $ \hat{f}_n - f^{\star}$, to bound $ \hat{f}_{n,N} - f^{\star}$.

Second, we focus back on the distribution regression setting \eqref{eq:regression:problem}.
We study in turn three specific Hilbertian embeddings discussed in Section \ref{subsection:intro:distribution:regression}: the one based on the Sinkhorn distance, the one based on mean embeddings and the one based on the sliced Wasserstein distance. In the three cases, we prove that the near-unbiased condition indeed holds, making our general results above applicable.
Applying these results provides rates of convergence for the two-stage sampling distribution regression  problem based on the (very) recent Sinkhorn Hilbertian embedding \citep{bachoc2023gaussian}, for which no such rates were previously existing. 
Applying these results to mean embeddings provides strictly improved rates of convergence compared to \citet{szabo2015two,szabo2016learning}, in the sense that as $n \to \infty$, we need a strictly smaller order of magnitude of $N \to \infty$, for the convergence rate on $\hat{f}_{n,N} - f^{\star}$ to reach the minimax rate on $\hat{f}_{n} - f^{\star}$ provided by \citet{caponnetto2007optimal}. Finally, applying our previous general results to Hilbertian embeddings based on the sliced Wasserstein distance yields a similar strict improvement compared to \citet{meunier2022distribution}.

Lastly, we complement our theoretical insights with extensive numerical experiments.
The aim of these experiments is three-fold: illustrating the effects of $n$ and $N$, comparing the mean embeddings in practice, and demonstrating the benefit of two-stage distribution regression, in particular for complex and high-dimensional problems. First, we study the simulated problem of regressing the number of modes of Gaussian mixtures, which enables us in particular to illustrate the effects of $n$ and $N$. Then, we use distribution regression to provide a solution to an {\it ecological inference} problem. This kind of problem is frequent in econometrics, when one aims at predicting the mean behavior for subgroups when only group level data are available. 
Inspired by the seminal work in \cite{flaxman2015supported} and \cite{flaxman2016understanding}, 
we forecast the votes of groups of individuals while only observing their features' distributions. We prove the scalability and the  flexibility of distribution regression to handle this practical use case, characterized by the challenging values 979 for $n$, 2 500 for $N$ (as the average number of samples per $\mu_i$ in this example) and 3 899 for $d$. This numerical study also enables us to compare the Hilbertian embeddings considered, with respect to various statistical and computational criteria. \textcolor{black}{Last, we provide further insight on ecological inference by carrying out a simulation study mimicking it. In particular, we exhibit an empirical  curse of dimensionality and we show that kernel distribution regression enables to recover the true unknown effects of the variables in the data generating process.}

\subsection{Organization of the work}
 This paper falls into the following parts. Section \ref{s:intro} explains the framework of distribution regression and introduces the main notations.
 Section~\ref{section:general:hilbert} provides our results listed above for kernel ridge regression on general Hilbert spaces.
 Section \ref{s:applitheorem} provides our three applications listed above on Hilbertian embeddings for distribution regression. The numerical experiments are exposed in Section \ref{s:expe}. A conclusive discussion is provided in Section \ref{section:conclusion}. 
All the proofs are postponed to the Appendix.

\subsection{\textcolor{black}{Overview of complementary works on distribution regression, functional data analysis and related problems}}
\label{subsection:review:FDA}
\textcolor{black}{ The model \eqref{eq:regression:problem} is also known in the literature as {\it scalar-on-distribution regression model}. Here we review various approaches to tackle this model, and other related ones. These reviewed approaches are complementary to kernel ridge regression, which is discussed above and is the main focus of the paper.}

\textcolor{black}{ \cite{poczos2013distribution} focused on \eqref{eq:regression:problem} and proposed a Nadaraya-Watson type estimator applied to a kernel density estimator (they called it {\it Kernel-Kernel Estimator}). \cite{oliva2014fast} proposed the {\it Double-Basis Estimator}, having less computation complexity when evaluating new predictions
after training and having a faster rate of convergence than the   Kernel-Kernel Estimator. 
Both these references address the two-stage sampling setting described in Section~\ref{subsection:two:stage}. Finally
\cite{petersen2016functional} introduce mappings to Hilbert spaces in order to exploit functional data analysis, discussed next.}
Also, \cite{Talsk.2021} proposed a spline approximation of discretized densities to model sediment grain size as distributions and employed Bayes space methodology  to build a {\it  scalar-on-density regression.}
	\cite{Ghosal.2023.BioStat} proposed a semiparametric regression method, which  offers direct interpretability in terms of the quantile-levels of subject-specific distributions to capture the distributional nature of wearable data. Finally, \cite{Matabuena.2023.JRAM-C} applied kernel smoothing (and kernel ridge regression) to predict scalar responses (such as the age) from distributional accelerometer data.

\textcolor{black}{
Distribution regression is indeed related  to the field  of functional data analysis, on which we refer to \cite{ramsay2007applied,Morris2015,Wang2016} for overviews. The problem most similar to distribution regression is {\it scalar-on-function} regression, with observations of the form
\begin{equation}
\label{eq:regression:problemFuncti}
    Y_i = f^{\star}(\phi_i) + \epsilon_i,
\end{equation}
for $i=1,\ldots,n$,
where $\phi_i$ belongs to a Banach space of functions (typically $\cL^2([0,1])$) and $(Y_i, f^{\star}, \epsilon_i)$ are as in \eqref{eq:regression:problem}. 
For contributions on scalar-on-function regression, we refer in particular to \cite{Cardot1999,Mller2005,crambes2009smoothing,Delaigle2012,ferraty2022scalar,Berrendero2024} and references therein.
 \cite{Morris2015,Wang2016,betancourt2024fungp} provide lists of publicly available software.
Scalar-on-function regression can also be tackled with Bayesian Gaussian process models in applications \citep{morris2012gaussian,muehlenstaedt2017computer,betancourt2020gaussian}.
We note that (generalized) functional linear models are commonly exploited for scalar-on-function regression \citep{Cardot1999,Mller2005,Delaigle2012}, in which case there is a clear interpretability benefit (for instance understanding which parts of a functional covariate support are the most important to predict the scalar response).
This interpretability benefit extends to the distribution regression method in \cite{petersen2016functional}, enabling quantifying the effects of features of distribution predictors.
In contrast, with kernel ridge regression for our distribution regression setting, it is typically less direct to interpret, for instance, which of the $d$ variables of the support $\Omega$ are most important. Nevertheless, note that, in Section~\ref{subsection:ecological:simulated}, we can use the model obtained  from kernel distribution regression to more indirectly get interpretable results, in particular on the effect of the variables.  
}
\textcolor{black}{
The related {\it function-on-scalar regression} model is also tackled in \cite{chiou2004functional,Wang2016,Morris2015}. 
Also, the {\it function-on-function} regression model has been studied in \cite{Cuevas2002,hormann2015note,Manrique2018}, the latter performing a ridge-regularized functional linear regression.}

\textcolor{black}{
In the above discussion, as well as in this paper, the predictands ($(Y_i)_{i=1}^n$ in this paper) are scalar or functions, thus belonging to a linear space.  
The problem of distribution-valued predictands has also been addressed recently, with additional challenges caused by the absence of linearity. 
In particular, \cite{Petersen2019}  consider regression problems with predictands in a general metric space, addressing then specifically the 
{\it distribution-on-scalar} regression model with data $(t_i,\nu_i)_{i=1}^n$,
where now the outputs $(\nu_i)_{i=1}^n$ are univariate probability distributions and the covariates $(t_i)_{i=1}^n$ are real numbers (Section 6 there). Note also that regression of univariate distributions
from vectors has been tackled in \cite{zhou2024wasserstein}.  } 

\textcolor{black}{
Also, \cite{Chen2021,ghodrati2022distribution} consider the {\it distribution-on-distribution} model, i.e., the data are $(\mu_i,\nu_i)_{i=1}^n$
where  both  $(\mu_i)_{i=1}^n$ and $(\nu_i)_{i=1}^n$ are i.i.d. samples of univariate distributions. 
Additionally, we refer to \cite{Okano2024} for the Gaussian case and to \cite{chen2024slicedwassersteinregression} for  exploiting the sliced Wasserstein distance to address multidimensional distributions as predictands.}

\textcolor{black}{
In the above references and settings, there are various counterparts to our two-stage sampling framework (Section~\ref{subsection:two:stage}). 
For functional data analysis, the functions (for instance $(\phi_i)_{i=1}^n$ in \eqref{eq:regression:problemFuncti}) are usually observed only at finite sets of grid points, and can also be affected by observation noise.
We note that it would be interesting to apply our results for regression on general Hilbert spaces (Section~\ref{section:general:hilbert})
to Hilbert spaces of functional covariates. 
This would be possible in cases where the near-unbiased condition can be proved, which could occur for instance with observation noise. 
}

\textcolor{black}{
For distribution-on-scalar and distribution-on-distribution regression, usually only samples from the distributions are observed, similarly as in this paper. We note that the corresponding references above usually take the intermediary step of reconstructing explicitly the distributions from the samples (for instance \cite{Chen2021,chen2024slicedwassersteinregression} mention density and c.d.f. estimation).
In our setting, this intermediary step is arguably less prominent, since just the kernel values between empirical distributions need to be evaluated (see \eqref{eq:hatfnN:practice} below). 
We also mention that, in the literature, two-stage sampling can also refer to underlying clusters of pairs of covariates/predictands, see  \cite{Scott1982} for vector/scalar pairs and \cite{conde2021functionalregressionclusteringmultiple,Wang2016b,Li2021} in functional-data analysis settings.}

\textcolor{black}{
Regarding theoretical results and proof methods, essentially, this paper and the various references discussed above are complementary, with only limited similarities. In these references, the assumptions on the data distributions typically differ from instance to instance, and also differ from our paper and its closely related works \citet{caponnetto2007optimal,szabo2015two,szabo2016learning,meunier2022distribution}. 
Also, our proof techniques address specificities of kernel ridge regression (further discussion can be found in Remark~\ref{remark:comment:regarding:existing}), while, in particular, the functional data analysis literature studies different procedures, typically relying on functional basis projections, see for instance \cite{Mller2005}.
}

\section{Presentation and notations} 
\label{s:intro}

\subsection{Distribution regression} \label{subsection:distribution:regression}

We are interested in a regression problem for which the covariates are distributions on a support (input) space $\Omega$.
There are thus random i.i.d. pairs $(\mu_1,Y_1), \ldots, (\mu_n,Y_n) \in  \mathcal{P} (\Omega) \times \mathbb{R} $,
where we let $ \mathcal{P} (\Omega)$ be the set of probability distributions on $\Omega$.
We write $\cL$ for the common distribution of $\mu_1,\ldots,\mu_n$ and $f^{\star}$ for the conditional expectation function of $Y_i$ given $\mu_i$: for $\mu \in  \mathcal{P} (\Omega)$, $f^{\star}(\mu) = \bE[ Y_i | \mu_i = \mu ]$. 
The function $f^{\star}$ is the target of interest in this paper and the goal is to construct a regression function \linebreak $\hat{f} : \cPO \to \R$ such that for any new pair $(\mu,Y)$, independent of and distributed as $(\mu_i,Y_i)_{i=1}^n$, $\hat{f}(\mu)$ is as close as possible to $f^{\star}(\mu)$, as measured by the squared norm $\int_{\cPO}
\left( 
f^{\star}(\mu)
-
\hat{f}(\mu)
\right)^2
\dd \mathcal{L}(\mu)
$, or by the (stronger) RKHS norm $\| \cdot \|_{\HH_K}$ introduced below.
Note that it is well-known that this squared norm error is also the excess quadratic risk $\bE_n[ (\hat{f}(\mu) - Y)^2 ] - \bE_n[ (f^{\star}(\mu) - Y)^2 ]$, where $E_n$ is the conditional expectation given $(\mu_i,Y_i)_{i=1}^n$.

In the following, we will assume that the support of the distributions satisfies the following mild condition. 

\begin{condition} \label{cond:Omega}
The input space $\Omega$ is compact in $ \mathbb{R}^d $.
\end{condition}

We will endow the covariate space $ \mathcal{P} (\Omega)$ with the Wasserstein distance $\mathcal{W}_1 $ that we define next. 
Note that other distances between distributions could be considered as well; our choice of $\mathcal{W}_1$ follows from the large recent body of literature demonstrating its relevance for theory and practice, see for instance \cite{Arjovsky2017wasserstein,srivastava2018scalable,bernton2019parameter,catalano2021measuring,manole2022minimax,niles2022minimax} among many other works.
For $\mu,\nu \in \cPO$, we let 
\[
\mathcal{W}_1(\mu,\nu)
=
\underset{
\substack{
\pi \in \Pi(\mu,\nu)
} 
}{
\inf
}
\int_{\Omega}
\left\|
x - y
\right\|
\dd \pi(x,y),
\]
where  $\Pi(\mu,\nu)$ is the set of probability measures $\pi$ on $\Omega \times \Omega$ with marginals $\mu$ and $\nu$, that is, 
for all $A,B$ measurable sets $\pi(A \times \Omega) = \mu(A),$ $
	\pi(\Omega  \times B) = \nu(B).$
Then, from \citet[Thm. 6.18 and Rem. 6.19]{Villani2003}, Condition \ref{cond:Omega} implies that $\cPO$ is a compact metric space with the distance $\mathcal{W}_1 $.
It is thus well-behaved as a covariate set. 
We endow $\Omega$ and $\cPO$ with their Borel $\sigma$-algebra, which also defines expectations and integrals, as the squared norm and excess quadratic risk above.

\subsection{Hilbertian embedding for kernel ridge regression} \label{subsection:hilbertian:embedding}

Hilbertian embedding consists in associating to any distribution $\mu \in \cPO$ an element $x_{\mu} \in \HH$, where $\HH$ is a separable Hilbert space.
For specific examples of the space $\HH$ and the mapping $\mu \mapsto x_{\mu}$, we refer to Section \ref{s:applitheorem}.
We write $\langle\cdot, \cdot \rangle_{\HH}$ for the inner product on $\HH$ and $\|\cdot\|_{\HH}$ for the norm.
For a function $f$ operating on $\cPO$ \textcolor{black}{but depending only on the Hilbertian embedding value},
we use the convenient abuse of notation of extending it to
the image set $\{x_{\mu} ;  \mu \in \cPO \}$, that is we write $f(x_{\mu}) = f(\mu)$ for $\mu \in \cPO$.  
We let $x_i = x_{\mu_i}$ for \linebreak $i=1,\ldots,n$. Then the i.i.d. pairs $(x_i,Y_i)_{i=1}^n$ constitute a dataset from which the regression function $f^{\star}$ (seen as operating on $\{x_{\mu} ;  \mu \in \cPO \} \subset \HH$ with the previous notational convention) can be estimated, \textcolor{black}{in the case where it only depends on the Hilbertian embedding value}.

For this estimation, we consider kernel ridge regression with the squared exponential kernel $K: \HH \times \HH \to \R$ defined by, for $u,v \in \HH$, 
\begin{equation} \label{eq:squared:exponential}
K(u,v)
=
F(\| u - v \|_\HH)
=
e^{- \| u - v \|_\HH^2 },
\end{equation}
letting $F(t) = e^{-t^2}$.  As pointed out for instance in \cite{bachoc2020gaussian}, any function of the form $(u,v) \mapsto \tilde{F}(\|u - v\|_{\HH})$, for $\tilde{F}:\mathbb{R}^+ \to  \mathbb{R}$ such that $\tilde{F}(\sqrt{.})$ is  a completely monotone function, is a kernel (a non-negative definite function). 
 Note that our analysis could be extended to general functions $\tilde{F}$ instead of the specific $F(t) = e^{-t^2}$.
 We nevertheless focus on $F$ in this paper, since it is arguably the most popular in learning applications of kernels, and to promote a simplicity of exposition by avoiding additional parameters that are not of primary focus. 
 
 \textcolor{black}{ Then each $x\in \HH$ is associated to a continuous function $K_x=K(x, \cdot):\HH\to \R $ and the space 
$ {\rm span}( \{ K_x: x\in  \HH\} )$ is a vector space. Its closure by the norm $\|\cdot\|_{\HH_K}$ induced by the inner product 
$$ \left\langle \sum_{i=1}^{\ell_1}\alpha_{i} K_{u_i} , \sum_{j=1}^{\ell_2} \beta_j K_{v_j} \right \rangle_{\HH_K}
= \sum_{i=1}^{\ell_1} \sum_{j=1}^{\ell_2} \alpha_{i}  \beta_j K(u_i,v_j),
~ ~ ~ ~
\ell_1,\ell_2 \in \mathbb{N}, \alpha_i , \beta_j \in \mathbb{R},
u_i,v_j \in \HH,
$$
}defines a new Hilbert space, namely \textcolor{black}{the RKHS $\HK$ of the kernel $K$} (see e.g.,  \cite{berlinet2011reproducing}). 
Then the kernel ridge regressor $\hat{f}_n$ is defined as the unique minimizer over $\HK$ of $R_n(f)$, where
\begin{equation} \label{eq:Rn}
    R_n(f)
    =
    \frac{1}{n}
    \sum_{i=1}^n
    \left( 
Y_i - f(x_i)
    \right)^2
    +
    \lambda \| f \|^2_{\HH_K},
\end{equation}
for $ 0 < \lambda < \infty$ a deterministic ridge parameter. \textcolor{black}{The kernel ridge regressor is explicitly given by
\[
\hat{f}_n(x) 
=
r_n(x)^\top (\Sigma_n + n \lambda I_n )^{-1} Y_{[n]},
~ ~ x \in \HH,
    \]
where 
$r_n(x) = (K(x,x_1) , \ldots , K(x,x_n))^\top $,
$Y_{[n]} = (Y_1 , \ldots , Y_n)^\top $ and $\Sigma_n$ is the $n \times n$ matrix with component $i,j$ given by $K(x_i,x_j)$, see, e.g., \cite[Section 2.4.2.2]{berlinet2011reproducing}.
Hence, in practice, it is sufficient to solve a single linear system of size $n$, in order to compute exactly the regressor values for all $x$.  
Note that an alternative more abstract expression of $\hat{f}_n$ is provided in Lemma~\ref{lem:expression:hatf} of the Appendix (see also Remark~\ref{remark:comment:regarding:existing}). 
}

In our asymptotic results in Sections \ref{section:general:hilbert} and \ref{s:applitheorem}, $\lambda$ will not be fixed, and we will consider $n \to \infty$ and $\lambda \to 0$. \cite{caponnetto2007optimal} provide convergence rates for $\| \hat{f}_n - f^{\star}\|_{\HK}$, that we will present in details in Section \ref{subsection:convergence:rates:hilbert}.

\subsection{Two-stage sampling}

The focus of this paper is on the case where the covariate distributions $\mu_1,\ldots,\mu_n$ of the learning set are {\it unobserved} and we only observe samples from them.
That is,
for $ i = 1, \ldots ,n $, we observe random $X_{i,1},\ldots,X_{i,N} \in \Omega$ such that, conditionally to $(\mu_i,Y_i)_{i=1}^n$, the $ nN $ variables $ (X_{i,j}) $ are independent and $ X_{i,j} $ follows the distribution $  \mu_{i} $. 
Hence, $N$ can be interpreted as an observation budget on $\mu_1 , \ldots , \mu_n$.  
For $i=1,\ldots,n$, we write $\mu_i^N = (1/N) \sum_{j=1}^N \delta_{X_{i,j}}$ for the (observed) empirical counterpart to $\mu_i$.
We then let $x_{N,i} = x_{\mu_i^N}$, using the Hilbertian embedding.
Thus $x_{N,i}$ is the observed counterpart to $x_i$.

From the noisy regression dataset $(x_{N,i},Y_i)_{i=1}^n$, we can define $\hat{f}_{n,N}$, as the unique minimizer over $\HK$ of  $R_{n,N}(f)$, defined as 
\begin{equation} \label{eq:Rn:empirical}
    R_{n,N}(f)
    =
    \frac{1}{n}
    \sum_{i=1}^n
    \left( 
Y_i - f(x_{N,i})
    \right)^2
    +
    \lambda \| f \|^2_{\HH_K}.
\end{equation}
\textcolor{black}{Similarly as for $\hat{f}_n$ above, $\hat{f}_{n,N}$ is explicitly given by
\begin{equation}
\label{eq:hatfnN:practice}
\hat{f}_{n,N}(x) 
=
r_{n,N}(x)^\top (\Sigma_{n,N} + n \lambda I_n )^{-1} Y_{[n]},
~ ~ x \in \HH,
\end{equation}
where 
$r_{n,N}(x) = (K(x,x_{N,1}) , \ldots , K(x,x_{N,n}))^\top $,
$Y_{[n]}$ is as above and $\Sigma_{n,N}$ is the $n \times n$ matrix with component $i,j$ given by $K(x_{N,i},x_{N,j})$. Also as above, an alternative more abstract expression is provided in Lemma~\ref{lem:expression:hatf} of the Appendix.}

Next, in Section \ref{section:general:hilbert}, we focus on the Hilbertian covariates $(x_i,x_{N,i})_{i=1}^n$, not exploiting the fact that they stem from distributions $(\mu_i)$ and their samples $(X_{i,j})$ for now.
We provide error bounds on $\hat{f}_n - \hat{f}_{n,N}$, which corresponds to studying the effect of the noise on the regression covariates. 
From these bounds, we deduce bounds on $\hat{f}_{n,N} - f^{\star}$. 
Then, in Section \ref{s:applitheorem}, we come back to the distributions and their samples, applying Section \ref{section:general:hilbert} to various Hilbertian embeddings. 

\section{Improved error bounds for kernel ridge regression on Hilbert spaces} \label{section:general:hilbert}

The content of this section, although motivated by Hilbertian embeddings of distributions (Section \ref{s:intro}) and presented under this setting, actually holds for any separable Hilbert space $\HH$ 
and any i.i.d. triplets $(x_i,x_{N,i},Y_i)_{i=1}^n$, see Remark \ref{rem:given:mui:xi}.

\begin{condition} \label{cond:separable}
The Hilbert space $\HH$ is separable.
\end{condition}

Outside of Remark \ref{rem:given:mui:xi},
we consider that $(x_i,x_{N,i},Y_i)_{i=1}^n$ are obtained by Hilbertian embeddings of distributions, as in Section \ref{s:intro}. 

\subsection{The near-unbiased condition}
\label{subsection:kernel:regression:hilbert:spaces}

The key assumption is the following and will be referred to as the near-unbiased condition. 
\begin{condition} [Near-unbiased condition] \label{condition:near:unbias}
For all $s >0$,
there is a constant \linebreak$0 < c_s < \infty$ such that the following holds.
For $i = 1 , \ldots , n$, there are random $a_{N,i}$ and $b_{N,i}$ such that   
 \[
x_{N,i} - x_i
=
a_{N,i} + b_{N,i}
 \]
and, conditionally to $(\mu_i,Y_i)_{i=1}^n$, the following holds. The $n$ triplets  $(a_{N,i},b_{N,i})_{i=1}^n$ are independent and satisfy
\begin{equation} \label{eq:general:cond:aNi}
\bE_n 
[
\| a_{N,i}
\|_{\HH}^s
]
\leq 
\frac{c_s}{N^{s/2}} 
\end{equation}
and
\begin{equation} \label{eq:general:cond:bNi}
\bE_n 
[
\|
b_{N,i}
\|_{\HH}^s
]
\leq 
\frac{c_s}{N^{s}}.
\end{equation}
Moreover, the random variables $a_{N,i}$  are centered, that is, for any fixed $x \in \HH$,
\begin{equation} \label{eq:general:cond:aNi:unbiased}
\bE_n
\left[
\langle x ,  a_{N,i} \rangle_{\HH}
\right]
= 0.
\end{equation}
Above, $\bE_n$ denotes the conditional expectation given $(\mu_i,Y_i)_{i=1}^n$.
\end{condition}

\begin{remark} \label{rem:given:mui:xi}
 As announced, the content of Section \ref{section:general:hilbert} actually holds for any i.i.d. triplets $(x_i,x_{N,i},Y_i)_{i=1}^n$, not necessarily obtained by Hilbertian embedding of distributions. In this more general setting, the conditioning with respect to   $(\mu_i,Y_i)_{i=1}^n$ should be replaced by a conditioning with respect to $(x_i,Y_i)_{i=1}^n$, in Condition \ref{condition:near:unbias}. More generally, it would also be sufficient to take a conditioning with respect to a $\sigma$-algebra $\cA_n$ such that $(x_i,Y_i)_{i=1}^n$ is $\cA_n$-measurable. Note that, under Hilbertian embedding of distributions, Condition \ref{condition:near:unbias} is stated in this way, with $\cA_n$ the $\sigma$-algebra generated by  $(\mu_i,Y_i)_{i=1}^n$.   
\end{remark}

As shown in Section \ref{s:applitheorem}, the near-unbiased condition holds
for three important examples of Hilbertian embeddings discussed in Section \ref{subsection:intro:distribution:regression}: the Sinkhorn distance, mean embeddings and the sliced Wasserstein distance. 
This condition first entails that the covariate error \linebreak $x_{N,i} - x_i$ is of order $N^{-1/2}$. 
The interpretation is that
for these three examples, \linebreak $x_{N,i} - x_i = x_{\mu_i^N} - x_{\mu_i} $ and we can show that, so to speak, the mapping $\mu \mapsto x_{\mu}$ is ``well-behaved'' enough.
That is, this mapping yields a difference of order $N^{-1/2}$ between a measure and its empirical counterpart with $N$ samples
(similarly as if the mapping simply consisted, say, in taking the expectation of a fixed function). 
Second, the near-unbiased condition entails that the expectation of the error $x_{N,i} - x_i$ is of order $N^{-1}$, thus much smaller than $N^{-1/2}$. Again, the interpretation is that the previous mapping is ``well-behaved''.

Finally, we remark that Condition \ref{condition:near:unbias} could be weakened by requiring \eqref{eq:general:cond:aNi} and \eqref{eq:general:cond:bNi} to hold only for a finite range of values of $s$, while still enabling to show the results provided next. 
Since these inequalities hold for all values of $s$ in the three applications of Section \ref{s:applitheorem}, we do not explicitly weaken Condition \ref{condition:near:unbias}.
Similarly, Condition \ref{condition:near:unbias} and the results provided next could be extended to more general rates of decay in \eqref{eq:general:cond:aNi} and \eqref{eq:general:cond:bNi}, allowing, for instance, for dependence between the samples $(X_{i,j})_{j=1}^N$ leading to $x_{N,i}$.

\subsection{Improved error bounds on $\hat{f}_n - \hat{f}_{n,N}$} \label{section:general:Hilbert:error:bound}

The purpose of Sections \ref{section:general:Hilbert:error:bound} and \ref{subsection:sharper:rates:different:norms} is to bound $\hat{f}_n - \hat{f}_{n,N}$, which corresponds to the negative impact of not observing $x_1,\ldots,x_n$, that is of the covariate noise. 
The following theorem is one of the main results of the paper. 
In this theorem, the statement is given conditionally to $(\mu_i,Y_i)_{i=1}^n$, letting $(x_{N,i})_{i=1}^n$ be the only remaining source of randomness. \textcolor{black}{We write $\bE_n$ to denote the conditional expectation given $(\mu_i,Y_i)_{i=1}^n$. }
This conditional result will yield unconditional ones in the rest of Section \ref{section:general:hilbert}.

\begin{theorem} \label{theorem:error:bound}
Assume that Conditions \ref{cond:separable} and \ref{condition:near:unbias} hold. Let $Y_{\max,n} = \max_{i=1,\ldots,n} |Y_i|$. Let $c_n = \| \hat{f}_n \|_{\HK}$.  Then, there is a constant $c^{(1)}$ (deterministic; not depending on $n$, $N$, $(\mu_i,Y_i)_{i=1}^n$ and $\lambda$) such that
\begin{align*}
\sqrt{   
\bE_n 
\left[ 
\| \hat{f}_n - \hat{f}_{n,N}\|_{\HK}^2
\right]}
   \leq &
\frac{c^{(1)} (Y_{\max,n} + c_n ) }{ \lambda N} 
+
\frac{c^{(1)} (Y_{\max,n} + c_n)}{\lambda \sqrt{n} \sqrt{N}} 
\\ 
& + 
\left(
1
+
\frac{\sqrt{N} }{\sqrt{n} } 
\right)^{-1}
\left(
 \frac{c^{(1)}  (Y_{\max,n} + c_n)}{\lambda n }
+
\frac{c^{(1)}  (Y_{\max,n} + c_n )}{\lambda^2 n \sqrt{N}}
\right).
\end{align*}
\end{theorem}

The bound in Theorem \ref{theorem:error:bound} voluntarily involves four summands, in order to cover all possible regimes of asymptotic growth and decay of $n$, $N$, $\lambda$, $c_n$ and $Y_{\max,n}$.
As discussed in Section \ref{section:introduction}, the proof techniques used by \cite{szabo2015two,szabo2016learning,Fang2020,meunier2022distribution}, although stated for specific examples of Hilbertian embeddings, actually hold in the general context of Section \ref{section:general:Hilbert:error:bound}.
These proof techniques yield the bound of order $( Y_{\max,n} + c_n  ) / \sqrt{N} \lambda$ that we state in Lemma \ref{lemma:existing:bound:general:Hilb} of the Appendix. 
For a large number of regimes of $n$, $N$, $\lambda$, $c_n$ and $Y_{\max,n}$, our new bound in Theorem \ref{theorem:error:bound} improves on this existing one.
In particular, a notable regime of interest is given in the following corollary, which directly follows from Theorem \ref{theorem:error:bound}.
In this corollary, the results given are asymptotic, in the sense that $n,N \to \infty$ and $\lambda \to 0$. 
We will state other asymptotic results in the sequel, in a similar manner.

\begin{corollary} \label{cor:error:bound}
Consider the setting and notation of Theorem \ref{theorem:error:bound}.
Let $n,N \to \infty$ and $\lambda \to 0$.
Assume further that $1 / \lambda = \cO( \sqrt{N} )$, $n = \cO(N)$ and $\bE [ c_n^2 ]$ and $\bE[ Y_{\max,n}^2 ]$ are bounded. Then, we get the following bound,
\[
\sqrt{   
\bE
\left[ 
\| \hat{f}_n - \hat{f}_{n,N}\|_{\HK}^2
\right]} 
=
\cO
\left( 
\frac{1}{\lambda \sqrt{n} \sqrt{N}}
\right).
\]
\end{corollary}
\begin{remark}[Comparison with existing results and proofs] 
\label{remark:comment:regarding:existing}
For the choices of parameters of Corollary \ref{cor:error:bound}, the sharpest existing bound is given by Lemma \ref{lemma:existing:bound:general:Hilb} of the Appendix, discussed above, and is of order $\cO
\left( 
1 / \lambda \sqrt{N}
\right)$. Hence, with Corollary \ref{cor:error:bound}, we provide an improvement of order $\sqrt{n}$. Intuitively, this improvement is permitted by exploiting the independence of $n$ nearly centered variables (from Condition \ref{condition:near:unbias}). 

More details can be obtained by comparing the proofs of Theorem \ref{theorem:error:bound} and Lemma \ref{lemma:existing:bound:general:Hilb}. In the proof of Lemma \ref{lemma:existing:bound:general:Hilb}, averages of random variables are bounded by their averages of norms, see for instance \eqref{eq:existing:approach:bound:A} in the Appendix. In contrast, in the proof of Theorem \ref{theorem:error:bound}, these averages are approximated by averages of centered uncorrelated variables, for which the variance can be bounded. We refer to Section \ref{subsubsection:bounding:B} of the Appendix for details, in particular where $B_{222}$ in \eqref{eq:Bdeuxdeuxdeux} is created and bounded.

Creating these approximations by averages of centered uncorrelated variables is actually challenging in the proof of Theorem \ref{theorem:error:bound}. Thus, this proof strongly differs from that of Lemma \ref{lemma:existing:bound:general:Hilb}.
It contains techniques that may be considered of general interest, for instance the statements, proofs and uses of Lemmas
\ref{lemma:hadamard:diff:Hilbert}
and \ref{lemma:convex:gradient} in the Appendix, and the coupling arguments between \eqref{eq:coupling:un} and \eqref{eq:coupling:deux} there. Note that the proof of Lemma~\ref{lemma:existing:bound:general:Hilb} (corresponding to the existing results) exploits the explicit expressions of $\hat{f}_n$ and $\hat{f}_{n,N}$ (Lemma~\ref{lem:expression:hatf} in the Appendix). In contrasts, surprisingly, in order to prove Theorem \ref{theorem:error:bound}, it turns out that it was necessary to exploit the more abstract definitions of $\hat{f}_n$ and $\hat{f}_{n,N}$ as minimizers of convex functions (see the use of Lemma \ref{lemma:convex:gradient}). 
\end{remark}

\textcolor{black}{Note that the convergence rate in Corollary~\ref{cor:error:bound} does not depend on the dimension $d$ of the measures $\mu_i$ and their samples $X_{i,j}$. This is because the rates in \eqref{eq:general:cond:aNi} and \eqref{eq:general:cond:bNi} in Condition~\ref{condition:near:unbias}, the near-unbiased condition, also do not depend on $d$.}

\subsection{Sharper error bounds with the supremum norm} \label{subsection:sharper:rates:different:norms}

The convergence results of Section~\ref{section:general:Hilbert:error:bound} are given using the norm induced by the RKHS $\HH_K$, namely $\|.\|_{\HK}$.
Here, we investigate the supremum norm on a compact subset of $\HH$.
We define $\cE \subset \HH$ as the probabilistic support 
of the distribution $\cL$ of the covariates $(x_i)_{i=1}^n$ (the points around which every neighborhood has non-zero probability) and we make the following assumption.

\begin{condition} \label{cond:support}
The set $\cE$ is compact.
\end{condition} 
In particular, one can see that Condition \ref{cond:support} holds with $x_i = x_{\mu_i}$ as in Section \ref{subsection:hilbertian:embedding} and when the embedding $\mu \mapsto x_{\mu} $ is continuous, since $\cPO$ is compact. This is the case for the examples treated in Section
\ref{s:applitheorem}, and potentially for many other ones as well.

We let $K_{\cE}$ be the restriction of $K$ to $\cE \times \cE$ and we let $\HH_{\cE,K}$ be the RKHS of $K_{\cE}$. We write $\langle\cdot, \cdot \rangle_{\HH_{\cE,K}}$ for the inner product on $\HH_{\cE,K}$ and $\|\cdot\|_{\HH_{\cE,K}}$ for the norm.
Then from \citet[Thm. 6]{berlinet2011reproducing},
for a function $g \in \HK$, the restriction of $g$ to $\cE$, written $g_{|\cE}$, is in $\HH_{\cE,K}$ and we have $\| g_{|\cE} \|_{\HH_{\cE,K}} \leq \| g \|_{\HK}$. 
It is also well-known \citep[Thm. 17]{berlinet2011reproducing} that a function $g \in \HH_{\cE,K}$ is continuous.
Furthermore, from Cauchy-Schwarz inequality and the reproducing property \citep[Def.~1]{berlinet2011reproducing}, 
the supremum norm of $g$ on $\cE$, $\| g \|_{\cE,\infty}$, is bounded by $ \max_{u \in \cE} \sqrt{K(u,u)} = 1 $ times the RKHS norm $\| g_{|\cE} \|_{\cE,\HK}$. 
Hence, our convergence rates in Theorem \ref{theorem:error:bound} and Corollary \ref{cor:error:bound} measured with the norm $\| \cdot \|_{\HK}$ also hold when measured with the weaker norm $\| \cdot \|_{\cE,\infty}$.

In fact, we will show that, for this weaker norm, these rates can be improved.
More precisely, in Theorem \ref{bound:Cinfty} below, we show that whenever $a_{n,N} \|  \hat{f}_{n,N}-\hat{f}_n\|_{\HEK} $ is bounded in probability, with $a_{n,N} \to \infty$ (under conditions given below), then $a_{n,N}\| \hat{f}_{n,N}-\hat{f}_n\|_{\cE,\infty}$ goes to zero in probability. In words, convergence rates for the RKHS norm yield faster convergence rates for the supremum norm. 
\textcolor{black}{
\begin{theorem}\label{bound:Cinfty}
Let $n,N \to \infty$ and $\lambda \to 0$. Recall $c_n$ and $Y_{\max,n}$ from Theorem \ref{theorem:error:bound}.  
Assume that Conditions \ref{cond:separable}, \ref{condition:near:unbias} and
\ref{cond:support} hold.
Consider a sequence $a_{n,N} \to \infty$ such that 
\[
a_{n,N} = o \left( 
\frac{\min\left(N, \sqrt{n N}
 \right)}{ 1+ \bE c_n + \bE Y_{\max,n} }
\right).
\] 
Then 
$
a_{n,N}\| \hat{f}_{n,N}-\hat{f}_n\|_{\HEK}
=\cO_{\bP} (1)$
implies
$a_{n,N} \| \hat{f}_{n,N}-\hat{f}_n\|_{\cE,\infty}
=o_{\bP}(1).$
\end{theorem}
}
Theorem \ref{bound:Cinfty} directly applies to the bound of Corollary \ref{cor:error:bound} and improves it for the supremum norm.

\begin{corollary}
Consider the setting of Corollary \ref{cor:error:bound} and assume that  Condition \ref{cond:support} holds. Then we have
\[
 \| \hat{f}_{n,N}-\hat{f}_n\|_{\cE,\infty} 
=
o_{\bP}
\left( 
\frac{1}{\lambda \sqrt{n} \sqrt{N}}
\right).
\]
\end{corollary}

\subsection{Reaching the minimax rate for $f^{\star} - \hat{f}_{n,N}$} \label{subsection:convergence:rates:hilbert}

We are now interested in the error $f^{\star} - \hat{f}_{n,N}$, and its decay rate as $n,N \to \infty$.
Similarly as \citet{szabo2015two,szabo2016learning,meunier2022distribution}, we will rely on \cite{caponnetto2007optimal} that provide minimax rates of convergence for $f^{\star} - \hat{f}_{n}$. We will then study the order of magnitude of $N$ that is large enough for $\hat{f}_{n,N} -  \hat{f}_{n}$ to be of the same order as these minimax rates, enabling $\hat{f}_{n,N} - f^{\star}$ to enjoy them as well.
We shall focus on the setting called ``well-specified'' in \cite{szabo2016learning}, under which $f^{\star}$ belongs to $\HK$, \textcolor{black}{ that is $f^{\star}(\mu)$ depends on $\mu$ only through $x_{\mu}$ and, seen as operating on $\HH$, it belongs to $\HK$.}

 We now introduce various quantities enabling to express these minimax rates for $\hat{f}_n$, assuming that Condition \ref{cond:separable} holds throughout.
\textcolor{black}{For $x \in \HH$, we recall 
$K_x = K(x , \cdot) \in \HK$ (Section~\ref{subsection:hilbertian:embedding}).}
Let us also define, using the same notation for convenience,
$K_x : \mathbb{R} \to \HH_K$  as $K_x : t \mapsto t K_x$.
Then, we define $K_x^\star= \langle \cdot, K_x \rangle_{\HK}$, the linear operator from $\HK$ to $\R$ such that, for $f \in \HK$, we get, $f(x) = K_x^\star f.$
(Note that  $K_x^\star$ is the adjoint operator of $K_x$.) 

Let us write $\mathcal{L}$ for the distribution of the random inputs $(x_i)_{i=1}^n$
(since $(x_i)_{i=1}^n$ are obtained by Hilbertian embedding from $(\mu_i)_{i=1}^n$ as in Section \ref{s:intro}, we thus use the convenient abuse of notation of writing $\mathcal{L}$ for the distribution of both $\mu_i$ and $x_i$). 
Then, we define $T : \HH_K \to \HH_K$ as the linear operator 
\[
T = \bE T_{x_1}  =\int_{\HH} T_{x} \dd \mathcal{L}(x),
\]
where for any $x$ in $\HH$, $T_x=K_x K_x^\star$. As shown in \citet[Prop. 1 and Eq. (29)]{caponnetto2007optimal}, $T$ is a positive trace class operator on $\HK$ and, for $f \in \HK$ and $x \in \HH$, we have
\[
(Tf)(x)
=
\int_{\HH}
f( x' )
K( x' , x )
\dd \cL(x').
\]

Then \cite{caponnetto2007optimal} (and subsequently also \cite{szabo2015two,szabo2016learning}) quantify the hardness of the regression task by the following condition.  

\begin{condition} \label{cond:P:b:c:class}
There exist $ b>1$ and $c \in (1,2]$ such that the following holds. 
\begin{enumerate}
    \item There exists $g \in \HH_K$ such that $f^{\star} = T^{\frac{c-1}{2}} g$.
    \item In the spectral decomposition of $T = \sum_{\ell=1}^\infty \lambda_\ell \langle  \cdot , e_{\ell} \rangle_{\HK} e_{\ell}$, where $(e_{\ell})_{\ell=1}^\infty$ is a basis of $\mathrm{Ker}(T)^\perp$, the eigenvalues of $T$ satisfy that $\lambda_{\ell} \ell^b$ is lower and upper bounded as $\ell \to \infty$.  
\end{enumerate}
\end{condition}
Condition \ref{cond:P:b:c:class} corresponds to the class $\mathcal{P}(b,c)$ in \cite{szabo2016learning}. 
Intuitively, the hardness of the regression problem decreases with $b$ and $c$. Indeed, an increased $c$ can be interpreted as a less complex function $f^{\star}$, and an increased $b$ corresponds to a smaller effective dimension of $\HK$, with respect to the distribution $\cL$, see \cite{caponnetto2007optimal}.

Under Condition \ref{cond:P:b:c:class}, the minimax rate for estimating $f^{\star}$ is $n^{-\frac{bc}{2(bc+1)}}$  and is reached by  $\hat{f}_n$ for an appropriate choice of the ridge parameter $\lambda$ \cite[Thm.~1]{caponnetto2007optimal}. Note that \cite{caponnetto2007optimal}, and then also
\cite{szabo2016learning}, write this minimax rate as $n^{-\frac{bc}{bc+1}}$, because they measure the estimation error with a squared norm (as they consider the excess quadratic risk, see Section \ref{subsection:distribution:regression}).  

Hereafter we apply Theorem \ref{theorem:error:bound} to determine a sufficiently large order of magnitude of the number of samples $N$ (as a function of $n$), for $\hat{f}_{n,N}$ to  achieve the same (minimax) convergence rate as $\hat{f}_n$.
\textcolor{black}{
\begin{theorem} \label{theorem:general:minimax}
Let $n,N \to \infty$ and $\lambda \to 0$.
Assume that Conditions \ref{cond:separable}, \ref{condition:near:unbias} and \ref{cond:P:b:c:class} (with the constants $b$ and $c$) and the following hold: 
 \begin{enumerate}
     \item  There exists a constant $Y_{\max}$ such that, almost surely, $|Y_{i}| \leq Y_{\max}$.
     \item 
     $\lambda n^{\frac{b}{bc+1}}$ is lower and upper bounded and $N / n^a$ is lower bounded, where 
     $$ a=\begin{cases}
\max\left(
\frac{b + \frac{bc}{2}}{bc+1}
,
\frac{2b -1}{bc+1}
,
\frac{
4b - bc - 2
}{
bc+1
}
\right)
\le 1
~ ~ ~ & 
    {\rm if}\ b(1 - \frac{c}{2}) \leq \frac{3}{4}\\
\max\left( \frac{b + \frac{bc}{2}}{bc+1} , \frac{2b -\frac{1}{2}}{bc+1} \right)
> 1
~ ~ ~ & 
    {\rm if}\ b(1 - \frac{c}{2}) > \frac{3}{4}
 \end{cases}.$$
 \end{enumerate}
Then, we have
\[
\sqrt{
\int_{\HH}
\left( 
f^{\star}(x)
-
\hat{f}_{n,N}(x)
\right)^2
\dd \mathcal{L}(x)
}
=
\mathcal{O}_{\bP}
\left( 
n^{-\frac{bc}{2(bc+1)}}
\right).
\]
\end{theorem}
}
In the above theorem, the sufficient order of magnitude of $N$ for $\hat{f}_{n,N}$ to achieve the minimax convergence rate
is $n^a$.
When the theorem is applied to mean embeddings (Section~\ref{subsection:application:mean:embedding})
and the sliced Wasserstein distance
(Section \ref{subsection:application:sliced:wasserstein}), this order $n^a$ is strictly smaller than the orders previously provided by the state-of-the-art references \cite{szabo2015two,szabo2016learning,meunier2022distribution}.

\begin{remark} \label{rem:inference:on:fn}
\textcolor{black}{A question which goes beyond our results in Section~\ref{section:general:hilbert} is that of obtaining statistical tests or confidence regions on the unknown $\hat{f}_{n}$ or $f^{\star}$, based on the observed $\hat{f}_{n,N}$. 
A simple example would be the test of a null hypothesis where $\hat{f}_{n}$ or $f^{\star}$ is a constant or the zero function.
A potential approach toward this goal would be to derive limiting distribution results on $\hat{f}_{n,N} - \hat{f}_{n}$ or $\hat{f}_{n,N} - f^{\star}$. We leave this challenging problem open for future work.}

\textcolor{black}{
We note that in the related topic of functional data analysis (see Section~\ref{subsection:review:FDA}), these tests or confidence regions do exist, see for instance \cite{cardot2003testing,Mller2005,kong2016classical}.}
\end{remark}

\section{Applications to various Hilbertian embeddings for distribution regression} \label{s:applitheorem}

\subsection{A Sinkhorn Hilbertian embedding over distributions}  \label{subsection:application:sinkhorn} 

The recent reference \citet{bachoc2023gaussian} constructs a new Hilbertian embedding on $\cPO$, based on the Sinkhorn distance. For the sake of completeness, we recall this construction. 
Initially, \citet{bachoc2020gaussian} suggest to express dissimilarities between distributions as dissimilarities between their optimal transport maps.
Then \citet{bachoc2023gaussian} extend this approach, but using the Sinkhorn's dual potentials rather than the transport maps, which yields strong computational benefits. 

First, \cite{bachoc2023gaussian} consider a fixed probability measure $\cU \in \cPO$, called a reference measure. They consider the Sinkhorn's (entropic regularized) optimal transport problem between other distributions and this reference one.
Then, they exploit the dual formulation of this problem, pointed out in \cite{genevay2019phd}, defining, for $\mu \in \cPO$, the optimization problem
\begin{equation}\label{dual_entrop}
	\begin{aligned}
		&   \sup_{h\in \cL^1(\mu),g\in \cL^1(\cU)} ~ ~ ~
  \int_{\Omega}  h(x)
		\dd \mu(x)+ \!\int_{\Omega} g(y) \dd \cU(y)\\ 
		& \hspace{50mm} - \!\epsilon \int_{\Omega \times \Omega} e^{\frac{1}{\epsilon} \left({h(x)+g(y)- \frac{1}{2}\|x-y\|^2}\right)} \dd \mu(x) \dd \cU(y).
	\end{aligned}
\end{equation}
Above, $\epsilon >0$ is a regularization parameter, that is fixed throughout Section \ref{subsection:application:sinkhorn}.  
 Problem~\eqref{dual_entrop} enables  \cite{bachoc2023gaussian} to define $g^{\mu}$ as the value of $g^{\star}$ where $(h^{\star},g^{\star})$ is the unique maximizer in \eqref{dual_entrop} for which 
 $g^{\star}$ is centered with respect to $\cU$. 

Note that in practice, the minimization of \eqref{dual_entrop} is achieved using the Sinkhorn's algorithm and that
 several toolboxes have been developed to compute regularized optimal transport such among others as \cite{flamary2017pot} for \texttt{Python}, \cite{klatt2017package} for \texttt{R}, making all computations feasible. We refer to \cite{peyre2019computational} and references therein for further details.

\citet{bachoc2023gaussian} suggest, among others, the following kernel  $K$ 
 defined by
 $$ \cPO\times \cPO\ni (\mu , \nu)\mapsto K( \mu , \nu )
=
F
\left( 
\|g^{\mu}-g^{\nu}\|_{\cL^2(\mathcal{U})}
\right), $$
 $\mu , \nu \in \cPO$,
 where we recall that $F(t) = e^{-t^2}$. 
Note that $\|g^{\mu}-g^{\nu}\|_{\cL^2(\mathcal{U})}$ is well-defined and finite as pointed out in this reference.
This fits to the general Hilbertian embedding framework of Section \ref{subsection:hilbertian:embedding}, with the Hilbert space $\HH = \cL^2(\mathcal{U})$ and the embedding $x_{\mu} = g^{\mu}$ for $\mu \in \cPO$.
In particular, $\HH$ is separable as assumed in Condition \ref{cond:separable}.

We will show how Theorem \ref{theorem:general:minimax} can be applied to this Sinkhorn Hilbertian embedding.
Recall that we consider i.i.d. (unobserved) distributions $(\mu_i)_{i=1}^n$ defined on $\cPO$ and the corresponding (observed) output variables $(Y_i)_{i=1}^n$. The Hilbertian embeddings of $(\mu_i)_{i=1}^n$ are $(x_i)_{i=1}^n$.
Also, we observe random variables  $(X_{i,j})_{i=1,\ldots,n,j=1,\ldots,N}$ with $(X_{i,j})_{j=1}^N$ distributed as $\mu_i$. We thus let $\mu_i^N = (1/N) \sum_{j=1}^N \delta_{X_{i,j}}$, for $i=1,\ldots,n$, so that we define $x_{N,i}~=~g^{\mu_i^N}$.

We first prove in the following lemma that Condition \ref{condition:near:unbias} is satisfied, using results from \cite{gonzalez2022weak}.
Note that this lemma could be extended by replacing the quadratic cost by more general ones in the exponential in \eqref{dual_entrop} \citep{gonzalez2023weak}. 
However, note that the basic properties of kernels, such as universality, or of the associated Gaussian processes, such as sample continuity,
based on these more general costs, are currently unknown.
Hence, stating explicitly an extended version of this lemma is a prospect for future work.

\begin{lemma} \label{lemma:near:unbias}
Assume that Condition \ref{cond:Omega} holds.
For all $s >0$,
there is a constant $c = c(\Omega,\epsilon,s)$ such that the following holds. 
For $\mu  \in \mathcal{P}(\Omega)$, let $X_1 , \ldots , X_N$ be i.i.d. with distribution $\mu$ and let $\mu^N = (1/N) \sum_{j=1}^N \delta_{X_j}$. 
Then, there are random elements $a_N, b_N \in \mathcal{L}^2(\mathcal{U})$ (functions of $X_1 , \ldots , X_N$) such that
\[
g^{\mu^N}
-
g^{\mu}
=
a_N + b_N,
\]
\textcolor{black}{where $\bE 
\left[
\|a_N\|_{\mathcal{L}^2(\mathcal{U})}^s
\right]
\leq 
\frac{c}{N^{s/2}}, $ $\bE 
\left[
\|b_N\|_{\mathcal{L}^2(\mathcal{U})}^s
\right] \le \frac{c}{N^s}$
and  $\bE
\left[
\langle h ,  a_N \rangle_{\mathcal{L}^2(\mathcal{U})}
\right]
= 0$ for all $h \in \mathcal{L}^2(\mathcal{U})$.}
\end{lemma}

Then, we have the following straightforward corollary of
Lemma \ref{lemma:near:unbias} and
Theorem \ref{theorem:general:minimax}. Recall that, from Section \ref{s:intro},
for a measure $\mu$, and for $f \in \HK$, we conveniently write \linebreak $f( \mu ) = f( x_{\mu} )$, that is we identify functions on $\HH$ (restricted to the image set $\{x_{\mu} ;  \mu \in \cPO \}$) with functions on $\cPO$. 
Recall also that we write $f^{\star} (\mu) = \bE [  Y_i | \mu_i = \mu ]$.
We recall $\hat{f}_n$ and $\hat{f}_{n,N}$, defined in Section \ref{s:intro}. Both functions can thus be seen as regression functions on $\cPO$.

\textcolor{black}{
\begin{corollary}    \label{corollary:sinkhorn:minimax}
Assume that Conditions \ref{cond:Omega} and  \ref{cond:P:b:c:class} (with the constants $b$ and $c$ and $\HH = \cL^2(\mathcal{U})$) hold.
Let $n$, $N$, $\lambda$, $a$ and $Y_{\max}$ be as in Theorem \ref{theorem:general:minimax}.
Then, we have 
\[
\sqrt{
\int_{\cPO}
\left( 
f^{\star}(\mu)
-
\hat{f}_{n,N}(\mu)
\right)^2
\dd \mathcal{L}(\mu)
}
=
\mathcal{O}_{\bP}
\left( 
n^{-\frac{bc}{2(bc+1)}}
\right),
\]
where $\mathcal{L}$ is the distribution of $\mu_i$. 
\end{corollary}
}

Note that Corollary \ref{corollary:sinkhorn:minimax} provides the first consistency result with a  rate of convergence for the  ridge regression based on the (recent) Sinkhorn-based kernel in \cite{bachoc2023gaussian}, with or without noisy observations of the inputs $\mu_1 , \ldots , \mu_n$.

\subsection{Mean embedding} \label{subsection:application:mean:embedding}
We prove that the standard mean embedding studied in \cite{szabo2015two,szabo2016learning} falls into the scope of our results.
We consider a continuous kernel $k$ on $\Omega$ and we let $\HH_k$ be its RKHS.
For $\mu \in \cPO$, let $x_{\mu} \in \HH_k$ be defined by, for $t \in \Omega$, 
\[
x_{\mu}(t)
=
\int_{\Omega}
k(u,t) 
\dd \mu (u).
\]
With this definition of $(x_{\mu})_{\mu \in \cPO}$,
the general setting of Section \ref{subsection:hilbertian:embedding} applies, with $\HH = \HH_k$. In particular, since $k$ is continuous, $\HH$ is separable.   
Then, Condition \ref{condition:near:unbias} 
is simply shown to hold, as here it holds the stronger property that $x_{N,i} - x_i$ is exactly unbiased.

\begin{lemma} \label{lemma:near:unbiased:mean:embedding}
Assume that Condition \ref{cond:Omega} holds.
Then, Condition \ref{condition:near:unbias} holds, with $a_{N,i} = x_{N,i} - x_i$ and 
$b_{N,i} = 0$. 
\end{lemma}
\textcolor{black}{
In the proof of Lemma \ref{lemma:near:unbiased:mean:embedding} in the Appendix, 
note that it is relatively immediate to show \eqref{eq:general:cond:aNi:unbiased} in  Condition \ref{condition:near:unbias} (similarly, it is clear that $\mathbb{E}_n[x_{\mu_{i}^N}(t)] = x_{\mu_i}(t)$ for all $t \in \Omega$, with $\bE_n$ as in that condition). The proof of \eqref{eq:general:cond:aNi} is also quite short, since $x_{\mu_{i}^N} = (1/N)\sum_{j=1}^N 
k(X_{i,j},\cdot)$ is an average of i.i.d. elements, conditionally to $\mu_i$.
}

\textcolor{black}{Hence, Condition~\ref{condition:near:unbias} holds for the mean embedding, so that 
Theorem \ref{theorem:general:minimax} applies.   As in Section \ref{subsection:application:sinkhorn}, 
for a measure $\mu$, and for $f \in \HK$, we write $f( \mu ) = f( x_{\mu} )$ and $f^{\star} (\mu) = \bE [  Y_i | \mu_i~=~\mu ]$. 
}
\textcolor{black}{
\begin{corollary}    \label{corollary:mean:embedding:minimax}
Assume that Conditions \ref{cond:Omega} and  \ref{cond:P:b:c:class} (with the constants $b$ and $c$ and $\HH = \HH_k$) hold.
Let $n$, $N$, $\lambda$, $a$ and $Y_{\max}$ be as in Theorem \ref{theorem:general:minimax}. 
Then, we have 
\[
\sqrt{
\int_{\cPO}
\left( 
f^{\star}(\mu)
-
\hat{f}_{n,N}(\mu)
\right)^2
\dd \mathcal{L}(\mu)
}
=
\mathcal{O}_{\bP}
\left( 
n^{-\frac{bc}{2(bc+1)}}
\right),
\]
where $\mathcal{L}$ is the distribution of $\mu_i$. 
\end{corollary}
}

Similarly as discussed for Theorem \ref{theorem:general:minimax}, $n^{-\frac{bc}{2(bc+1)}}$ is the minimax rate and we show that the order $N^a$ of samples is sufficient for $\hat{f}_{n,N}$ to reach it. 
The state of the art references \citet{szabo2015two,szabo2016learning} show that this minimax rate is reached by $\hat{f}_{n,N}$ when $N$ is of order at least $ n^{ \frac{b(c+1)}{bc+1} } $ (with an additional log factor, \textcolor{black}{which was later removed by \cite{Fang2020}}). 

It can be checked that the number of samples we require, $N^a$, is of strictly smaller order than $ n^{ \frac{b(c+1)}{bc+1} } $, for all values of $b,c$, which constitutes a strong improvement. 
In particular, in \citet{szabo2015two,szabo2016learning}, $N$ always needs to be of order strictly larger than $n$, while when $b(1 - \frac{c}{2}) \leq \frac{3}{4}$, Corollary \ref{corollary:mean:embedding:minimax} allows for $N$ to be of smaller or strictly smaller order than $n$, which constitutes a major  improvement in practice. 
For a typical example where $b=2$, $c=3/2$, we improve the necessary number of samples from $N \gtrsim n^{5/4}$ to $N \gtrsim n^{7/8}$.

\begin{remark} \label{rem:curse:dimension}
    \textcolor{black}{
  In Theorem~\ref{theorem:general:minimax}, the exponents in the rates $n^{-\frac{bc}{2(bc+1)}}$ and $n^a$ typically depend on the ambient dimension $d$ of the support space $\Omega$.
  Indeed, they are expressed from the constants $b$ and $c$ from Condition~\ref{cond:P:b:c:class} that typically depend on $d$. 
A general mathematical quantification of this dependence remains open for future work in kernel distribution regression \citep{szabo2015two,szabo2016learning,meunier2022distribution,Fang2020}. Nevertheless, intuitively, a larger $d$ is expected to yield a slower convergence rate $n^{-\frac{bc}{2(bc+1)}}$ for recovering $f^{\star}$, in particular by decreasing $b$ in Item 2 of Condition~\ref{cond:P:b:c:class} (slower eigenvalue decay of $T$).
It is thus expected that the distribution regression problem suffers from the curse of dimensionality.} 

    \textcolor{black}{Support for this intuition can be obtained by considering the simple case of a linear kernel $k(u,v) = u^\top v $ for the mean embedding. In this case, $x_{\mu}$ is the linear function $u \mapsto u^\top  \int_{\Omega} v \dd \mu (v) $ on $\Omega$. Thus, the input space for regression with the kernel $K$, $\{ x_{\mu} ; \mu \in \cPO  \}$, is included in a linear space of finite dimension $d$. With standard kernel regression on finite-dimensional linear spaces, it is usual that the convergence rate is negatively impacted by the ambient dimension, see for instance Corollaries~2 and 3 in~\citet{li2024towards} for a recent instance.}
     
     \textcolor{black}{
We also refer to Remark~\ref{rem:curse:dimension:two} for the impact of $d$ specifically on the Hilbertian embeddings in the two-stage sampling setting. Note finally 
 that, numerically, in Section \ref{subsection:ecological:simulated}, the performance of distribution regression decreases as the dimension increases.}
\end{remark}

\subsection{Embedding based on the sliced Wasserstein distance} 
\label{subsection:application:sliced:wasserstein}

We now consider the Hilbertian embedding based on the sliced Wasserstein distance \citep{kolouri2018sliced,manole2022minimax,meunier2022distribution}. 
For a real-valued random variable $X$, we write $F_X$ for its c.d.f. For a univariate probability distribution $\mu$, we let $F_{\mu} = F_X$ where $X$ is a random variable with distribution $\mu$. 
If $\mu$ is a distribution on $\mathbb{R}^d$, for a $d \times 1$ column vector $\theta$, we let $\mu_{\theta}$ be the distribution of $\theta^\top X$ where $X$ is a random column vector with distribution $\mu$. For a c.d.f. $G$, we use the usual definition $G^{-1}(t) = \inf \{ x \in \mathbb{R} , G(x) \geq t \}$, for $t \in [0,1]$. 

When $d=1$, we let $\Lambda$ be the Dirac probability measure at $1$ and when $d \geq 2$, we let $\Lambda$ be the uniform distribution on the unit sphere $\cS^{d-1}$ of $\mathbb{R}^d$. As a convention, we let $\cS^0 = \{ 1 \}$.
For $\epsilon \in [0,1/2)$, and for $\mu , \nu \in \cPO$, we define 
\begin{equation} \label{eq:sliced:wasserstein}
\SW(\mu,\nu)^2
=
\frac{1}{1-2\epsilon}
\int_{\cS^{d-1}} 
\int_{\epsilon}^{1- \epsilon}
\left(
F_{\mu_{\theta}}^{-1}(t)
-
F_{\nu_{\theta}}^{-1}(t)
\right)^2
\dd \Lambda (\theta)
\dd t, 
\end{equation}
where $1-2 \epsilon$ at the denominator is used to integrate with respect to a probability measure.
The quantity $\SW(\mu,\nu)$ is the trimmed sliced Wasserstein distance  
when $d \geq 2$ and $\epsilon >0$, see \cite{manole2022minimax}. 
When $d \geq 2$ and $\epsilon =0$, it is  the sliced Wasserstein distance, studied in particular in \cite{kolouri2018sliced,meunier2022distribution}. 
Finally when $d=1$, $\SW(\mu,\nu)$ is the trimmed Wasserstein distance for $\epsilon>0$ and the Wasserstein distance for $\epsilon = 0$. 

It is easily seen (see also Proposition 5 in \cite{meunier2022distribution} for $d \geq 2$ and $\epsilon = 0$) that $\SW(\mu,\nu)$ is a Hilbert norm with the following embedding of distributions. Let $\HH = \cL^2 ( \Lambda \times \cU( [\epsilon , 1-\epsilon] )  )$, where  $\cU( [\epsilon , 1-\epsilon])$ is the uniform distribution on $[\epsilon , 1-\epsilon]$. For $\mu \in \cPO$, define $x_{\mu} \in \HH$ by, for $\theta \in \cS^{d-1}$ and $t \in [\epsilon,1- \epsilon]$,
\[
x_{\mu} ( \theta , t )
=
F_{\mu_{\theta}}^{-1} (t). 
\]

We recall the dataset $(\mu_i)_{i=1}^n$, $(Y_i)_{i=1}^n$ and $(X_{i,j})_{i=1,\ldots,n,j=1,\ldots,N}$
and the true and empirical Hilbertian embeddings  $(x_i,x_{N,i})_{i=1}^n$, similarly as in Sections \ref{subsection:application:sinkhorn} and
\ref{subsection:application:mean:embedding}.
Then, we will show that Condition \ref{condition:near:unbias} 
holds
under the following regularity assumption. 
\textcolor{black}{
\begin{condition} \label{cond:sliced:wasserstein}
There exist constants $c^{(2)}$ 
and $0 \leq \delta \leq \epsilon$, with $\delta < \epsilon$ if
$\epsilon >0$,
such that the following holds almost surely: 
\begin{enumerate}
    \item For every $\theta \in \cS^{d-1}$, there exist $a_i(\theta)$ and $ b_i( \theta )$ with   $- \infty < a_i(\theta) < b_i( \theta ) < \infty$ and such that
$F_{\mu_{i,\theta}}: (a_i(\theta) , b_i( \theta ))\to  (0,1 )$ is bijective.
\item Furthermore,
$F_{\mu_{i,\theta}}^{-1}$ is twice differentiable on $(\delta, 1- \delta)$ with first and second derivatives bounded in absolute value by $c^{(2)}$. 
\end{enumerate}
\end{condition}
}
In Condition \ref{cond:sliced:wasserstein}, $\mu_{i,\theta} = (\mu_i)_{\theta}$ with the above notation, so that as $\mu_i$ is random,  $a_i(\theta)$ and $b_i( \theta )$ are also allowed to be random. 
\textcolor{black}{The following result offers sufficient criteria for ensuring that Condition~\ref{cond:sliced:wasserstein} is satisfied.}
\textcolor{black}{
\begin{lemma} \label{lemma:example:sliced}
	When $d=1$, let $0 \leq \epsilon < 1/2$ and when $d \geq 2$,  let $0 < \epsilon < 1/2$.
	Assume that there are fixed $\tau >0$, $\kappa < \infty$ and $T < \infty$ such that the following holds almost surely: 
 \begin{enumerate}
     \item The support of $\mu_i$ is convex, is contained in $[- \kappa , 
 \kappa]^d$ and contains the Euclidean ball of radius $
	\tau$ centered at $0$.
 \item Furthermore, $\mu_i$ has a density on its support, taking values in $[1/T , T]$, and which is differentiable with gradient bounded by $T$ in Euclidean norm.
 \end{enumerate}
Then Condition \ref{cond:sliced:wasserstein} holds, with any $0 \leq \delta \leq  \epsilon$ ($\delta  < \epsilon$ if $\epsilon >0$) if $d=1$ and with any $0 < \delta < \epsilon$ if $d \geq 2$.
\end{lemma}
}
In Lemma \ref{lemma:example:sliced} in the case $d \geq 2$, the measure $\mu_{i,\theta}$ may have a zero density at the ends of its (convex) support, for some directions $\theta$. Indeed, this density at a point $x \in \mathbb{R}$ is obtained by a $d-1$-dimensional integration over a domain which volume can vanish when $x$ reaches these support ends, because of boundary effects. As a consequence, the inverse c.d.f. $F_{\mu_{i,\theta}}^{-1}$ may not be differentiable at these support ends.
This is why in Lemma \ref{lemma:example:sliced}, we consider $\delta >0$ for $d \geq 2$, and the proof shows in particular that the above volume at $x$ is lower-bounded when $x$ is bounded away from the support ends.
Next, we show that Condition \ref{condition:near:unbias}  holds

\begin{lemma}  \label{lemma:near:unbiased:sliced:wasserstein}
Assume that Conditions \ref{cond:Omega} and \ref{cond:sliced:wasserstein} hold.
Then, Condition \ref{condition:near:unbias} holds. 
\end{lemma}

Condition \ref{cond:sliced:wasserstein} and Lemma \ref{lemma:near:unbiased:sliced:wasserstein} provide natural and convenient-to-express statements and enable to simply apply Theorem \ref{theorem:general:minimax} to (trimmed) sliced Wasserstein kernels next. We think that Lemma \ref{lemma:near:unbiased:sliced:wasserstein} could be extended to milder conditions than Condition \ref{cond:sliced:wasserstein}, where the main challenge would be to study finely the bias of empirical quantiles, beyond the current analysis in the proof of  Lemma \ref{lemma:near:unbiased:sliced:wasserstein}. A related work in this direction is \cite{portnoy2012nearly}. 
With Lemma~\ref{lemma:near:unbiased:sliced:wasserstein}, we obtain the following corollary, similar to Corollaries \ref{corollary:sinkhorn:minimax} and \ref{corollary:mean:embedding:minimax}. 
\textcolor{black}{
\begin{corollary}    \label{corollary:sliced:wasserstein:minimax}
Assume that Conditions \ref{cond:Omega},  \ref{cond:P:b:c:class} (with the constants $b$ and $c$ and $\HH = \cL^2 ( \Lambda \times \cU( [\epsilon , 1-\epsilon] )  )$) and \ref{cond:sliced:wasserstein}  hold.
Let $n$, $N$, $\lambda$, $a$ and $Y_{\max}$ be as in Theorem \ref{theorem:general:minimax}. 
Then, we have 
\begin{equation} \label{eq:rate:sliced}
	\sqrt{
		\int_{\cPO}
		\left( 
		f^{\star}(\mu)
		-
		\hat{f}_{n,N}(\mu)
		\right)^2
		\dd \mathcal{L}(\mu)
	}
	=
	\mathcal{O}_{\bP}
	\left( 
	n^{-\frac{bc}{2(bc+1)}}
	\right),
	\end{equation}
where $\mathcal{L}$ is the distribution of $\mu_i$.	
\end{corollary}
}
Similarly as in Section \ref{subsection:application:mean:embedding}, Corollary \ref{corollary:sliced:wasserstein:minimax} significantly improves the state of the art provided in 
\cite{meunier2022distribution}
with respect to the number of samples $N$ required. 
Indeed, in \cite{meunier2022distribution}, Corollary 9 and the discussion after provide a convergence rate of the left-hand side of \eqref{eq:rate:sliced} of order $n^{-1/4} + N^{-1/8}$ in the setting where the order of magnitude of a quantity $\cN(\lambda)$ there is $1 / \lambda$ as $\lambda \to 0$. 
Thus from \cite{meunier2022distribution}, the rate $n^{-1/4}$ is reached for $N$ of order at least $n^2$. 

This quantity $\cN(\lambda)$ is the effective dimension in \citet{caponnetto2007optimal} and
can also be found in the proof of Theorem \ref{theorem:general:minimax} where it is of order $\lambda^{-1/b}$. 
Hence Theorem~\ref{theorem:general:minimax} (and Corollary \ref{corollary:sliced:wasserstein:minimax}) with $b$ arbitrarily close to $1$ can be compared with this discussion in \cite{meunier2022distribution}. In this theorem, the case where $c$ is also arbitrarily close to $1$, which corresponds to the mildest assumption, provides the same convergence rate $n^{-1/4}$. This rate is reached with $N$ of order $n^{3/4}$. Hence,  Theorem \ref{theorem:general:minimax} drastically reduces the number of samples necessary to reach the minimax convergence rate, going from the order $n^2$ in \cite{meunier2022distribution} to the order $n^{3/4}$. In this context, we nevertheless acknowledge out that we assume stronger regularity conditions on the covariate measures $\mu_1 , \ldots , \mu_n$ in Condition \ref{cond:sliced:wasserstein}, which  are not necessary in~\cite{meunier2022distribution}.

\begin{remark} \label{rem:curse:dimension:two}
\textcolor{black}{
While Remark~\ref{rem:curse:dimension}
discusses a curse of dimensionality on the convergence rate $n^{-\frac{bc}{2(bc+1)}}$ in Theorem~\ref{theorem:general:minimax}, we recall here that the rate $1/\lambda n^{1/2} N^{1/2}$ in Corollary \ref{cor:error:bound} is dimension-free, it does not depend on the dimension $d$ of the support $\Omega$. This is because the rates in Condition~\ref{condition:near:unbias}, the near-unbiased condition, do not depend on $d$.
Since the Sinkhorn, mean and sliced-Wasserstein-based embeddings satisfy this condition, one can say that they are not impacted by a curse of dimensionality.  
} 

\textcolor{black}{
It is nevertheless interesting to point out that the constant $c_s$ in Condition~\ref{condition:near:unbias} can increase with $d$, which means that $d$ can still have a negative impact on the Hilbertian embeddings.
With our proof techniques, the strongest dependence on $d$ for $c_s$ occurs for the Sinkhorn Hilbertian embedding, where in particular a constant from \citet[Eq. 4.13]{Barrio2022AnIC} is used (see Section~\ref{supplement:subsection:proof:sinkhorn} in the Appendix) that increases exponentially with $d$ (although it is possible that a refinement of the proofs in \cite{Barrio2022AnIC} could yield a milder dependence on $d$, in line with \cite{rigollet2024sample}). 
For the sliced Wasserstein embedding, the constant $c^{(2)}$ in Condition~\ref{cond:sliced:wasserstein} can be seen to involve suprema over $\cS^{d-1}$ and thus it can be negatively impacted by $d$, and so does $c_s$ consequently. Finally, for the mean embedding, we note that the impact of $d$ is typically moderate. Inspection of the proof of Lemma~\ref{lemma:near:unbiased:mean:embedding} (Section~\ref{supplement:subsection:proof:unbiased:mean:embedding} in the Appendix) reveals that $c_s$ depends on $\sup_{u \in \Omega} \sqrt{k(u,u)}$ which in particular is equal to $1$ for $k(u,v) = e^{-\| u-v\|^2}$ or is equal to $\sqrt{d}$ for $k(u,v) = u^\top v$ and $\Omega = [0,1]^d$. 
}    
\end{remark}

\section{Numerical experiments} \label{s:expe}

In numerical experiments, with simulations and a data example, we study kernel ridge regression based on the three Hilbertian embeddings considered in Section \ref{s:applitheorem}, in conjunction with the squared exponential kernel studied in this paper\footnote{\textcolor{black}{The 
\texttt{Python} code to reproduce the experiments of Sections \ref{sec:convergencespeedxp} and \ref{subsection:ecological:inference} is publicly available at \url{https://github.com/Algue-Rythme/DistributionRegressionUS2016}.
The \texttt{R} code to reproduce the experiments of Section \ref{subsection:ecological:simulated}  is publicly available at \url{https://github.com/francoisbachoc/kernel_distribution_regression}.}}.

\subsection{Settings for the Hilbertian embeddings \textcolor{black}{in Sections \ref{sec:convergencespeedxp} and \ref{subsection:ecological:inference}}} \label{subsection:expe:the:methods}

Typically, the Hilbertian embeddings $\mu \mapsto x_{\mu}$ considered theoretically in this paper are valued in infinite-dimensional Hilbert spaces. On the other hand, the numerical implementations of these embeddings map distributions to vectors. We shall refer to the dimensions of these vectors as the embeddings' dimension. 

For the embedding based on the Sinkhorn distance (Section \ref{subsection:application:sinkhorn}), 
we rely on the original implementation of~\cite{bachoc2023gaussian}, written with the \texttt{ott-jax} toolbox~\citep{cuturi2022optimal}. We parameterize the reference distribution $\cU$ as a discrete point cloud with equal probabilities along the points.
The embedding dimension is thus simply the number of points.
These points are randomly sampled.
In \eqref{dual_entrop}, we set $\epsilon$ to $10^{-1}$ in Section \ref{sec:convergencespeedxp} and to $10^{-2}$ in Section \ref{subsection:ecological:inference}.
Then,
computing the embeddings $x_{\mu} = g^{\mu}$ for $\mu \in \cPO$ can be done efficiently in parallel on GPUs. The algorithm is illustrated in Figure~\ref{fig:sinkhornmu}.   

For the mean embeddings (Section \ref{subsection:application:mean:embedding}),
we first consider the simple linear embedding $x_{\mu}(t) = \int_{\Omega} u^\top t \dd \mu(u) $ based on $k$ being the linear kernel. The embedding dimension is $d$ in this case.
Then, we also consider the embedding based on random Fourier features \citep{rahimi2007random}, where the embedding dimension is the number of Fourier features.
In both cases, the implementation is straightforward with \texttt{Numpy}.

Consider finally the embedding based on the sliced Wasserstein distance (Section \ref{subsection:application:sliced:wasserstein}).
For a dataset $(X_{i,j} , Y_i)_{i=1,\ldots,n,j=1,\ldots,N}$,
standard implementations of kernel methods for this embedding 
involve pairwise computations of one-dimensional optimal transport problems, with random directions. For instance, this is the case for the \texttt{Python} Optimal Transport (\texttt{POT}) toolbox~\citep{flamary2017pot}.

Instead, we provide a \texttt{Numpy} implementation where
we compute separately the embeddings $x_{\mu_i^N}$, with the definition $x_{\mu} (\theta,t) =  F^{-1}_{\mu_{\theta}}(t)$ from Section \ref{subsection:application:sliced:wasserstein},
with $(\theta,t) \in \mathcal{S}^{d-1}\times [0,1]$,
see  also
\citet[Prop. 5]{meunier2022distribution}. 
The numerically implemented embeddings are the values of $F^{-1}_{\mu_{\theta}}(t)$ on a discretization of $\mathcal{S}^{d-1}\times [0,1]$. The embedding dimension is thus the size of the discretization.
%
Once the embeddings are computed (with a cost linear in $n$), we compute the $n \times n$ covariance matrix of the kernel values at $(\mu_i^N)_{i=1}^n$. 
In Figure \ref{fig:slicedsanitycheck}, we check numerically the validity of our implementation, by comparing it with the numerical results from \texttt{POT}, for a toy example in dimension $d=2$. 

\begin{figure}
\centering
\includegraphics[width=\linewidth]{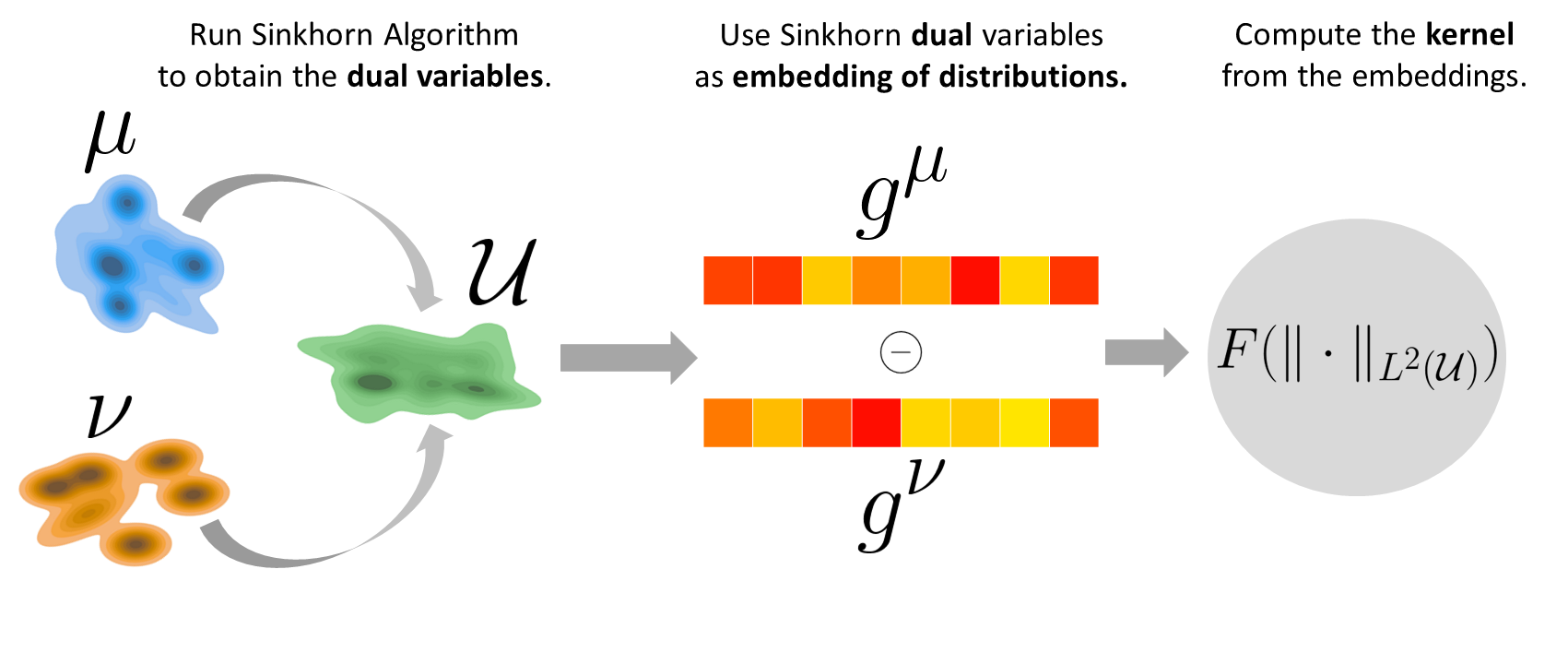}
\caption{Illustration of the embedding based on the Sinkhorn distance (Section \ref{subsection:application:sinkhorn}). Two distributions $\mu$ and $\nu$ are embedded as $g^\mu$ and $g^\nu$, on which the kernel $F( \| g^\mu - g^\nu \|_{\cL^2(\cU)})$ is computed.}
\label{fig:sinkhornmu}
\end{figure}

\begin{figure}
    \centering
    \includegraphics[width=0.5\linewidth]{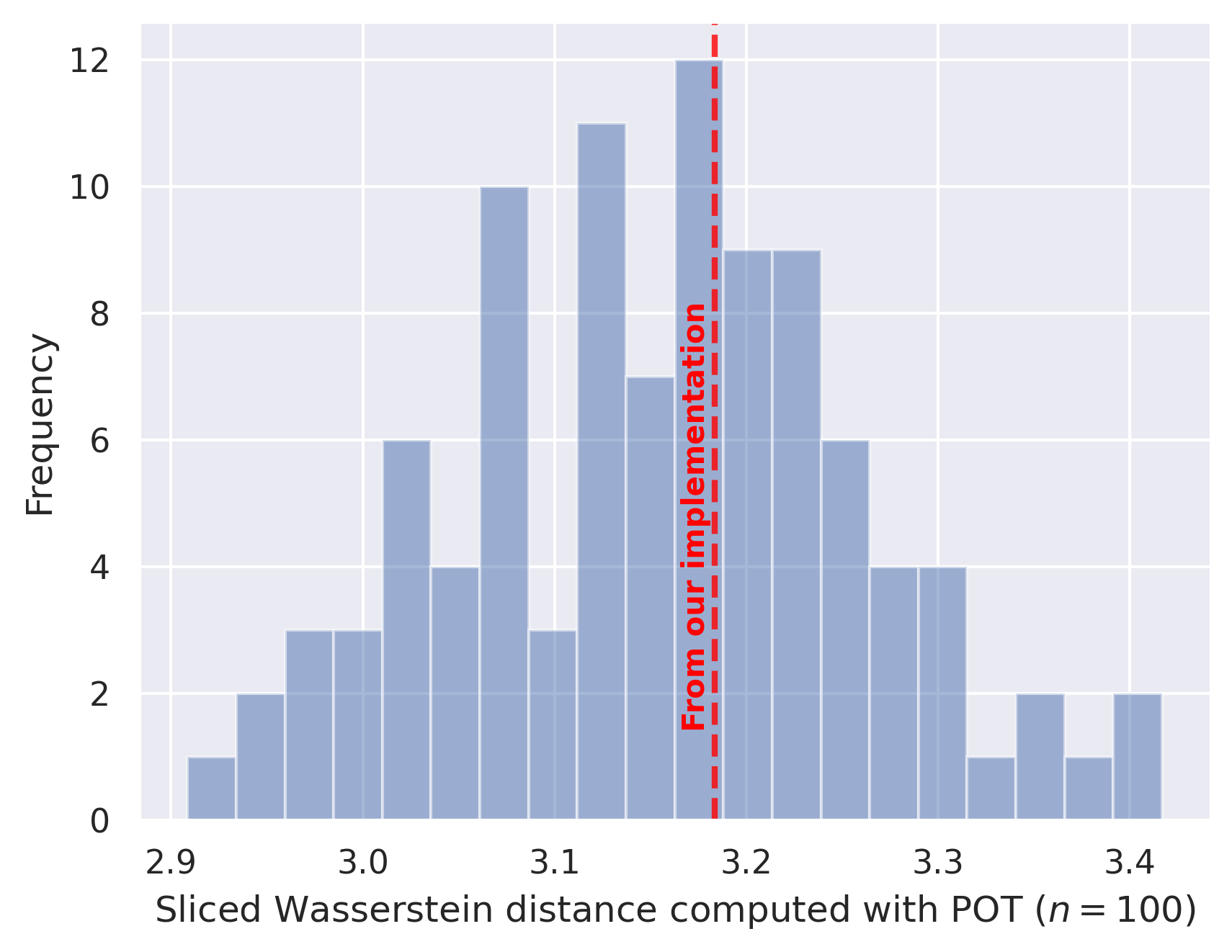}
    \caption{\textbf{Comparison of implementations for the sliced Wasserstein distance.} Stochastic results from the \texttt{POT} toolbox (in blue), compared to the deterministic result from the Hilbertian mapping of our implementation (in red), for two samples of size $500$ from Gaussian distributions ${\scriptsize \mu=\mathcal{N}\left([0, 0], \left[\begin{array}{cc}
                        1 & 0 \\
                        0 & 1 \\
                        \end{array}
                        \right]\right)}$ and ${\scriptsize \nu = \mathcal{N}\left([4, 2], \left[\begin{array}{cc}
                        2 & -0.8 \\
                        -0.8 & 1 \\
                        \end{array}
                        \right]\right)}$. The results from \texttt{POT} are stochastic because of the random directions. For our deterministic result, we use a discretization of the half-circle with $25$ directions of the form $(\frac{k}{25}-\frac{1}{2})\pi$ with $k=1,\ldots,25$, and a discretization of $[0,1]$ with $100$ equidistant points. In both cases, we compute a finite-dimensional version of $\SW(\mu,\nu)$ in \eqref{eq:sliced:wasserstein}.}
    \label{fig:slicedsanitycheck}
\end{figure}

 \subsection{Estimating the number of modes of Gaussian mixtures} \label{sec:convergencespeedxp}

We illustrate the impact of $n$ and $N$ numerically, on the problem of regressing the number of modes of Gaussian mixtures. This use case was introduced by~\cite{oliva2014fast}, and we consider the settings of \cite{meunier2022distribution}. The random $(\mu_i)_{i=1}^n$ are generated as follows. 
In dimension $d$, the number of modes $p$ is uniformly sampled in $\{1,\ldots, C\}$, where $C \in \mathbb{N}$ is a setting parameter.
Then for each component $b \in \{1 , \ldots,p\}$ of the mixture, the mean vector is sampled as $m_b~\sim~\mathcal{U}([-5,5]^d)$, and its associated covariance matrix is sampled as $\Sigma_b=a_{b} A_{b} A_{b}^\top +B_{b}$, where $a_{b} \sim \mathcal{U}([1, 4])$, $A_{b}$ is a $d \times d$ matrix with entries sampled independently from $\mathcal{U}([-1, 1])$ and $B_{b}$ is a diagonal matrix with entries sampled independently from $\mathcal{U}([0, 1])$. 
Therefore we set $\mu_i = \frac{1}{p}\sum_{b=1}^p\mathcal{N}(m_b,\Sigma_b)$ and $Y_i =  p$ to define the $i$-th element of the dataset.  
We sample $N$ points from each mixture $\mu_i$.
We illustrate the resulting dataset in Figure~\ref{fig:gmmexamples}. 

We test the three methods of Section \ref{subsection:expe:the:methods} on each different combination of values for $C~\in~\{2,10\}$, $d\in\{2,10\}$ and $\{ n, N \} \subset \{16,32,\ldots, 1024, 2048\}$.
For the mean embedding, we only consider the linear kernel (not the one based on random Fourier features).
For the sliced Wasserstein embedding, we use a discretization of $\mathcal{S}^{d-1}$ with 10 random directions, and a discretization of $[0,1]$ with 10 equispaced points.
For the embedding based on the Sinkhorn distance, we define the reference distribution $\cU$ by sampling 100 points uniformly in the unit ball.
In the two latter cases, the embedding dimension is 100.

We split each dataset into a train set containing 50\% of the mixtures, and we evaluate the explained variance score on the test set composed of the remaining 50\% mixtures.
Recall that the explained variance is one minus the ratio of the empirical variance of the errors $\hat{Y}_i - Y_i$ on the test set, divided by the empirical variance of the data $Y_i$ on the same test set.
\textcolor{black}{The regularization parameter $\lambda$ (see \eqref{eq:Rn:empirical}) is selected in $\{10^{-2},10^{-1},1,10,10^2\}$ with cross validation on the train set. Furthermore, each ``experiment'' (that is each quadruplet $(C,d,n,N)$, and there are 256 of them) is repeated 20 times (1 time to select the best $\lambda$ and then 19 times with the selected $\lambda$), and the results are averaged, which adds up to  18~432 kernel ridge regressions in total, and 64GB of raw data}. The averaged explained variance score as function of $(n,N)$ is plotted in Figure~\ref{fig:gmmresults}. 

\begin{figure}
    \centering
    \includegraphics[width=1\linewidth]{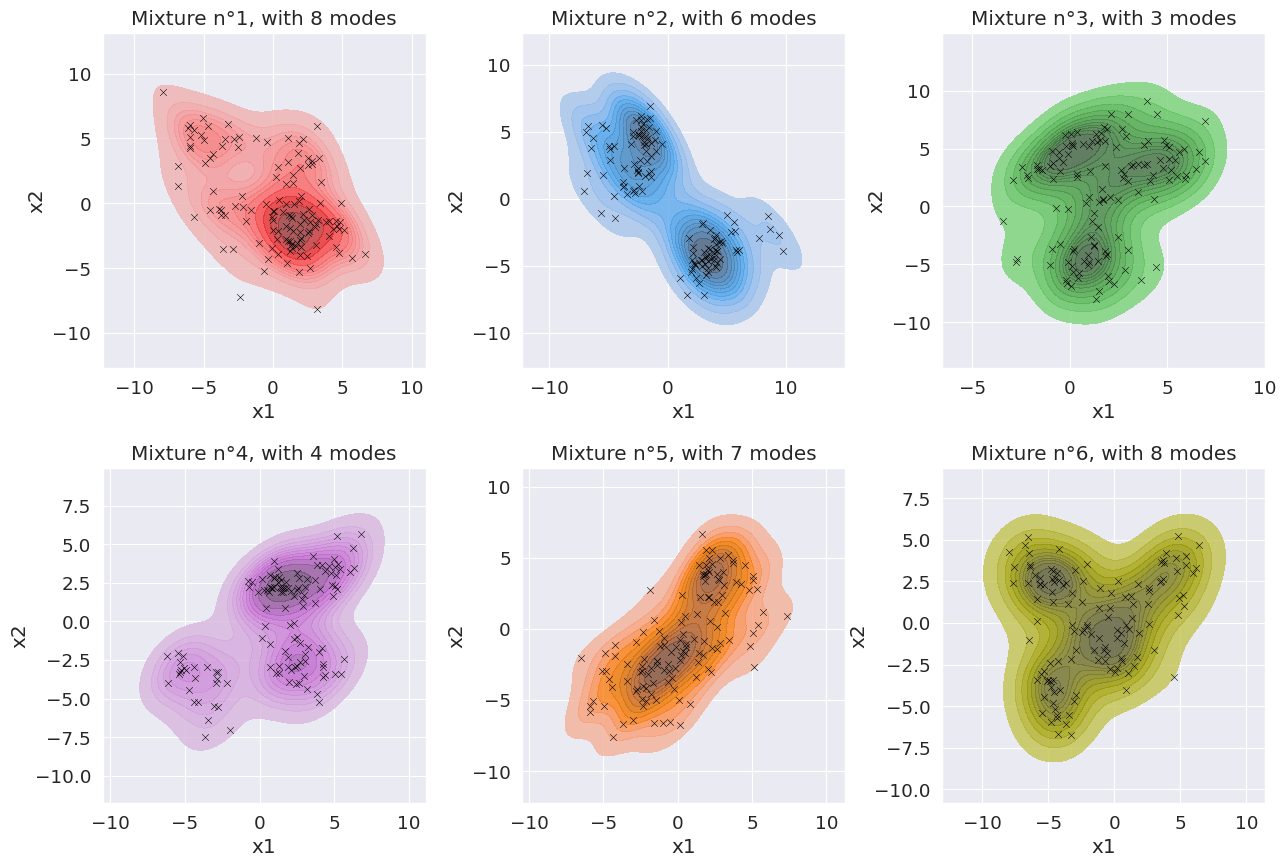}
    \caption{
    Examples of Gaussian mixture models used in the experiment of Section~\ref{sec:convergencespeedxp}, in dimension $d=2$ with at most $C=10$ components per mixture.}
    \label{fig:gmmexamples}
\end{figure}

 \begin{figure}
    \centering
    \begin{minipage}[]{1\textwidth}
        \includegraphics[width=\textwidth]{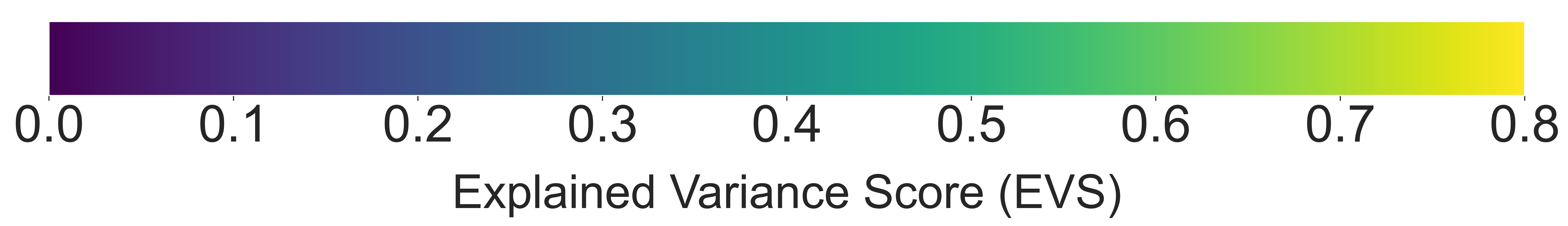}
    \end{minipage}

    \begin{minipage}[]{0.03\textwidth} 
        \resizebox{1\linewidth}{!}{\rotatebox{90}{\textbf{Number of Mixtures (n)}}}
    \end{minipage}
    \hfil
    \begin{minipage}[]{0.30\textwidth} 
        \includegraphics[width=\linewidth]{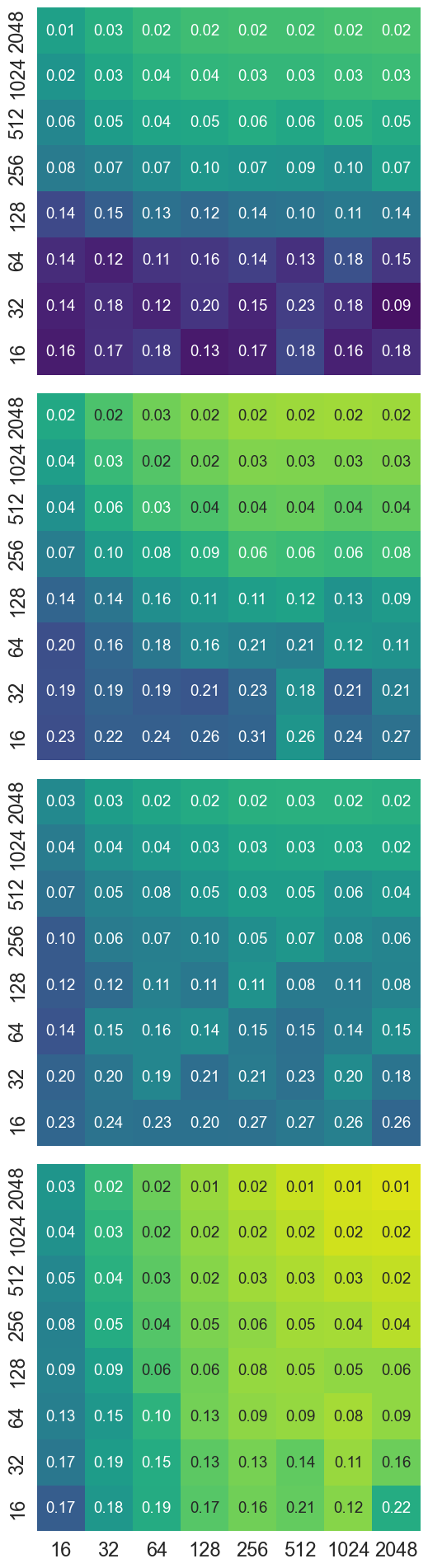}
        \subcaption{\textbf{Sinkhorn-$\cU$}.}
    \end{minipage}
    \hfil
    \begin{minipage}[]{0.28\textwidth}
        \includegraphics[width=\linewidth]{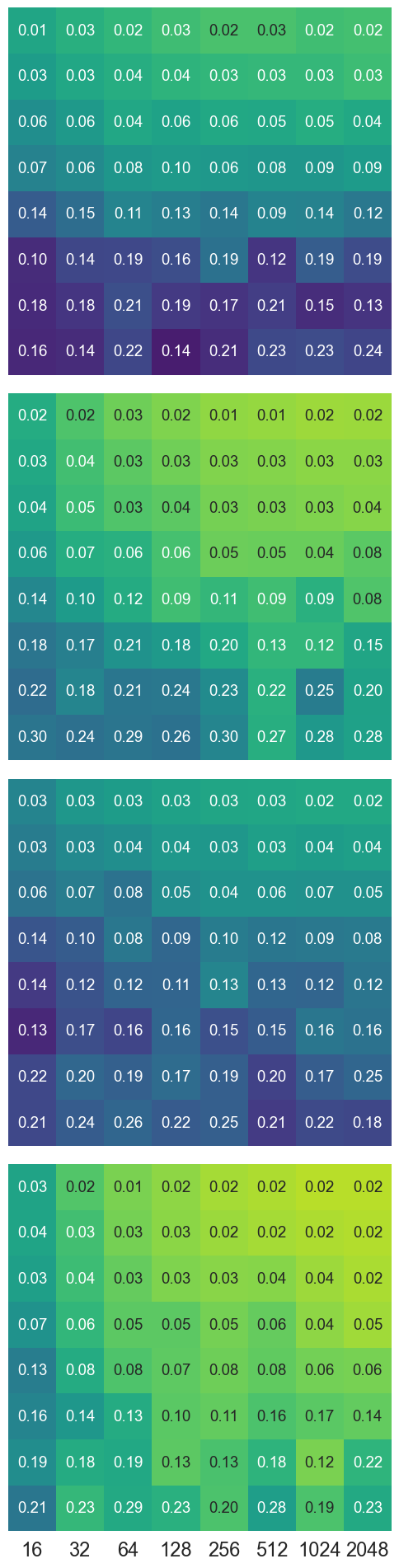}
        \subcaption{\textbf{Sliced Wasserstein}.}
    \end{minipage}
    \hfil
    \begin{minipage}[]{0.28\textwidth}
        \includegraphics[width=\linewidth]{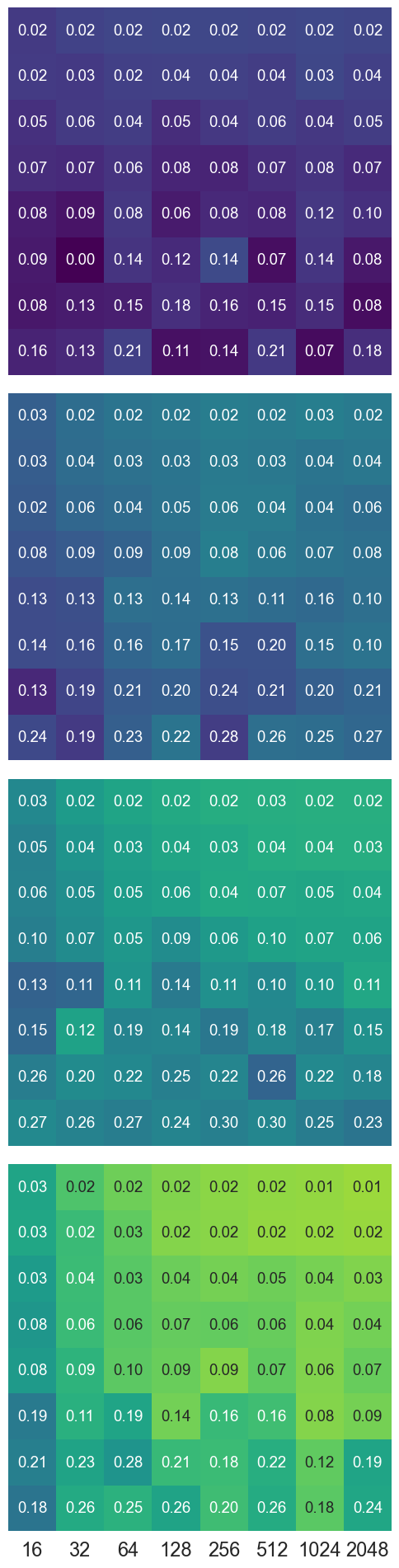}
        \subcaption{\textbf{Mean embedding}.}
    \end{minipage}
    \begin{minipage}[]{0.03\textwidth} 
        \vspace{-0.63cm} 
        \resizebox{1.15\linewidth}{!}{\rotatebox{90}{$d=10, C=10$ \quad $d=10, C=2$ \quad $d=2, C=10$ \quad $d=2, C=2$}}
    \end{minipage}

    \begin{minipage}[]{1\textwidth}
        \centering
        \resizebox{0.3\linewidth}{!}{\textbf{Points per Mixture (N)}}
    \end{minipage}

    \caption{Explained variance score for different embeddings of distributions, from the synthetic mode experiment described in section~\ref{sec:convergencespeedxp}, as a function of the total number of distributions $n$, and the number of samples $N$ per mixture. We plot the mean value of the explained variance score using the color, and the standard deviation inside the cell, computed over 20 independent runs. The dimension of the ambient space is denoted by $d$, and the maximum number of modes in the task is denoted by $C$. \textcolor{black}{For each set of parameters, all methods are benchmarked over the same datasets and splits.}}
    \label{fig:gmmresults}
\end{figure}

 The explained variance score increases with $n,N$, which illustrates Theorem \ref{theorem:general:minimax}, and its three applications, Corollaries \ref{corollary:sinkhorn:minimax},
\ref{corollary:mean:embedding:minimax}
and
\ref{corollary:sliced:wasserstein:minimax}. Furthermore, we see that, overall, increasing $n$ has a higher importance than increasing $N$ for improving the explained variance score. 
Also, there is a visual elbow effect, where, when $N$ is small, increasing it yields a strong improvement of the explained variance score. In contrasts, when $N$ is larger, increasing it further has a more limited impact, which is for instance particularly clear on the bottom-left panel of Figure \ref{fig:gmmresults}. This is in agreement with Theorem \ref{theorem:general:minimax}, where there is the threshold order $n^a$ for $N$ and increasing the order of magnitude of $N$ above this threshold does not improve the estimation error $f^{\star} - \hat{f}_{n,N}$. 

Overall, the three kernel regression methods have similar performances, with the exception of the mean embedding one (based on the linear kernel $k$) in the case $d=2$, that is significantly less accurate. 
Our interpretation is that in ambient dimension $d=2$, representing Gaussian mixtures by their two-dimensional mean vectors is too restrictive, and much more so than in ambient dimension $d=10$.

\subsection{A data example on ecological inference}
\label{subsection:ecological:inference}

We showcase an application of distribution regression to ecological inference, inspired by the seminal work of~\cite{flaxman2015supported}. 
We use 2015 US census data, covering 2 490 616 individuals $X_{i,j}$ ($0.75\%$ of the 2015 US population), and totaling $d = 3 899$ features each (with one-hot encoding of categorical ones), covering characteristics like gender, age, race, occupation, schooling degree or personal income. This yields a fine-grained dataset of US demographics over $n=979$ regions $\mu_i$, spanning the 50 American states (20 regions per state on average, and $N=2 500$ individuals ${X_{i,j} \sim \mu_i}$ per region on average).

We consider three targets $Y_i\in [0, 1]$ from the results of the 2016 presidential election: percentages of Republican vote, Democrat vote, and other vote. 
We perform distribution regression by adapting the \texttt{pummeler} package of~\cite{flaxman2015supported,flaxman2016understanding} to compute the Hilbertian embeddings described in Section \ref{subsection:expe:the:methods}.  
For the Sinkhorn distance, we consider the support sizes 16, 32 and 64 for the reference distribution $\cU$. 
For the generation of the points of $\cU$, the numerical variables are sampled from the standard normal distribution, while the categorical variables are sampled from a discrete distribution.
\textcolor{black}{The regularization parameter $\epsilon$ was selected by sweeping over negative powers of ten. For $\epsilon=10^{-3}$, the solver failed to converge in \texttt{float32}-arithmetic within 2 000 iterations. For $\epsilon=10^{-1}$, the excessive regularization caused features to be too similar, which degraded the performance. The value $\epsilon=10^{-2}$ was selected as the best tradeoff.}

For the mean embedding, we consider the linear kernel $k$, for an embedding in dimension 3 899, and the embedding based on random Fourier features in dimension 4 096. 
For the sliced Wasserstein distance, we study the values 1 024 and 4 096 for the embedding dimensions (the number of discretization points in Section \ref{subsection:expe:the:methods}).
We find that directly regressing the probabilities $Y_i\in [0, 1]$ yields consistently better results than 
regressing their logarithms. Therefore we only report results involving the direct regression of these probabilities.  We also
standardize the features to improve the numerical stability of the computations. Finally, we enforce a default regularization parameter $\lambda=10^{-3}$. 
Indeed, in preliminary (unreported) experiments with the alternative values $10^{-4},10^{-2},10^{-1}$, it appeared that this value was the most suitable. Note that, in general with kernel ridge regression, an overly small $\lambda$ would cause overfitting, while an overly large one would cause the prediction function to be too regularized, deteriorating its accuracy.

In Table~\ref{tab:ecologicalregression},
we report the mean accuracies of the methods, averaged over 5 random train/set splits of sizes 80\% (783 regions) / 20\% (196 regions) respectively, together with the empirical variance with respect to the random seed. 
For interpretation purposes, we also report the results achieved by the constant baseline prediction given by the empirical mean.
We also report the runtime required to compute the embeddings from the raw US census data, and the runtime required to perform kernel ridge regression, given the embeddings.

Table~\ref{tab:ecologicalregression} highlights global properties, and also specific benefits and drawbacks of each method.
Overall, the accuracy and computation time increases with the embedding dimension (except for instance for the accuracy of Sinkhorn from dimension 32 to 64).
The mean embedding with the linear kernel yields the fastest embedding computation but also the lowest prediction accuracy. Hence, despite the high ambient dimension (3 899), a linear embedding is too restrictive. In contrast, the (non-linear) mean embedding with the random Fourier features yields the highest accuracy.
The sliced Wasserstein embedding in dimension 1 024 is the fastest to compute (setting aside the linear mean embedding) and provides accuracies relatively close to the optimal one (with random Fourier features), for a significantly smaller embedding dimension (1 024 against 4 096). 
This is beneficial for dataset compression purposes. 
Hence, overall, the sliced Wasserstein embeddings provide an interesting tradeoff between runtime and final performance.

Finally, the Sinkhorn embeddings provide accuracies that are below those from the sliced Wasserstein ones and the mean embedding ones with Fourier features. On the other hand, the benefit of the Sinkhorn embeddings is that the embedding dimension is much smaller (a maximum of 64, against 1 024 to 4 096 for the other ones). Again, this is beneficial for dataset compression purposes, and opens a non-linear dimension reduction prospect.
On the Sinkhorn embeddings, we notice that the support points of the reference measure are randomly generated, and the weight probabilities are uniform. 
In \cite{bachoc2023gaussian}, these points and weights are optimized by Gaussian maximum likelihood instead. 
Hence, a numerical perspective to this work would be to also optimize these points and weights in the frame of Table \ref{tab:ecologicalregression}, based on cross validation error criteria for instance.
\textcolor{black}{This could result in an accuracy improvement, while still keeping a (very) small embedding dimension. However,
this would entail an additional computational cost, in particular for computing gradients with respect to the points and weights, by backpropagation.
It is thus a very challenging perspective from the numerical viewpoint, given the particularly high dimension and large sample size here.}

\textcolor{black}{Overall, setting mean embedding (linear) aside in Table \ref{tab:ecologicalregression},  it appears that the accuracy ranking of the embeddings is explained by their degrees of complexity. The most accurate one is mean embedding (Fourier) which most benefits from simplicity. Note that mean embedding is actually already performing well in \cite{flaxman2015supported,flaxman2016understanding}. Sliced Wasserstein, the next most accurate, also benefits from simplicity, although it is dependent on the number of random projections and discretized probabilities. Sinkhorn embedding, the least accurate here, suffers from the difficulty of tuning its parameters, particularly $\cU$ and $\epsilon$, as discussed above. Recall that the accuracy comparison is based on 5 repeated train-test decompositions of the entire dataset.}

In Figure~\ref{fig:ecologicalregression}, we provide a graphical example of successful distribution regression, for predicting the Democrat votes. 
We use the mean embedding with random Fourier features (having the best accuracy in Table \ref{tab:ecologicalregression}).
We split the dataset into 5 disjoint folds of sizes 195 or 196 each, we fit a kernel ridge regressor on four of the splits, and display its predictions on the fifth one, thus preventing overfitting. It appears that the ecological inference is successful, as the structure of the Democrat vote is preserved between reality and prediction.
In particular, the Democrat vote is well predicted in major cities of California and the Northeast. 
Among the rare exceptions to this accurate prediction, one can notice the extreme south of Florida, where the Democrat vote is strongly under-estimated. 

\begin{table} \setlength\tabcolsep{.5\tabcolsep}
\centering
\begin{tabular}{|cccc|cc|cc|} 
 \hline
 \multirow{4}{*}{\makecell{Hilbertian\\ embedding}} & \multirow{4}{*}{\makecell{Dim.}} & \multirow{4}{*}{\makecell{Hilbertian\\ embedding\\ runtime \\
 ($\downarrow$ is better)}} 
 & \multirow{4}{*}{\makecell{Ridge\\ regression\\ runtime 
 \\ ($\downarrow$ is better)}} & \multicolumn{2}{c|}{\makecell{Explained variance\\ score in \% \\ ($\uparrow$ is better)}} & \multicolumn{2}{c|}{\makecell{Mean absolute\\ error in \% \\ ($\downarrow$ is better)}} \\ 
 & & & & & & & \\
 & & & & Democrat & Republican & Democrat & Republican\\
 \hline
 \hline
 Constant baseline & 0 & 00m00s & $0.00$s & 0. & 0. & $12.4\pm 0.4$ & $12.7\pm 0.4$\\
 Mean embedding (linear) & 3899 & \textbf{02m30s} & $1.80$s & $27.4\pm 12$ & $03.7\pm 6.2$ & $10.0\pm 1.0$ & $13.1\pm 5.0$\\
Mean embedding (Fourier) & 4096 & 09m33s & $0.76$s & $\bm{82.1\pm 5.7}$ & $\bm{83.1\pm 2.3}$ & $\bm{4.4\pm 0.5}$ & $\bm{5.0\pm 0.3}$\\
 Sliced-Wasserstein & 1024 & 03m34s & 0.32s & $70.2\pm 5.6$ & $72.74\pm 4.1$ & $6.1\pm 0.5$ & $6.2\pm 0.4$\\
 Sliced-Wasserstein & 4096 & 03m44s & 0.68s & $75.9\pm 6.8$ & $75.1\pm 3.3$ & $5.3\pm 0.3$ & $6.2\pm 0.3$\\
 Sinkhorn & \textbf{16} & 26m49s & $\bm{0.16}$s & $50.6\pm 8.2$ & $48.8\pm 5.1$ & $7.9\pm 0.5$ & $8.3\pm 0.5$\\
 Sinkhorn & 32 & 28m27s & $\bm{0.16}$s & $67.1\pm 4.6$ & $66.0\pm 4.4$ & $6.6\pm 0.3$ & $6.9\pm 0.2$\\
 Sinkhorn & 64 & 30m42s & $0.23$s & $61.7\pm 3.0$ & $59.8\pm 4.0$ & $7.1\pm 0.3$ & $7.6\pm 0.3$\\
 \hline
\end{tabular}
\caption{We perform distribution regression to predict percentages of Democrat and Republican vote for the 2016 US presidential election, from socio-economics features extracted from 2015 US census data. We report the explained variance score and the mean absolute error over the test set, averaged over 5 random train/test splits of sizes 80\% / 20\% respectively. We also report the runtime required to compute the Hilbertian embeddings and  to perform ridge regression on the embeddings. Best scores per column are in bold font.}
\label{tab:ecologicalregression}
\end{table}

\begin{figure}
    \centering
    \begin{minipage}[]{0.42\textwidth}
    \includegraphics[width=5cm]{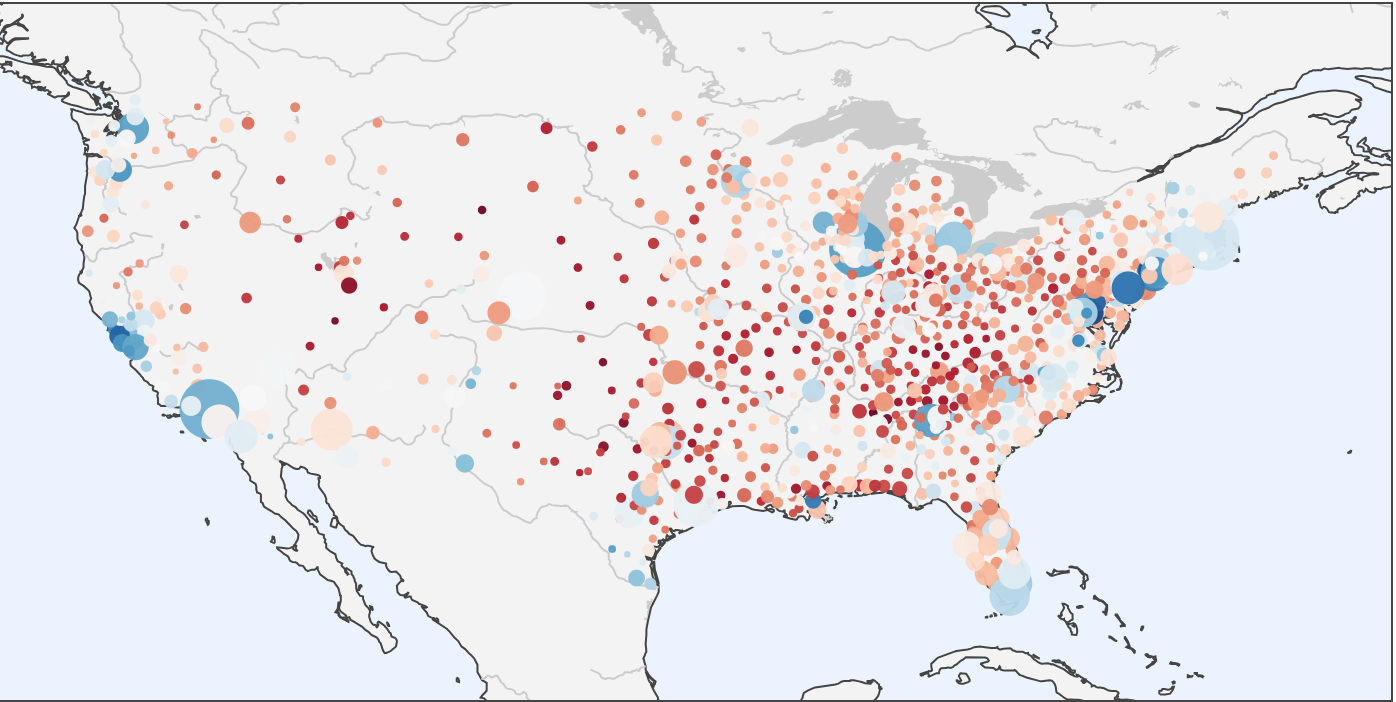}
    \subcaption{Democrat vote in the 2016 US presidential election.}
    \end{minipage}
    \hfil
    \begin{minipage}[]{0.42\textwidth}
  \includegraphics[width=5cm]{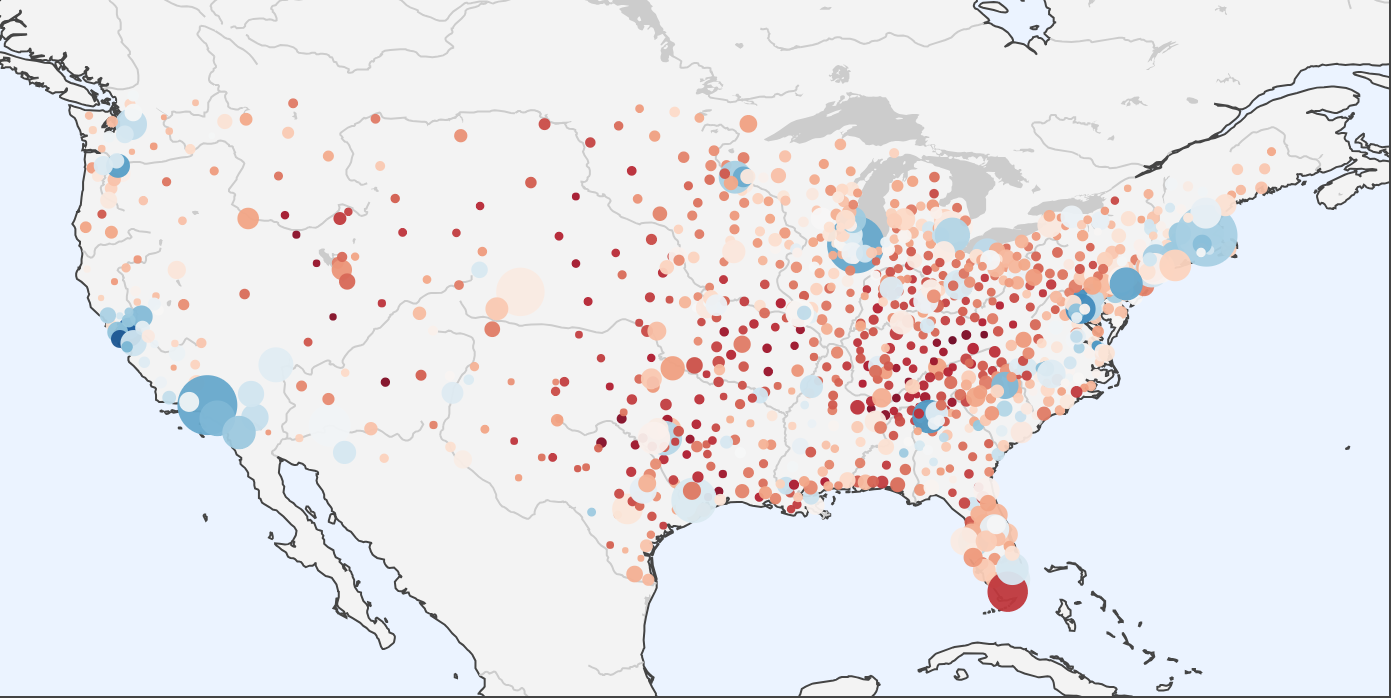}
    \subcaption{Distribution regression from {socio-economics} features: 4.4\% of mean error.}
    \end{minipage}
    \hfil
    \begin{minipage}[]{0.09\textwidth}
    \includegraphics[width=\linewidth]{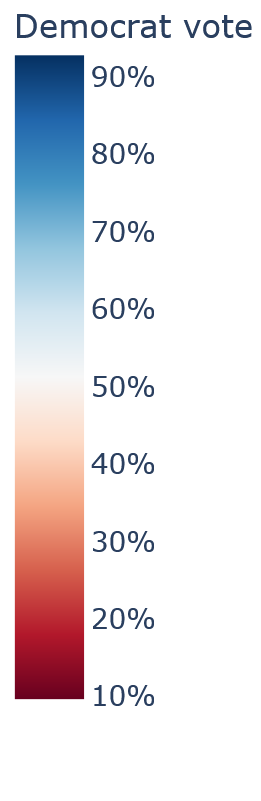}
    \end{minipage}
    \caption{Predicted and actual Democrat vote, in the 2016 US presidential election, in each of the 975 regions (Hawaii and Alaska excluded from the plot). The surface of the markers is proportional to the number of individuals in the 2015 US census data, totaling 2 490 616 individuals over the USA. The Democrat vote is successfully recovered from the socio-economics features.}
    \label{fig:ecologicalregression}
\end{figure}

\subsection{\textcolor{black}{Further insight on ecological inference with a simulation study}}
\label{subsection:ecological:simulated}

\textcolor{black}{In order to complement the previous data example on ecological inference, we now present a simpler simulation study mimicking it. We repeat $50$ Monte Carlo steps of data generation and kernel distribution regression computation. We consider $n=100$ independently and randomly generated distributions $\mu_i$ on $\mathbb{R}^d$ (representing the regions of Section~\ref{subsection:ecological:inference}), each associated to $N=200$ samples $(X_{i,j})_{j=1}^N$, independent given $\mu_i$ (representing the features of individuals/voters of Section~\ref{subsection:ecological:inference}). Given $\mu_i$, each $X_{i,j}$ is sampled as
\[
X_{i,j}
=
A_{i,j} {1}_d + B_{i,j},
\]
where $1_d$ is the vector of $\mathbb{R}^d$ composed of ones,
where $A_{i,j}$ is uniformly distributed on $[
- \alpha_i , \alpha_i] \subset \mathbb{R}$ and where $B_{i,j}$ is independent from $A_{i,j}$ and sampled as
\[
B_{i,j}
\sim 
\mathcal{N} 
\left( 
\begin{pmatrix}
    \beta_{i,1}
    \\ 
\beta_{i,2}
\\ 
0 \\
\vdots \\
0
\end{pmatrix}
,
\frac{1}{4} I_d
\right).
\]
Hence, each distribution $\mu_i$ is characterized by its parameters $\alpha_i,\beta_{i,1},\beta_{i,2}$ that are independently and uniformly distributed on $[0.05,0.1]$, $[-0.7,0.7]$ and $[-0.7,0.7]$. }

\textcolor{black}{To each individual $X_{i,j}$ there is an associated $Y_{i,j} \in \{0,1\}$ (representing the vote of the individual $X_{i,j}$) with
\[
\bP( Y_{i,j} = 1 | X_{i,j} )
=
\frac{e^{10(X_{i,j,1} - X_{i,j,2})}}{1+e^{10(X_{i,j,1} - X_{i,j,2})}},
\]
where $X_{i,j,1} $ and $X_{i,j,2} $ are the two first components of $X_{i,j}$. Finally, we let $Y_i = \frac{1}{N} \sum_{j=1}^N Y_{i,j}$, corresponding to the average vote of the region $\mu_i$.}

\textcolor{black}{
Hence, we model a situation where the variable $X_{i,j,1} $ is positively associated to the vote $Y_{i,j}=1$, the variable $X_{i,j,2} $ is negatively associated to the vote $Y_{i,j}=1$ and the other variables do not impact the vote. Also, the purpose of the variable $A_{i,j}$ above is to create a (moderate) dependence between the components of $X_{i,j}$.}

\textcolor{black}{
Regarding kernel distribution regression, we focus on the sliced Wasserstein embedding, for the sake of concision and since it provided a good tradeoff between accuracy and computation speed in Section~\ref{subsection:ecological:inference}. 
In a data-driven way,
we select  the values of the ridge parameter $\lambda$ in \eqref{eq:Rn:empirical} and  of a scale parameter $\ell$ such that the squared exponential kernel in \eqref{eq:squared:exponential} becomes $e^{- \| u - v \|_\HH^2 / \ell^2 }$. Here $\| u - v \|_\HH^2$ is numerically an average of squares over the random directions and the differences of ranked values (these ranked values corresponding to computing univariate Wasserstein distances), see \eqref{eq:sliced:wasserstein}. 
We fix the value $100$ for the number of random directions.
The selection of $\lambda$ and $\ell$ is done by minimizing the sum of squared errors over $10$ random splits of $(\mu_i,Y_i)_{i=1}^n$ between $80$ training pairs and $20$ test pairs. This minimization is over a squared regular grid of size $100$ in log scale for the values of $\lambda$ and $\ell$.  Since the simulation study here is of a smaller scale compared to Sections \ref{sec:convergencespeedxp} and \ref{subsection:ecological:inference}, the implementation consists in standalone \texttt{R} scripts.
}

\textcolor{black}{Figure \ref{figure:ecological:simulated} (top-left) provides the boxplots of the $50$ explained variance scores of the predictions by kernel distribution regression as presented above, over new independent test sets $(\mu_i,Y_i)_{i=n+1}^{n+n_{\text{test}}}$ with $n_{\text{test}} = 200$ and for $d \in \{5,10,15,20\}$. For $d=5$, most explained variance scores are above $0.98$, which is very high and confirms the efficiency of kernel distribution regression with the sliced Wasserstein embedding. The explained variance scores clearly decrease when $d$ increases, indicating a curse of dimensionality.}

\textcolor{black}{Next, for $d=5$, beyond prediction performances, we show how distribution regression can be applied to understand the impact of the $d$ features on the vote. For each $k \in \{1,\ldots,d\}$ and
for each set of samples $(X_{i,j})_{j=1}^N$,we split the set between the subset $(X_{i,j})_{j \in E_{+,k,i}}$ associated to the $N/2$ largest values of the $k$th feature $X_{i,j,k}$ and the subset $(X_{i,j})_{j \in E_{-,k,i}}$  associated to the $N/2$ smallest values. Then from the $n$ subsets $(X_{i,j})_{j \in E_{+,k,i}}$, $i=1,\ldots,n$, we use the trained kernel distribution regression predictors above to compute corresponding predictions  of percentages of votes $(\widehat{Y}_{+,k,i})_{i=1}^n$. Similarly we compute the predictions $(\widehat{Y}_{-,k,i})_{i=1}^n$. 
In each of the five panels, top-center, top-right and bottom, of Figure \ref{figure:ecological:simulated}, corresponding to the five first Monte Carlo steps of the simulation study, we show the five boxplots of $(\widehat{Y}_{+,k,i} - \widehat{Y}_{-,k,i})_{i=1}^n$ for $k \in \{1,2,3,4,5\}$. The boxplots for $k=1$ and $k=2$ clearly stand out visually, with the highest values for $k=1$ and the lowest values for $k=2$. Hence for $k=1$, the kernel distribution regression predictors successfully detect that larger values of $X_{i,j,1}$ are associated to a higher percentage of vote, and conversely for $k=2$. Furthermore, the predictors do not suggest similar effects for the other features $3,4,5$, which indeed have no effects on the vote in the true unknown data generating process.
We note that these conclusions are obtained by performing predictions on empirical distributions that do not necessarily ``look like'' the ones associated to the observed vote percentages $Y_i$, since we select only the lower or higher values of the feature $k$ to create $E_{-,k,i}$ and $E_{+,k,i}$. This can be interpreted as a robustness of kernel distribution regression, in this simulation study.}

\begin{figure}
\begin{tabular}{ccc}
\includegraphics[width=0.3\textwidth]{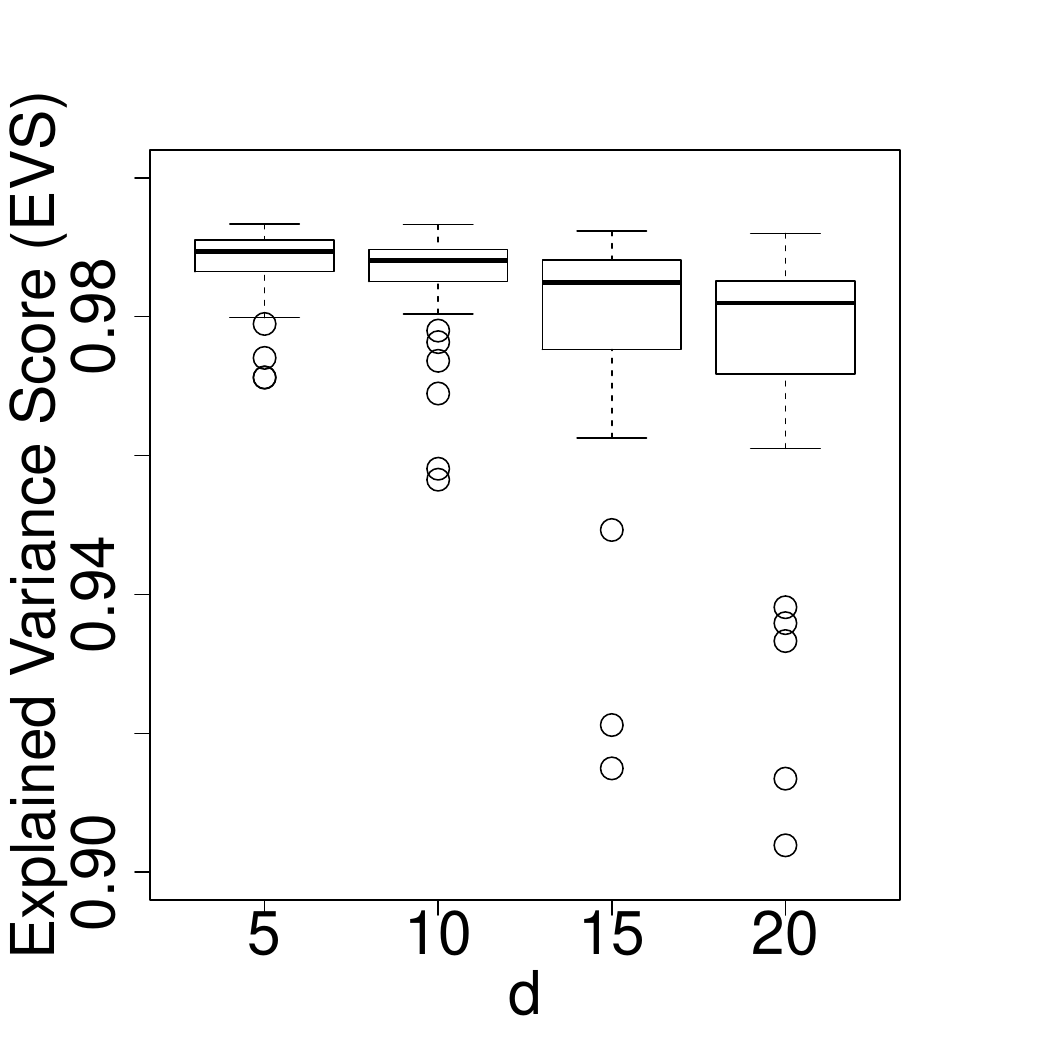} & \includegraphics[width=0.3\textwidth]{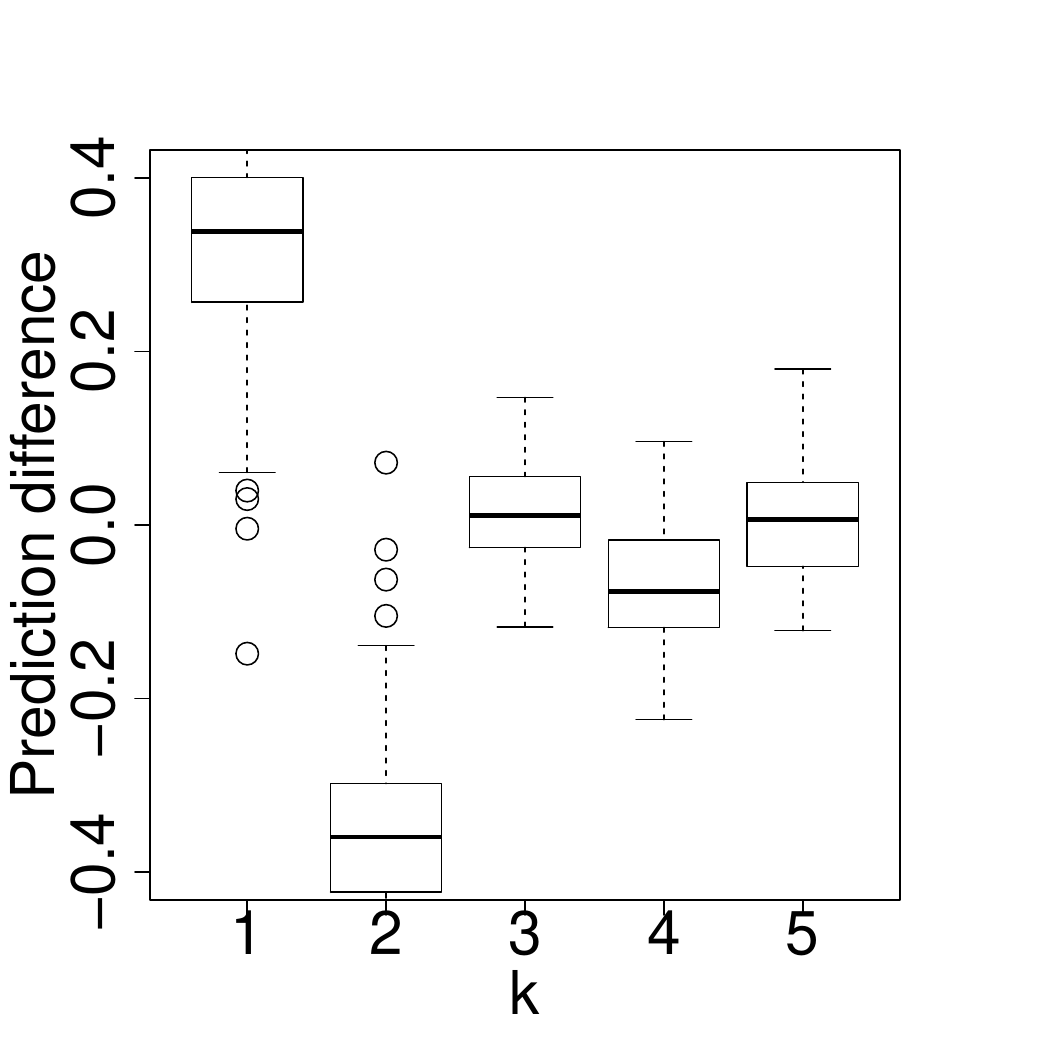}& \includegraphics[width=0.3\textwidth]{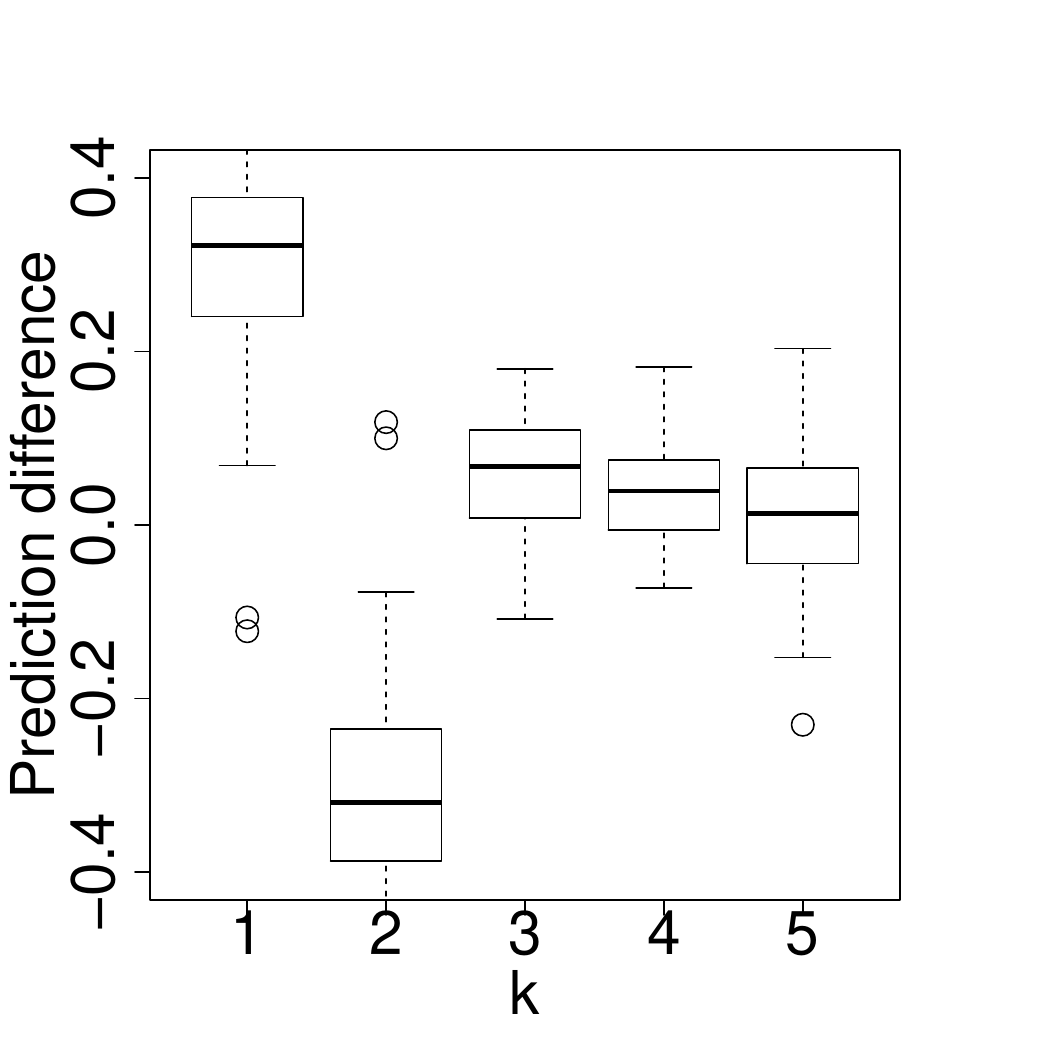} \\
\includegraphics[width=0.3\textwidth]{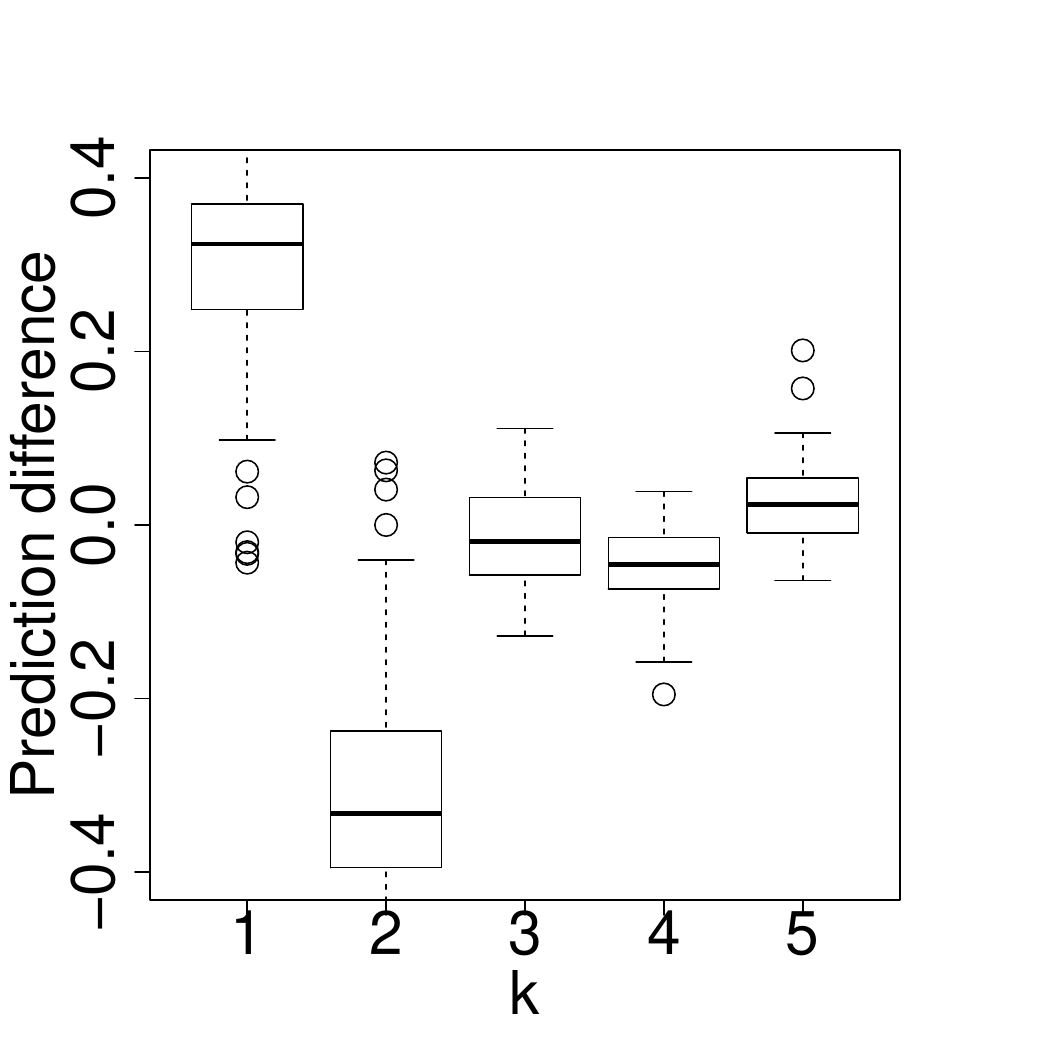} & \includegraphics[width=0.3\textwidth]{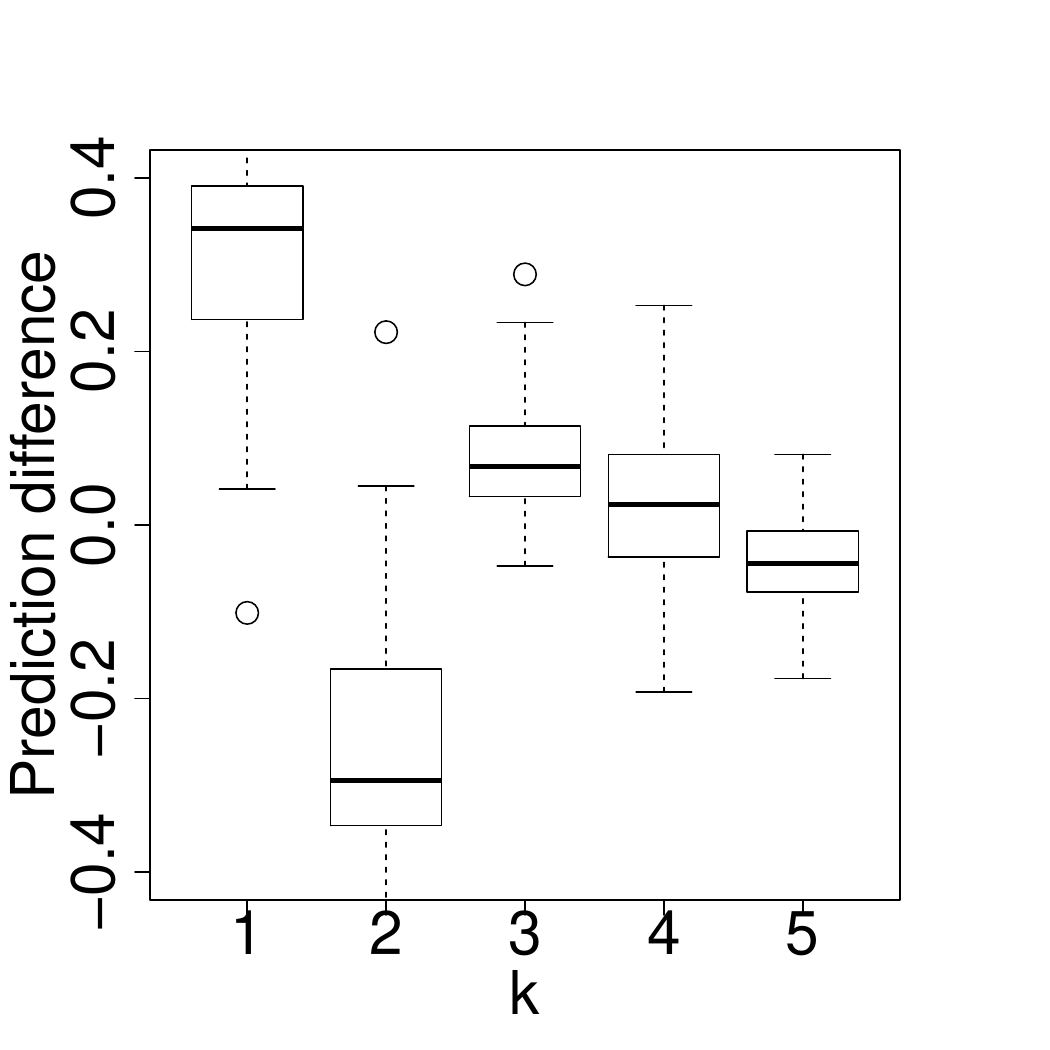} & \includegraphics[width=0.3\textwidth]{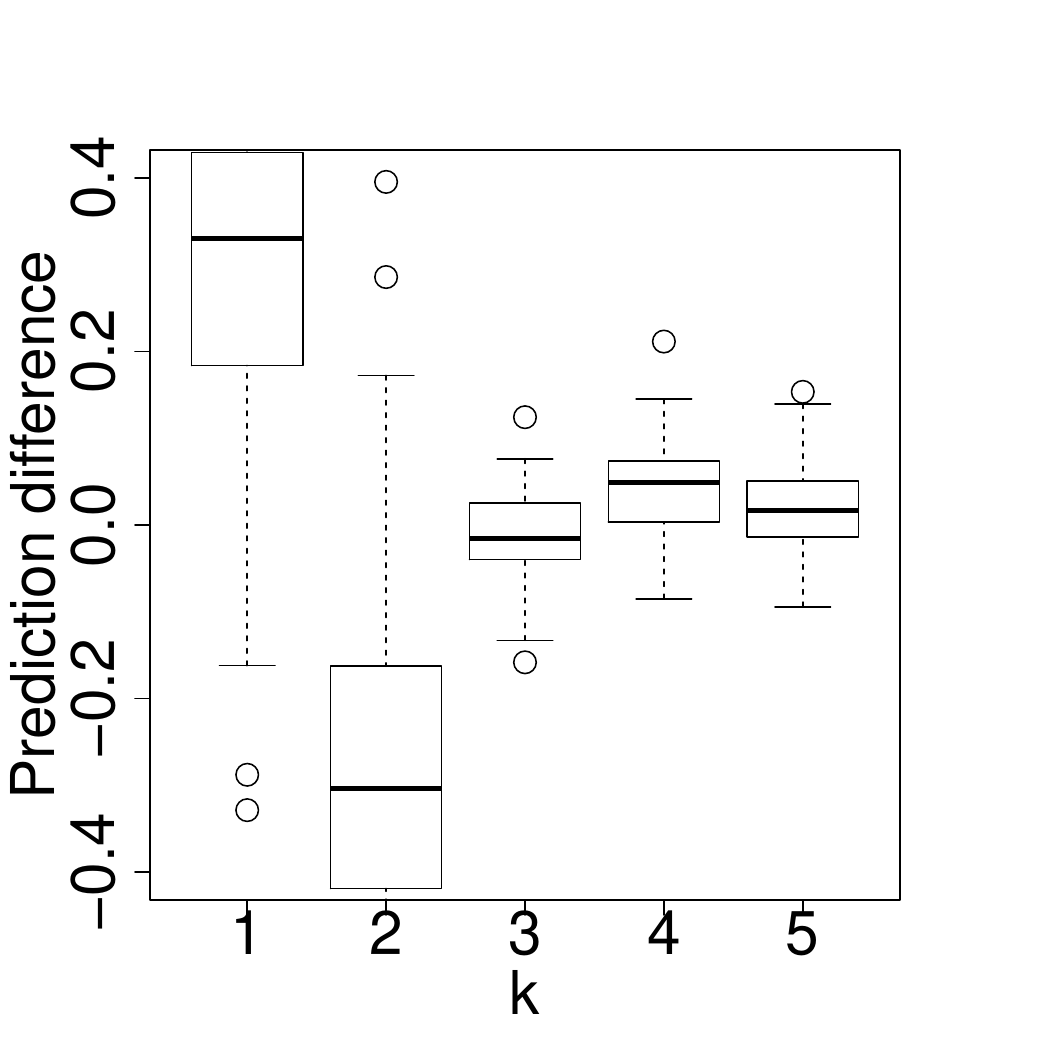}
\end{tabular}
\caption{\textcolor{black}{Simulation study of Section \ref{subsection:ecological:simulated} mimicking the ecological inference dataset of Section \ref{subsection:ecological:inference}. Top-left: boxplots of explained variance scores over test sets as function of ambient dimension $d$. Top-center, top-right and bottom: boxplots of the differences $(\widehat{Y}_{+,k,i} - \widehat{Y}_{-,k,i})_{i=1}^n$ as a function of $k$ for the five first Monte Carlo steps of the simulation study.
}}
\label{figure:ecological:simulated}
\end{figure}

\section{Conclusion} 
\label{section:conclusion}

This work contributes to an improved unified learning theory of distribution regression based on Hilbertian embeddings. We provide general error bounds, for the effect of two-stage sampling, based on innovative proof techniques (Section \ref{section:general:hilbert}). This enables us to improve the state of the art for three Hilbertian embedding methods (Section \ref{s:applitheorem}).
Applications to other potential embeddings would be possible as well.
Similarly, we focus on providing bounds in expectation, or convergence rates in probability, but our proof methods could naturally allow for concentration bounds as well. 

Open questions that go beyond the scope
of this article, but are currently under investigation, include the following. 
First, we study minimax rates for three Hilbertian embeddings in the well-specified case where their associated RKHSs contain the unknown regression function. An important question is to compare the flexibility of these well-specification assumptions, by comparing the RKHSs and their norms. 
This would provide additional theoretical insight, that could improve the understanding of numerical comparisons between these embeddings in distribution regression, like the comparison in Section \ref{s:expe}. 
Second, generalizing the analysis of kernel methods with distribution inputs, under a two-stage sampling, beyond regression would be valuable. 
In this view, other problems of interest include kernel-based classification, dimension reduction and testing.

\begin{acks}[Acknowledgments]
The authors are grateful to two anonymous referees, an Associate
Editor and the Editor, for their constructive comments that lead to an improvement of the paper. The authors benefited from useful feedback by Juan Cuesta-Albertos and Antonio Cuevas.
\end{acks}

\begin{funding}
Fran\c{c}ois Bachoc was supported by the Projects GAP (ANR-21-CE40-0007)
and
BOLD (ANR-19-CE23-0026)
of the French National Research Agency (ANR)
and by the Chair UQPhysAI of the Toulouse ANITI AI Cluster. The research of Alberto González-Sanz is partially supported by grant PID2021-128314NB-I00 funded by MCIN/AEI/10.13039/\linebreak[1]501100011033/FEDER,
UE.
\end{funding}

\appendix

\section{Proofs for Section \ref{section:general:Hilbert:error:bound}}
\label{appendix:proofs:error:bound:general:hilbert}

In Appendix \ref{appendix:proofs:error:bound:general:hilbert}, all the analysis is conducted conditionally to $(\mu_i,Y_i)_{i=1}^n$, which are thus treated as deterministic. In particular, the symbols $\bE$ and $\bP$ are implicitly conditional on $(\mu_i,Y_i)_{i=1}^n$. 
Note that, with $x_i = x_{\mu_i}$, then also $(x_i)_{i=1}^n$ is treated as deterministic.

\subsection{The existing bounds} \label{subsection:existing:bounds}

Consider the setting of Theorem \ref{theorem:error:bound}.
Here we review the bounds and proofs provided by recent existing references. These bounds are given for kernels on distributions using specific Hilbertian embeddings, but they can be straightforwardly stated for kernels on general Hilbert spaces, as we do below.
We will be especially close to \cite{meunier2022distribution} in terms of exposition, but similar ideas are also used for instance in  \cite{szabo2015two,szabo2016learning}. 
The purpose is twofold. First, this will allow to appreciate our improvement in Theorem \ref{theorem:error:bound}. Second, these existing bounds will help as intermediary results in our proofs. 

For $x \in \HH$, recall $K_{x} = K(x , \cdot) \in \HH_K$.
We write 
\begin{align*}
\Phi_n 
 : 
 ~ ~ ~
 \mathbb{R}^n 
&
\to 
     \HH_K 
     \\
     \begin{pmatrix}
         \alpha_1 \\
         \vdots \\
         \alpha_n
     \end{pmatrix}
&
     \mapsto 
     \sum_{i=1}^n \alpha_i K_{x_i},
\end{align*}
\begin{align*}
\Phi_{n,N}
 : 
 ~ ~ ~
 \mathbb{R}^n 
&
\to 
     \HH_K 
     \\
     \begin{pmatrix}
         \alpha_1 \\
         \vdots \\
         \alpha_n
     \end{pmatrix}
&
     \mapsto 
     \sum_{i=1}^n \alpha_i K_{x_{N,i}},
\end{align*}
\begin{align*}
L_n
 : 
 ~ ~ ~
    \HH_K 
&
\to 
     \HH_K 
     \\
    f
&
     \mapsto 
    \frac{1}{n}
    \sum_{i=1}^n 
    f(x_i) 
    K_{x_i}
\end{align*}
and
\begin{align*}
L_{n,N}
 : 
 ~ ~ ~
    \HH_K 
&
\to 
     \HH_K 
     \\
    f
&
     \mapsto 
    \frac{1}{n}
    \sum_{i=1}^n 
    f(x_{N,i}) 
    K_{x_{N,i}}.
\end{align*}

We can check that $L_n$ and $ L_{n,N} $ are semi-definite positive and self-adjoint on $\HH_K$. 
Then the next lemma is used for instance in \cite{meunier2022distribution} and can be checked directly. We \textcolor{black}{recall} $ Y_{[n]} = (Y_1 , \ldots , Y_n)^\top $.

\begin{lemma} \label{lem:expression:hatf}
    We have, with  $ \mathrm{id} $ the identity operator,
    \[
\hat{f}_n 
=
\left( 
L_n + \lambda id
\right)^{-1}
\frac{ \Phi_n }{n}
Y_{[n]}
    \]
    and
    \[
\hat{f}_{n,N}
=
\left( 
L_{n,N} + \lambda id
\right)^{-1}
\frac{ \Phi_{n,N} }{n}
Y_{[n]}.
    \]
\end{lemma}

Let us now bound $ \hat{f}_n - \hat{f}_{n,N} $. We have
\begin{align*}
\left\|
\hat{f}_n - \hat{f}_{n,N}
\right\|_{\HH_K} 
& = 
\left\|
\left( L_n + \lambda \mathrm{id} \right)^{-1}
\frac{\Phi_n}{n} Y_{[n]}
-
\left( L_{n,N} + \lambda \mathrm{id} \right)^{-1}
\frac{\Phi_{n,N}}{n} Y_{[n]}
\right\|_{\HH_K} 
\\
& \leq 
A + B,
\end{align*}
where we let
\[
A =
\left\|
\left( L_{n,N} + \lambda \mathrm{id} \right)^{-1}
\left( 
\frac{\Phi_{n,N}}{n} Y_{[n]}
-
\frac{\Phi_{n}}{n} Y_{[n]}
\right)
\right\|_{\HH_K} 
\]
and 
\[
B =
\left\|
\left[ 
\left( L_{n,N} + \lambda \mathrm{id} \right)^{-1}
-
\left( L_{n} + \lambda \mathrm{id} \right)^{-1}
\right]
\frac{\Phi_{n}}{n} Y_{[n]}
\right\|_{\HH_K}. 
\]

We have
\begin{align*}
    A \leq 
    \frac{1}{\lambda} 
    \left\|
\frac{\Phi_n}{n} Y_{[n]}
-
\frac{\Phi_{n,N}}{n} Y_{[n]}
\right\|_{\HH_K}.
\end{align*}

Hence, using the definition of $Y_{\max,n}$, then the reproducing property, and then Lemma~\ref{lemma:holder:square:exponential}, 
\begin{align} \label{eq:existing:approach:bound:A}
    A & \leq 
     \frac{ Y_{\max,n}}{\lambda n} 
     \sum_{i=1}^n
       \left\|
       K_{x_i}
       -
       K_{x_{N,i}} 
       \right\|_{\HH_K} 
       \\
       & = 
       \frac{ Y_{\max,n}}{\lambda n} 
     \sum_{i=1}^n
\sqrt{
2 F(0) - 2 F(  \| x_i - x_{N,i} \|_{\HH} )
       } . \notag 
       \\
       & \leq 
       \frac{ \sqrt{2} \sqrt{A_F} Y_{\max,n}}{\lambda n} 
     \sum_{i=1}^n
     \| x_i - x_{N,i} \|_{\HH}. \notag
       \end{align}

       Then, with similar arguments
\begin{align*}
B & =
\left\|
\left[ 
\left( L_{n,N} + \lambda \mathrm{id} \right)^{-1}
-
\left( L_{n} + \lambda \mathrm{id} \right)^{-1}
\right]
\frac{\Phi_{n}}{n} Y_{[n]}
\right\|_{\HH_K}
\\
& = 
\left\|
\left( L_{n,N} + \lambda \mathrm{id} \right)^{-1}
\left[ 
 L_{n} 
-
L_{n,N}
\right] 
\left( L_{n} + \lambda \mathrm{id} \right)^{-1}
\frac{\Phi_{n}}{n} Y_{[n]}
\right\|_{\HH_K} 
\\
\text{(Lemma \ref{lem:expression:hatf}:)} ~ ~ & = 
\left\|
\left( L_{n,N} + \lambda \mathrm{id} \right)^{-1}
\left[ 
 L_{n} 
-
L_{n,N}
\right] 
\hat{f}_n
\right\|_{\HH_K}
\\
& \leq 
\frac{1}{\lambda}
\left\|
\left[ 
 L_{n} 
-
L_{n,N}
\right] 
\hat{f}_n
\right\|_{\HH_K}
\\ &=
\frac{1}{\lambda}
\left\|
\frac{1}{n}
\sum_{i=1}^n
\left(
\hat{f}_n( x_i ) K_{x_i}
-
\hat{f}_n( x_{N,i} ) K_{x_{N,i}}
\right)
\right\|_{\HH_K}.
\end{align*}

Then 
\begin{align} \label{eq:existing:approach:bound:B}
B & \leq 
\frac{1}{\lambda n}
\sum_{i=1}^n
\left\|
\hat{f}_n( x_i )
\left(
 K_{x_i}
-
 K_{x_{N,i}} 
 \right) 
 \right\|_{\HH_K} 
 +
 \frac{1}{\lambda n}
\sum_{i=1}^n
\left\|
 \left(
 \hat{f}_n( x_i )
-
 \hat{f}_n( x_{N,i} )
\right)
 K_{x_{N,i}} 
 \right\|_{\HH_K} 
\notag 
\\
& \leq 
\frac{1}{\lambda n}
\sum_{i=1}^n
\|\hat{f}_n\|_{\HH_K} 
\| K_{x_i}
-
 K_{x_{N,i}}  \|_{\HH_K} 
 +
 \frac{1}{\lambda n}
\sum_{i=1}^n
 \|\hat{f}_n\|_{\HH_K} 
 \| K_{x_i}
-
 K_{x_{N,i}}  \|_{\HH_K} . 
\end{align}

Above, for bounding $\left\|
\hat{f}_n( x_i )
\left(
 K_{x_i}
-
 K_{x_{N,i}} 
 \right) \right\|_{\HH_K}$, we have used that $\hat{f}_n( x_i ) = \langle \hat{f}_n,K_{x_i} \rangle_{\HH_K}$, Cauchy-schwarz inequality and that $\|K_{x_i}\|_{\HH_K} = \sqrt{F(0)} = 1$. We have bounded  the quantity  $\left\|
 \left(
 \hat{f}_n( K_{x_i} )
-
 \hat{f}_n( K_{x_{N,i}} )
\right)
 K_{x_{N,i}} \right\|_{\HH_K}$ similarly.

Then, as for handling $A$,
\[
B \leq 
\frac{2 \sqrt{2 } \sqrt{A_F}
\|\hat{f}_n\|_{\HH_K}
}{\lambda n}
\sum_{i=1}^n
\|x_n - x_{N,i} \|_{\HH}.
\]

Hence finally,
\[
\left\|
\hat{f}_n - \hat{f}_{n,N}
\right\|_{\HH_K} 
\leq 
\frac{
\sqrt{2 } \sqrt{A_F}
\left(
2 
\|\hat{f}_n\|_{\HH_K}
+
Y_{\max,n}
\right)
}{\lambda n}
\sum_{i=1}^n
\|x_n - x_{N,i} \|_{\HH}.
\]

For $s \geq 1$, using now Hölder inequality and then 
Condition \ref{condition:near:unbias} and Lemma \ref{lemma:sum:and:power},
\begin{align*}
&\bE 
\left[
\left\|
\hat{f}_n - \hat{f}_{n,N}
\right\|_{\HH_K}^s
\right]^{1/s}
\\
=
&
\frac{
\sqrt{2 } \sqrt{A_F}
\left(
2 
\|\hat{f}_n\|_{\HH_K}
+
Y_{\max,n}
\right)
}{\lambda}
\bE
\left[
\left(
\frac{1}{n}
\sum_{i=1}^n
\|x_n - x_{N,i} \|_{\HH}
\right)^s
\right]^{1/s}
\\
\leq 
&
\frac{
\sqrt{2 } \sqrt{A_F}
\left(
2 
\|\hat{f}_n\|_{\HH_K}
+
Y_{\max,n}
\right)
}{\lambda}
\bE
\left[ 
\frac{1}{n}
\sum_{i=1}^n
\|x_n - x_{N,i} \|_{\HH}^s
\right]^{1/s}
\\
\leq 
&
\frac{
\sqrt{2 } \sqrt{A_F}
\left(
2 
\|\hat{f}_n\|_{\HH_K}
+
Y_{\max,n}
\right)
}{\lambda}
\bE
\left[
\frac{1}{n}
\sum_{i=1}^n
\frac{2 . 2^s c_s}{N^{s/2}}
\right]^{1/s}
\\
=
&
\frac{
2^{1+1/s} c_s^{1/s}
\sqrt{2 } \sqrt{A_F}
\left(
2 
\|\hat{f}_n\|_{\HH_K}
+
Y_{\max,n}
\right)
}{\sqrt{N} \lambda}.
\end{align*}

We thus have the following lemma, given by the proofs in \cite{meunier2022distribution} (see also \cite{szabo2015two,szabo2016learning}).

\begin{lemma} \label{lemma:existing:bound:general:Hilb}
  Under the setting of Theorem \ref{theorem:error:bound}, we have, for all $s \geq 1$,
  \[
  \bE 
\left[
\left\|
\hat{f}_n - \hat{f}_{n,N}
\right\|_{\HH_K}^s
\right]^{1/s}
\leq 
\frac{
2^{1+1/s} c_s^{1/s}
\sqrt{2 } \sqrt{A_F}
\left(
2 
\|\hat{f}_n\|_{\HH_K}
+
Y_{\max,n}
\right)
}{\sqrt{N} \lambda},
  \]
  where $c_s$ is from Condition \ref{condition:near:unbias} and $A_F$ from Lemma \ref{lemma:holder:square:exponential}.
  We recall that here $\bE$ denotes the conditional expectation given $(\mu_i,Y_i)_{i=1}^n$. 
\end{lemma}

More precisely, the arguments in \cite{meunier2022distribution} that correspond to the proofs of Lemma \ref{lemma:existing:bound:general:Hilb} are given between (29) and (35) in this reference.
For \cite{szabo2016learning}, these arguments are given in particular in Sections 7.1.1 and 7.2.2. 
For \cite{szabo2015two}, these arguments are given in particular in Section A.1.11.

\subsection{Proof of Theorem \ref{theorem:error:bound}}
\label{section:proof:main:error:bound}

\subsubsection{Starting the bound}

Using Lemma \ref{lemma:convex:gradient}, we obtain 
\begin{align*}
 \lambda \| \hat{f}_n - \hat{f}_{n,N} \|_{\HK}^2 
\leq &
\frac{1}{n} 
\sum_{i=1}^n 
( \hat{f}_n - \hat{f}_{n,N} ) (  x_{N,i} ) 
\hat{f}_n( x_{N,i} ) 
- 
( \hat{f}_n - \hat{f}_{n,N} ) (  x_i ) 
\hat{f}_n( x_i ) 
\\
& + 
\frac{1}{n} 
\sum_{i=1}^n 
Y_i 
\left( 
( \hat{f}_n - \hat{f}_{n,N} ) (  x_i ) 
-
( \hat{f}_n - \hat{f}_{n,N} ) (  x_{N,i} ) 
\right) 
\\
= &
\frac{1}{n} 
\sum_{i=1}^n 
\left[
( \hat{f}_n - \hat{f}_{n,N} ) (  x_{N,i} ) 
- 
( \hat{f}_n - \hat{f}_{n,N} ) (  x_i ) 
\right] 
\hat{f}_n( x_i ) 
\\
& + 
\frac{1}{n} 
\sum_{i=1}^n 
( \hat{f}_n - \hat{f}_{n,N} ) (  x_{N,i} ) 
( \hat{f}_n (  x_{N,i} )   - \hat{f}_{n}  (  x_i ) ) 
\\
& + 
\frac{1}{n} 
\sum_{i=1}^n 
Y_i 
\left( 
( \hat{f}_n - \hat{f}_{n,N} ) (  x_i ) 
-
( \hat{f}_n - \hat{f}_{n,N} ) (  x_{N,i} ) 
\right)  
\\
= &
\underbrace{
\frac{1}{n} 
\sum_{i=1}^n
\hat{f}_n( x_i ) 
\left[
( \hat{f}_n - \hat{f}_{n,N} ) (  x_{N,i} ) 
- 
( \hat{f}_n - \hat{f}_{n,N} ) (  x_i ) 
\right]}_{= C}
\\
& + 
\underbrace{
\frac{1}{n} 
\sum_{i=1}^n 
Y_i 
\left[ 
( \hat{f}_n - \hat{f}_{n,N} ) (  x_i ) 
-
( \hat{f}_n - \hat{f}_{n,N} ) (  x_{N,i} ) 
\right]}_{= B}  
\\
& + 
\underbrace{
\frac{1}{n} 
\sum_{i=1}^n 
( \hat{f}_n - \hat{f}_{n,N} ) (  x_i ) 
( \hat{f}_n (  x_{N,i} )   - \hat{f}_{n}  (  x_i ) ) }_{= D}
\\
& + 
\underbrace{
\frac{1}{n} 
\sum_{i=1}^n 
\left[ 
( \hat{f}_n - \hat{f}_{n,N} ) (  x_{N,i} )
-
( \hat{f}_n - \hat{f}_{n,N} ) (  x_i )
\right]
\left[ 
 \hat{f}_n (  x_{N,i} ) 
 -
 \hat{f}_n (  x_i ) 
\right]}_{= A}.  
\end{align*}

Recall the notation of the statement of Theorem \ref{theorem:error:bound}, $c_n = \| \hat{f}_n \|_{\HK}$ and $Y_{\max,n} = \max_{i=1,\ldots,n} |Y_i|$. For the rest of the proof, let us also introduce the notation $T_{n,N} = \| \hat{f}_n - \hat{f}_{n,N}\|_{\HK}$. We also let $\cte$ be a quantity that does not depend on $n$, $N$, $\lambda$, $\mu_1,\ldots,\mu_n$, $Y_1 , \ldots , Y_n$, and which value is allowed to change from occurrence to occurrence. 

\subsubsection{Bounding $\mathbb{E} A$}

We have
\begin{align*}
A = &
\frac{1}{n} 
\sum_{i=1}^n 
\left[ 
( \hat{f}_n - \hat{f}_{n,N} ) (  x_{N,i} )
-
( \hat{f}_n - \hat{f}_{n,N} ) (  x_i )
\right]
\left[ 
 \hat{f}_n (  x_{N,i} ) 
 -
 \hat{f}_n (  x_i ) 
\right]
\\
= &
\frac{1}{n} 
\sum_{i=1}^n 
\left\langle \hat{f}_n - \hat{f}_{n,N} , 
K_{x_{N,i}}
-
K_{  x_i } 
\right\rangle_{\HK}
\left\langle 
 \hat{f}_n,
 K_{x_{N,i}}
-
K_{  x_i } 
\right\rangle_{\HK}.
\end{align*}
Hence, using the Cauchy-Schwarz inequality,
\begin{align*}
|A|
\leq &  
\frac{1}{n} 
\sum_{i=1}^n 
\| \hat{f}_n - \hat{f}_{n,N}\|_{\HK}
\| \hat{f}_n \|_{\HK}
\|K_{x_{N,i}}
-
K_{  x_i }  \|_{\HK}^2.
\end{align*}

Hence, again using Cauchy-Schwarz, and with similar arguments as in Section  \ref{subsection:existing:bounds}, 
\[
\bE  |A| 
\leq 
\frac{\cte c_n}{N}
\sqrt{\bE[ T_{n,N}^2 ]}.
\]

\subsubsection{Bounding $\mathbb{E} B$} \label{subsubsection:bounding:B}

\begin{lemma} \label{lemma:hatfnN:moins:hatfnNi}
Let $s \geq 1$. 
There is a constant $c_1$ such that the following holds. 
For $i = 1, \ldots , n$,
let $\bar{f}_{n,N,i}$ be defined as $\hat{f}_{n,N}$ but with $x_{N,i}$ replaced by $x_i$. 
Then,
\[
\bE^{1/s} 
\left[
\| \hat{f}_{n,N} - \bar{f}_{n,N,i}
\|_{\HH_K}^s 
\right]
\leq 
\frac{
c_1(Y_{\max,n} + c_n)
}{
\lambda n \sqrt{N}
}
+
\frac{
c_1(Y_{\max,n} + c_n)
}{
\lambda^2 n N
}.
\]
\end{lemma}
\begin{proof}
Let us use Lemma \ref{lemma:convex:gradient}
with, for $j=1,\ldots,n$, $j \neq i$, $\ell_j(h) = \tilde{\ell}_j(h) = h(x_{N,j})$, $\ell_i(h) = h(x_{N,i})$, $\tilde{\ell}_i(h) = h(x_{i})$, $f = \hat{f}_{n,N}$ and $g = \bar{f}_{n,N,i}$. Using also the definition $Y_{\max,n}$, we have 
    \begin{align*}
&
\lambda \| \hat{f}_{n,N} - \bar{f}_{n,N,i}
\|_{\HH_K}^2
\\
\leq &
\frac{1}{n} 
\left( 
\hat{f}_{n,N} - \bar{f}_{n,N,i}
\right)
(x_{i}) 
\hat{f}_{n,N}(x_{i})
-
\frac{1}{n}
\left( 
\hat{f}_{n,N} - \bar{f}_{n,N,i}
\right)
(x_{N,i}) 
\hat{f}_{n,N}(x_{N,i})
\\
 & +
 \frac{1}{n} 
 Y_i 
 \left[
 \left(\hat{f}_{n,N} - \bar{f}_{n,N,i} \right)
 (x_{N,i})
-
 \left(\hat{f}_{n,N} - \bar{f}_{n,N,i} \right)
 (x_{i})
 \right]
 \\
 = &
 \frac{1}{n} 
\left( 
\hat{f}_{n,N} - \bar{f}_{n,N,i}
\right)
(x_{i}) 
\hat{f}_{n,N}(x_{i})
-
 \frac{1}{n} 
\left( 
\hat{f}_{n,N} - \bar{f}_{n,N,i}
\right)
(x_{i}) 
\hat{f}_{n,N}(x_{N,i})
\\
&
+
 \frac{1}{n} 
\left( 
\hat{f}_{n,N} - \bar{f}_{n,N,i}
\right)
(x_{i}) 
\hat{f}_{n,N}(x_{N,i})
-
\frac{1}{n} 
\left( 
\hat{f}_{n,N} - \bar{f}_{n,N,i}
\right)
(x_{N,i}) 
\hat{f}_{n,N}(x_{N,i})
\\
 & +
 \frac{1}{n} 
 Y_i 
 \left[
 \left(\hat{f}_{n,N} - \bar{f}_{n,N,i} \right)
 (x_{N,i})
-
 \left(\hat{f}_{n,N} - \bar{f}_{n,N,i} \right)
 (x_{i})
 \right]
 \\
 \leq &
 \frac{ \| \hat{f}_{n,N} - \bar{f}_{n,N,i}
\|_{\HH_K}}{n}
\left(
2
\|\hat{f}_{n,N}\|_{\HH_K} 
.
\|
K_{  x_{i} }
-
K_{  x_{N,i} }  \|_{\HK}  
+
Y_{\max,n}
\|
K_{  x_{N,i} }
-
K_{  x_{i} }  \|_{\HK}  
\right).
    \end{align*}
Hence
\begin{align*}
\| \hat{f}_{n,N} - \bar{f}_{n,N,i}
\|_{\HH_K}
\leq &
\frac{2
\|\hat{f}_{n,N}\|_{\HH_K} 
}{ \lambda n}
\|
K_{  x_{i} }
-
K_{  x_{N,i} }  \|_{\HK}  
+
\frac{ Y_{\max,n}}{ \lambda n }
\|
K_{  x_{N,i} }
-
K_{  x_{i} }  \|_{\HK}  
\\ 
\leq &
\frac{2
(c_n + Y_{\max,n}) 
}{ \lambda n}
\|
K_{  x_{i} }
-
K_{  x_{N,i} }  \|_{\HK}  
+
\frac{2
\| \hat{f}_{n,N} - \hat{f}_n
\|_{\HH_K}
}{ \lambda n}
\|
K_{  x_{i} }
-
K_{  x_{N,i} }  \|_{\HK}.  
\end{align*}

We then use the Cauchy-Schwarz inequality, together with Lemma \ref{lemma:existing:bound:general:Hilb} and Condition \ref{condition:near:unbias}, which yields
    \begin{align*}
        \bE 
        \left[
\| \hat{f}_{n,N} - \bar{f}_{n,N,i}
\|_{\HH_K}^s 
\right]
\leq 
\left(
\frac{
\cte
(c_n + Y_{\max,n}) 
}{ \lambda n \sqrt{N}}
\right )^s 
+
\left(
\frac{\cte}{\lambda n}
\right)^s
\sqrt{
\frac{
\left(
c_n
+
Y_{\max,n}
\right)^{2s}
}{N^s \lambda^{2s}}
}
\sqrt{
\frac{1}{N^s}
}.
    \end{align*}
    Re-arranging this last bound concludes the proof.
\end{proof}

We have
\begin{align} \label{eq:coupling:un}
B
= &  
\frac{1}{n} 
\sum_{i=1}^n 
Y_i 
\left[ 
( \hat{f}_n - \hat{f}_{n,N} ) (  x_i ) 
-
( \hat{f}_n - \hat{f}_{n,N} ) (  x_{N,i} ) 
\right].
\end{align}

Let us define $\tilde{f}_{n,N}$ in the same way as $\hat{f}_{n,N}$, but where $(x_{N,1} , \ldots , x_{N,n})$ is replaced by an independent copy
(with the same distribution) $(\tilde{x}_{N,1} , \ldots , \tilde{x}_{N,n})$. Assume also that $(\tilde{x}_{N,1} , \ldots , \tilde{x}_{N,n})$ is chosen as being stochastically independent from $(x_{N,i},a_{N,i},b_{N,i})_{i=1}^n$ (from Condition \ref{condition:near:unbias}).

Then, for $i = 1, \ldots , n$,
let $\tilde{f}_{n,N,i}$ be defined as $\tilde{f}_{n,N}$ but with $\tilde{x}_{N,i}$ replaced by $x_{N,i}$.
With these definitions, the variable
$( \hat{f}_n - \hat{f}_{n,N} ) (  x_i ) 
-
( \hat{f}_n - \hat{f}_{n,N} ) (  x_{N,i} ) $
has the same distribution as the variable
$(   \hat{f}_{n} - \tilde{f}_{n,N,i} ) (  x_i ) 
-
(  \hat{f}_{n} - \tilde{f}_{n,N,i} ) (  x_{N,i} ) $.
Indeed, both variables are of the form $( \hat{f}_n - g ) (  x_i ) 
-
( \hat{f}_n - g ) (  z_{N,i} ) $ where $g$ is computed from $(z_{N,j})_{j=1}^n$, that are distributed as $(x_{N,j})_{j=1}^n$. 

Hence
\begin{equation} \label{eq:coupling:deux}
\bE B 
=
\bE
\left[ 
\frac{1}{n} 
\sum_{i=1}^n 
Y_i 
\left(
( \hat{f}_n -  \tilde{f}_{n,N,i} ) (  x_i ) 
-
( \hat{f}_n -  \tilde{f}_{n,N,i} ) (  x_{N,i} )
\right)
\right].
\end{equation}

Then we have
\begin{align*}
\bE B 
= &
\bE
\left[
\underbrace{ 
\frac{1}{n} 
\sum_{i=1}^n 
Y_i 
\left(
( \hat{f}_n -  \tilde{f}_{n,N} ) (  x_i ) 
-
( \hat{f}_n -  \tilde{f}_{n,N} ) (  x_{N,i} ) 
\right)
}_{= B_2}
\right]
\\
& +
\bE
\left[ 
\underbrace{
\frac{1}{n} 
\sum_{i=1}^n 
Y_i
\left(
(  \tilde{f}_{n,N} -  \tilde{f}_{n,N,i} ) (  x_i ) 
-
( \tilde{f}_{n,N} -  \tilde{f}_{n,N,i} ) (  x_{N,i} ) 
\right)
}_{= B_1}
\right].
\end{align*}

We have using Cauchy-Schwarz and Condition \ref{condition:near:unbias}, 

\begin{align*}
\bE
 |B_1| 
\leq &
Y_{\max,n} 
\max_{i=1,\ldots,n}
\bE 
\left[ 
\| \tilde{f}_{n,N} -  \tilde{f}_{n,N,i}
\|_{\HK}
\| 
K_{x_i}
-
K_{x_{N,i}}
\|_{\HK}
\right] 
\\ 
\leq &
\frac{ \cte Y_{\max,n} }{ \sqrt{N} } 
\max_{i=1,\ldots,n}
\sqrt{
\bE 
\left[ 
\| 
\tilde{f}_{n,N} -  \tilde{f}_{n,N,i}
\|_{\HK}^2
\right] 
}.
\end{align*}

Note that 
\begin{equation} \label{eq:trick:loo}
( \tilde{f}_{n,N} -  \tilde{f}_{n,N,i} )
=
 \tilde{f}_{n,N} -  \tilde{f}_{n,N,-i} 
+
 \tilde{f}_{n,N,-i} -  \tilde{f}_{n,N,i} 
\end{equation}
where $\tilde{f}_{n,N,-i} $ is computed as $\tilde{f}_{n,N}$ but with $\tilde{x}_{N,i}$ replaced by $x_i$. Both random quantities $ \tilde{f}_{n,N} -  \tilde{f}_{n,N,-i} $ and 
$ \tilde{f}_{n,N,i} - \tilde{f}_{n,N,-i}  $
have the same distribution as the quantity 
$\hat{f}_{n,N} - \bar{f}_{n,N,i}$ in Lemma \ref{lemma:hatfnN:moins:hatfnNi}. 
Hence, from Lemma \ref{lemma:hatfnN:moins:hatfnNi}, we have
\begin{align} \label{eq:Bdeux}
\bE
 |B_1| 
\leq &
\frac{ \cte Y_{\max,n} }{ \sqrt{N} } 
\frac{
c_1(Y_{\max,n} + c_n)
}{
\lambda n \sqrt{N}
}
+
\frac{ \cte Y_{\max,n} }{ \sqrt{N} } 
\frac{
c_1(Y_{\max,n} + c_n)
}{
\lambda^2 n N
} 
\notag
\\
= &
\frac{
\cte  Y_{\max,n} (Y_{\max,n} + c_n)
}{
\lambda n N
}
+
\frac{
\cte Y_{\max,n} (Y_{\max,n} + c_n)
}{
\lambda^2 n N^{3/2}
}.
\end{align}

Then, consider 
\[
B_2 
=
\frac{1}{n} 
\sum_{i=1}^n 
Y_i 
\left(
( \hat{f}_n -  \tilde{f}_{n,N} ) (  x_i ) 
-
( \hat{f}_n -  \tilde{f}_{n,N} ) (  x_{N,i} )
\right).
\]

For $i=1,\ldots,n$, we apply Lemma \ref{lemma:hadamard:diff:Hilbert} with $f$ there given by $\hat{f}_n - \tilde{f}_{n,N}$. This gives,
\[
( \hat{f}_n -  \tilde{f}_{n,N} ) (  x_{N,i} ) 
-
( \hat{f}_n - \tilde{f}_{n,N} ) (  x_i ) 
=
\psi_{N,i}( x_{N,i} - x_{i} )
+
r_{N,i},
\]
where $\psi_{N,i} $ is linear continuous and satisfies, for $x$ with $\|x\|_{\HH}  =1 $, $| \psi_{N,i}(x) | \leq \cte \|\hat{f}_n - \tilde{f}_{n,N}\|_{\HH_K}$, and where $| r_{N,i}| \leq  \cte \|\hat{f}_n - \tilde{f}_{n,N}\|_{\HH_K} \|x_{N,i} - x_{i}\|_{\HH} ^2 $.

This gives
\[
\bE 
|B_2|
\leq
\bE
\left[
\left|
\underbrace{
\frac{1}{n} 
\sum_{i=1}^n 
Y_i 
\psi_{N,i}( x_{N,i} - x_{i} )
}_{= B_{22}}
\right| 
\right]
+
\bE
\left[
\left|
\underbrace{
\frac{1}{n} 
\sum_{i=1}^n 
Y_i 
r_{N,i}
}_{= B_{21}}
\right| 
\right].
\]
We have
\[
\bE
|B_{21}|
\leq 
\cte Y_{\max,n}
\max_{i=1,\ldots,n}
\bE
\left[
\|\hat{f}_n - \tilde{f}_{n,N}\|_{\HH_K} \|x_{N,i} - x_{i}\|_{\HH} ^2
\right].
\]

Note that $\hat{f}_n - \tilde{f}_{n,N}$ has the same distribution as
$\hat{f}_n - \hat{f}_{n,N}$. With the Cauchy-Schwarz inequality and Condition \ref{condition:near:unbias}, this yields
\begin{equation} \label{eq:Bundeux}
\bE |B_{21}|
\leq 
\frac{\cte Y_{\max,n}
}{N}
\sqrt{
\bE
\left[
T_{n,N}^2
\right]
}.
\end{equation}

Then, for $B_{22}$, we apply Condition \ref{condition:near:unbias} which gives, with $a_{N,i}$ and $b_{N,i}$ defined in this condition, 
\[
\bE 
| B_{22} |
\leq 
\bE 
\left[ 
\left|
\underbrace{
\frac{1}{n} 
\sum_{i=1}^n 
Y_i 
\psi_{N,i}(
a_{N,i})
}_{= B_{222}}
\right| 
\right]
+
\bE 
\left[ 
\left|
\underbrace{
\frac{1}{n} 
\sum_{i=1}^n 
Y_i 
\psi_{N,i}(b_{N,i})
}_{= B_{221}}
\right| 
\right].
\]
We have, using Cauchy-Schwarz, and the bound on $b_{N,i}$ in  Condition \ref{condition:near:unbias},
\begin{equation*}
\bE  |B_{221}| 
\leq 
\frac{\cte Y_{\max,n}}{N}
\sqrt{
\bE
\left[
\|\hat{f}_n - \tilde{f}_{n,N}\|_{\HH_K}^2 
\right]
}.
\end{equation*}

As before, $\hat{f}_n - \tilde{f}_{n,N}$ has the same distribution as
$\hat{f}_n - \hat{f}_{n,N}$.
This yields
\begin{equation} \label{eq:Bunundeux}
\bE |B_{221}| 
\leq 
\frac{\cte Y_{\max,n}}{N}
\sqrt{
\bE
\left[
T_{n,N}^2 
\right]
}.
\end{equation}

Consider finally $B_{222}$, with
\begin{equation} \label{eq:Bdeuxdeuxdeux}
B_{222}
=
\frac{1}{n} 
\sum_{i=1}^n 
Y_i 
\psi_{N,i}(
a_{N,i}).
\end{equation}

Let 
$\cB$ be the $\sigma$-algebra generated by 
$(\tilde{x}_{N,1} , \ldots , \tilde{x}_{N,n})$. Then $\tilde{f}_{n,N}$ is $\cB$-measurable (recall that $x_1 , \ldots , x_n$ are deterministic in Appendix \ref{appendix:proofs:error:bound:general:hilbert}). Then, for $i = 1, \ldots , n$, also $\psi_{N,i}$ is $\cB$-measurable (as it depends only on $\tilde{f}_{n,N}$, $\hat{f}_n$ and 
$x_i$). On the other hand, $a_{N,i}$ is independent of $\cB$ by definition of $(\tilde{x}_{N,1} , \ldots , \tilde{x}_{N,n})$.

Hence we have, for $i = 1, \ldots , n$, using the Riesz representation theorem,
\[
\bE
\left[
\psi_{N,i}(
a_{N,i})
\right]
=
\bE 
\left[
\bE
\left[
\left.
\psi_{N,i}(
a_{N,i})
\right| 
\cB
\right]
\right]
=0.
\]

Then, also, for $i \neq j$, conditionally to $\cB$, the variables $a_{N,i} $ and $a_{N,j}$ are independent and keep their unconditional distributions.
Thus we have 
\begin{align*}
\bE
\left[
\psi_{N,i}(
a_{N,i})
\psi_{N,j}(
a_{N,j})
\right]
 = &
\bE 
\left[
\bE
\left[
\left.
\psi_{N,i}(
a_{N,i})
\psi_{N,j}(
a_{N,j})
\right| 
\cB
\right]
\right]
\\
= &
\bE 
\left[
\bE
\left[
\left.
\psi_{N,i}(
a_{N,i})
\right| 
\cB
\right]
\bE
\left[
\left.
\psi_{N,j}(
a_{N,j})
\right| 
\cB
\right]
\right]
\\
= & 0.
\end{align*}

Hence we obtain, exploiting again the independence between $\hat{f}_n - \tilde{f}_{n,N}$ and $a_{N,i}$,
\begin{align*}
\bE |B_{222}|
\leq &
Y_{\max,n}
\sqrt{
\frac{1}{n^2} 
\sum_{i=1}^n 
\bE
\left[
(\psi_{N,i}(
a_{N,i}))^2
\right]
}
\notag
\\
 \leq & 
\cte Y_{\max,n}
\sqrt{
\frac{1}{n^2} 
\sum_{i=1}^n 
\bE
\left[
\|\hat{f}_n - \tilde{f}_{n,N}\|_{\HK}^2
\| a_{N,i}\|_{\HH}^2
\right]
}
\notag
\\
 = & 
\cte Y_{\max,n}
\sqrt{
\frac{1}{n^2} 
\sum_{i=1}^n 
\bE
\left[
\|\hat{f}_n - \tilde{f}_{n,N}\|_{\HK}^2
\right]
\bE
\left[
\| a_{N,i}\|_{\HH}^2
\right]
}
\notag
\\
\leq &
\frac{\cte Y_{\max,n}}{\sqrt{n}}
\max_{i=1,\ldots,n} 
\sqrt{
\bE
\left[
\|\hat{f}_n - \tilde{f}_{n,N}\|_{\HK}^2
\right]
\bE
\left[
\| a_{N,i}\|_{\HH}^2
\right]
}
\notag
\\
\leq &
\frac{\cte Y_{\max,n}}{\sqrt{n} \sqrt{N}}
\sqrt{
\bE
\left[
\|\hat{f}_n - \tilde{f}_{n,N}\|_{\HK}^2
\right]}.
\end{align*}

As before, $\hat{f}_n - \tilde{f}_{n,N}$ has the same distribution as
$\hat{f}_n - \hat{f}_{n,N}$.
This yields
\begin{align} \label{eq:Bununun}
\bE |B_{222}|
\leq
\frac{
\cte
Y_{\max,n}
}{
\sqrt{n}
\sqrt{N} }
\sqrt{
\bE\left[
T_{n,N}^2
\right]
}.
\end{align}

Combining \eqref{eq:Bdeux}, \eqref{eq:Bundeux}, \eqref{eq:Bunundeux} and \eqref{eq:Bununun} yields

\begin{align*}
\bE
B
\leq &
\frac{\cte Y_{\max,n} (Y_{\max,n} + c_n)}{\lambda n N}
+
\frac{\cte Y_{\max,n} (c_n + Y_{\max,n})}{\lambda^2 n N^{3/2}}
\\ 
&
+
\frac{\cte Y_{\max,n} }{N} 
\sqrt{
\bE\left[
T_{n,N}^2
\right]
}
+
\frac{\cte Y_{\max,n}}{\sqrt{n} \sqrt{N}}
\sqrt{
\bE
\left[
T_{n,N}^2
\right]
}.
\end{align*}

\subsubsection{Bounding $\mathbb{E} C$}

The term $\bE C$ is handled exactly as $\bE B$ since $Y_i$ (from $B$) is replaced by $\hat{f}_n(x_i)$ (in $C$). When handling $B$ we only used that $Y_i$ is deterministic and bounded by $Y_{\max,n}$. For $C$, we only use that $\hat{f}_n(x_i)$ is deterministic and bounded by $c_n$. Hence we have

\begin{align*}
\bE
C
\leq &
\frac{\cte c_n (Y_{\max,n} + c_n)}{\lambda n N}
+
\frac{\cte c_n (c_n + Y_{\max,n})}{\lambda^2 n N^{3/2}}
\\ 
&
+
\frac{\cte c_n }{N} 
\sqrt{
\bE\left[
T_{n,N}^2
\right]
}
+
\frac{\cte c_n}{\sqrt{n} \sqrt{N}}
\sqrt{
\bE
\left[
T_{n,N}^2
\right]
}.
\end{align*}

\subsubsection{Bounding $\mathbb{E} D$}

Recall 
\[
D = \frac{1}{n} 
\sum_{i=1}^n 
( \hat{f}_n - \hat{f}_{n,N} ) (  x_i ) 
( \hat{f}_n (  x_{N,i} )   - \hat{f}_{n}  (  x_i ) ). 
\]

We use the same definitions $\tilde{f}_{n,N}$ and $\tilde{f}_{n,N,i}$ as for bounding $\bE B$ above. Then, with the same arguments as above, 
\[
\bE D
=
\bE 
\left[
\frac{1}{n} 
\sum_{i=1}^n 
( \hat{f}_n - \tilde{f}_{n,N,i} ) (  x_i ) 
( \hat{f}_n (  x_{N,i} )   - \hat{f}_{n}  (  x_i ) )
\right].
\]

Hence
\begin{align*}
\bE D
= &
\bE 
\left[
\underbrace{
\frac{1}{n} 
\sum_{i=1}^n 
( \hat{f}_n - \tilde{f}_{n,N} ) (  x_i ) 
( \hat{f}_n (  x_{N,i} )   - \hat{f}_{n}  (  x_i ) )
}_{=D_2}
\right]
\\
& + 
 \bE 
 \left[
\underbrace{
\frac{1}{n} 
\sum_{i=1}^n 
(  \tilde{f}_{n,N} - \tilde{f}_{n,N,i}) (  x_i ) 
( \hat{f}_n (  x_{N,i} )   - \hat{f}_{n}  (  x_i ) ) 
}_{=D_1}
\right].
\end{align*}

Using Cauchy-Schwarz, we obtain 
\[
\bE |D_1|
\leq 
c_n 
\sqrt{
\bE 
\left[ 
 \|K_{x_{N,i}} - K_{x_i} \|_{\HK}^2
\right]
}
\max_{i=1,\ldots,n}
\sqrt{
\bE
\left[ 
 \|
 \tilde{f}_{n,N}
 -
 \tilde{f}_{n,N,i}
  \|_{\HK}^2
\right]
}.
\]

Then $\|K_{x_{N,i}} - K_{x_i} \|_{\HK}^2$ above is treated with the 
same arguments as in Section \ref{subsection:existing:bounds}. Also, 
$ \|
 \tilde{f}_{n,N}
 -
 \tilde{f}_{n,N,i}
  \|_{\HK}$ is treated as in \eqref{eq:trick:loo}. This yields
\begin{align} \label{eq:bound:un}
\bE |D_1|
\leq & 
\frac{c_n \cte}{\sqrt{N}} 
\left(
\frac{
c_1(Y_{\max,n} + c_n)
}{
\lambda n \sqrt{N}
}
+
\frac{
c_1(Y_{\max,n} + c_n)
}{
\lambda^2 n N
}
\right) \notag
\\ 
= &
\frac{
\cte c_n (Y_{\max,n} + c_n)
}{
\lambda n N
}
+
\frac{
\cte c_n (Y_{\max,n} + c_n)
}{
\lambda^2 n N^{3/2}
}.
\end{align}

For $i=1,\ldots,n$, we apply Lemma \ref{lemma:hadamard:diff:Hilbert} with $f$ there given by $\hat{f}_n $. This gives
\[
\hat{f}_n (  x_{N,i} )   - \hat{f}_{n}  (  x_i ) 
=
\psi_{N,i}
\left( 
x_{N,i}
-
x_i
\right) 
+
r_{N,i},
\]
where $\psi_{N,i} $ is linear continuous and satisfies, for $x$ with $\|x\|_{\HH} =1 $, $| \psi_{N,i}(x) | \leq \cte \|\hat{f}_n \|_{\HK}$, and where $| r_{N,i}| \leq  \cte \|\hat{f}_n \|_{\HK} \|x_{N,i} - x_i\|_{\HH}^2 $. 
In addition, we apply Condition \ref{condition:near:unbias}, and we can write
\[
x_{N,i} - x_i
=
a_{N,i} + b_{N,i},
\]
with $a_{N,i}$ and $b_{N,i}$ defined in this condition.
Then,
\begin{align*}
D_2 
= & 
\frac{1}{n} 
\sum_{i=1}^n 
( \hat{f}_n - \tilde{f}_{n,N} ) (  x_i ) 
( \hat{f}_n (  x_{N,i} )   - \hat{f}_{n}  (  x_i ) ) 
\\
= &
\underbrace{
\frac{1}{n} 
\sum_{i=1}^n 
( \hat{f}_n - \tilde{f}_{n,N} ) (  x_i ) 
\psi_{N,i}(a_{N,i})
}_{= D_{22}}
\\
& + 
\underbrace{
\frac{1}{n} 
\sum_{i=1}^n 
( \hat{f}_n - \tilde{f}_{n,N} ) (  x_i ) 
\left( 
\psi_{N,i}(b_{N,i})
+ r_{N,i}
\right).
}_{= D_{21}}
\end{align*}

We have, using Cauchy-Schwarz,
\begin{align*}
    \bE |D_{21}|
    \leq &
    \cte
    c_n
    \sqrt{ \bE 
    \left[ 
\| \hat{f}_n - \tilde{f}_{n,N} \|_{\HK}^2 
    \right]
    }
        \max_{i=1,\ldots,n}
    \sqrt{ 
\bE
\left[ 
\|b_{N,i}\|_{\HH}^2
+
\|x_{N,i}
-
x_i
\|^4_{\HH} 
\right]
    }
    \\
    \leq &
    \frac{\cte c_n}{N} 
   \sqrt{ \bE 
    \left[ 
\| \hat{f}_n - \tilde{f}_{n,N} \|_{\HK}^2 
    \right]
    },
\end{align*}
using Condition \ref{condition:near:unbias}. 
As observed when handling $\bE B$ above, $ \hat{f}_n - \tilde{f}_{n,N}$ has the same distribution as $ \hat{f}_n - \hat{f}_{n,N}$. Hence 
\begin{align} \label{eq:bound:Ddeuxun}
    \bE |D_{21}|
    \leq 
    \frac{\cte c_n}{N} 
   \sqrt{ \bE 
    \left[ 
T_{n,N}^2
    \right]
    }.
\end{align}
Finally, consider 
\[
D_{22}
=
\frac{1}{n} 
\sum_{i=1}^n 
( \hat{f}_n - \tilde{f}_{n,N} ) (  x_i ) 
\psi_{N,i}(a_{N,i}).
\]
Above, $\psi_{N,i}$ is deterministic since it is defined from $\hat{f}_n$ and $x_i$. Also, similarly as when handling $\bE B$, $a_{N,i}$ and
$( \hat{f}_n - \tilde{f}_{n,N} ) $ are independent.
Hence, with the same arguments as when handling $B$ above, $D_{22}$ is a sum of decorrelated centered variables.
Hence, we have 
\begin{align*}
\bE [ D_{22}^2 ]
= &
\frac{1}{n^2}
\sum_{i=1}^n 
\bE
\left[ 
\left(
( \hat{f}_n - \tilde{f}_{n,N} ) (  x_i ) 
\right)^2
\left(
\psi_{N,i}(a_{N,i})
\right)^2
\right]
\\
\leq &
\frac{\cte}{n^2}
\sum_{i=1}^n 
\bE
\left[ 
\|
\hat{f}_n - \tilde{f}_{n,N} 
\|_{\HK}^2
c_n^2
\|a_{N,i}\|^2_{\HH}
\right]
\\
= &
\frac{\cte c_n^2}{n^2}
\sum_{i=1}^n
\bE
\left[ 
\|
\hat{f}_n - \tilde{f}_{n,N} 
\|_{\HK}^2
\right]
\bE
\left[ 
\|a_{N,i}\|^2_{\HH}
\right]
\\
\leq &
\frac{\cte c_n^2}{n}
\bE
\left[ 
\| \hat{f}_n - \tilde{f}_{n,N}\|_{\HK}^2
\right]
\frac{1}{N},
\end{align*}
using Condition \ref{condition:near:unbias} at the end. Hence,
\begin{equation} \label{eq:D:deux:deux}
\bE  |D_{22}| 
\leq 
\frac{\cte c_n}{\sqrt{n} \sqrt{N}}
\sqrt{
\bE
\left[ 
 T_{n,N}^2
\right]
}.
\end{equation}

Combining \eqref{eq:bound:un}, \eqref{eq:bound:Ddeuxun} and \eqref{eq:D:deux:deux}, we obtain 
\begin{align*}
\bE D
\leq & 
\frac{
\cte c_n (Y_{\max,n} + c_n)
}{
\lambda n N
}
+
\frac{
\cte c_n (Y_{\max,n} + c_n)
}{
\lambda^2 n N^{3/2}
}
\\
& +
    \frac{\cte c_n}{N} 
   \sqrt{ \bE 
    \left[ 
T_{n,N}^2
    \right]
    }
    +
    \frac{\cte c_n}{\sqrt{n} \sqrt{N}}
\sqrt{
\bE
\left[ 
 T_{n,N}^2
\right]
}.
\end{align*}

\subsubsection{Completing the proof}

Combining the bounds on $\bE A$, $\bE B$, $\bE C$ and $\bE D$ yields
\begin{align*}
    \lambda 
    \bE [ T_{n,N}^2]
   \leq &
   \sqrt{\bE [ T_{n,N}^2 ]} 
   \left( 
\frac{\cte (c_n + Y_{\max,n}) }{N} 
+
\frac{\cte (Y_{\max,n} + c_n)}{\sqrt{n} \sqrt{N}} 
   \right) 
   \\
   & + 
   \frac{\cte  (Y_{\max,n} + c_n)^2}{\lambda n N}
+
\frac{\cte  (c_n + Y_{\max,n})^2}{\lambda^2 n N^{3/2}}.
\end{align*}

  For $x ,a ,b \geq 0$, if $x^2 \leq a x +b$ then $x \leq \max(2 a , b/a)$. This is seen by separating the cases $x \geq b/a$ and $x \leq b/a$. Hence

\begin{align*}
\sqrt{    \bE [ T_{n,N}^2]}
   \leq &
\frac{\cte (c_n + Y_{\max,n}) }{ \lambda N} 
+
\frac{\cte (Y_{\max,n} + c_n)}{\lambda \sqrt{n} \sqrt{N}} 
\\ 
& + 
\left(
\frac{\cte (c_n + Y_{\max,n}) }{N} 
+
\frac{\cte (Y_{\max,n} + c_n)}{\sqrt{n} \sqrt{N}} 
\right)^{-1}
 \\ 
 &
 \left(
 \frac{\cte  (Y_{\max,n} + c_n)^2}{\lambda n N}
+
\frac{\cte  (c_n + Y_{\max,n})^2}{\lambda^2 n N^{3/2}}
\right).
\end{align*}

Re-arranging we obtain
\begin{align*}
\sqrt{    \bE [ T_{n,N}^2]}
   \leq &
\frac{\cte (c_n + Y_{\max,n}) }{ \lambda N} 
+
\frac{\cte (Y_{\max,n} + c_n)}{\lambda \sqrt{n} \sqrt{N}} 
\\ 
& + 
\left(
1
+
\frac{\sqrt{N} }{\sqrt{n} } 
\right)^{-1}
\left(
 \frac{\cte  (Y_{\max,n} + c_n)}{\lambda n }
+
\frac{\cte  (c_n + Y_{\max,n})}{\lambda^2 n \sqrt{N}}
\right).
\end{align*}

This completes the proof.

\subsection{Lemmas}

The following two lemmas are elementary.

\begin{lemma} \label{lemma:holder:square:exponential}
Recall the definition $F(t) = e^{-t^2}$.
 There is an absolute constant $A_F$ such that for $t \geq 0$,
 \[
 1 - F(t)
\leq 
A_F t^2.
 \]
\end{lemma}

\begin{lemma} \label{lemma:sum:and:power}
For $u,v,w \geq 0$, we have $(u + v)^w \leq 2^w (u^w + v^w)$.
\end{lemma}

The next lemma may be known by the experts, but we nevertheless provide a proof for self-sufficiency.

\begin{lemma} \label{lemma:kernel:change:input}

Let $f \in \HH_K$.
Let $k \in \mathbb{N}$ and $u,v_1,\ldots,v_k \in \HH$. Let $f_k : \mathbb{R}^k \to \mathbb{R}$ be defined by
\[
f_k(t_1 , \ldots , t_k) = 
f(u + \sum_{i=1}^k t_i v_i). 
\]
Let $K_k$ be the kernel on $\mathbb{R}^k$ defined by 
\[
K_k( t_1 , \ldots , t_k , t'_1 , \ldots , t'_k ) 
=
K( u + \sum_{i=1}^k t_i v_i , u + \sum_{i=1}^k t'_i v_i ).
\]
Let $\HH_{K_k}$ be the RKHS of $K_k$. Then $f_k \in \HH_{K_k}$ and $\|f_k\|_{\HH_{K_k}} \leq \|f\|_{\HK}$.   
\end{lemma}

\begin{proof}
    Let $\bar{f}$ be the restriction of $f$ to 
    \[
    \{  u + \sum_{i=1}^k t_i v_i , t_1 , \ldots , t_k \in \mathbb{R}  \}
    \]
    and $\bar{K}$ be the restriction of $K$ to the same space. Let $\HH_{\bar{K}}$ be the RKHS of $\bar{K}$.
    Then \cite[Th. 6]{berlinet2011reproducing}, $\bar{f}$ belongs to  $\HH_{\bar{K}}$ and $\|\bar{f} \|_{\HH_{\bar{K}}} \leq \|f \|_{\HH_K}$. From \citet[Th. 3]{berlinet2011reproducing}, $\bar{f}$ is the pointwise limit of a Cauchy sequence $(\bar{f}_n)_{n \in \mathbb{N}}$ in $\HH_{\bar{K}}$ of the form 
    \[
    \bar{f}_n( \cdot ) = 
    \sum_{i=1}^n \alpha^n_i K \left(  u + \sum_{j=1}^k t^n_{i,j} v_j  , \cdot \right). 
    \]
    Hence $f_k(t_1 , \ldots , t_k)$ is the limit (pointwise) of 
    \[
        \sum_{i=1}^n \alpha^n_i K \left(  u + \sum_{j=1}^k t^n_{i,j} v_j  , u + \sum_{j=1}^k t_{j} v_j  \right). 
    \]
    Hence again from \citet[Thm. 3]{berlinet2011reproducing}, $f_k \in \HH_{K_k}$ and 
    \begin{align*}
        \|f_k\|_{\HH_{K_k}}^2
        =
        \lim_{n \to \infty} 
   \sum_{i,i'=1}^n \alpha^n_i \alpha^n_{i'} K \left(  u + \sum_{j=1}^k t^n_{i,j} v_j  , u + \sum_{j'=1}^k t^n_{i',j'} v_{j'} \right) 
   =
    \|\bar{f}\|_{\HH_{\bar{K}}}^2.
    \end{align*}
    This concludes the proof.
\end{proof}

The next lemma enables to linearize (with quantitative control) the functions in $\HH_K$.

\begin{lemma} \label{lemma:hadamard:diff:Hilbert}
There exists an absolute constant $c_2$ such that the following holds. 
Let $f \in \HH_K$. Then for each $x \in \HH$ there exists a unique linear continuous function $\psi_x : \HH \to \mathbb{R}$ such that  for $y \in \HH$,
\begin{equation} \label{eq:Hadamard:deux}
\left| 
f( x + y )
-
f(x)
-
\psi_x(y)
\right| 
\leq 
c_2 \|f \|_{\HH_K}
\|y\|_{\HH}^2.
\end{equation}
Furthermore 
\begin{equation} \label{eq:Hadamard:un}
\sup_{x \in \HH}
\sup_{\substack{y \in \HH \\ \|y\|_{\HH} = 1 } }
| \psi_x (y) |
\leq 
c_2 \|f \|_{\HH_K}.
\end{equation}

\end{lemma}

\begin{proof}
    Let $x , y \in \HH$ with $ \|y\|_\HH = 1 $. Consider the function
    \[
    t \in \mathbb{R} 
    \mapsto 
    f(x + t y).
    \]
    From Lemma \ref{lemma:kernel:change:input}, this function is in the RKHS of the kernel $(t , t') \mapsto e^{-|t-t'|^2}$, with RKHS norm no larger than $\|f\|_{\HH_K}$. Hence, see for instance \citet[Lem. 4.1]{van2009adaptive}, this function is twice continuously differentiable with first and second derivative bounded in absolute value by $c_2 \|f\|_{\HH_K}$, when choosing $c_2$ large enough. Applying a Taylor expansion (on the real line), we obtain, for $t_0 \in \mathbb{R}$, 
    \[
\left| 
f(x + t_0 y) 
-
f(x)
- 
\left(
\frac{\partial}{\partial t}
f(x + ty) 
\right)_{t = 0}
t_0
\right|
\leq 
\frac{t_0^2}{2} 
c_2 \|f\|_{\HH_K}.
    \]
Hence, defining 
\[
\psi_x(t_0 y)
=
\left(
\frac{\partial}{\partial t}
f(x + ty) 
\right)_{t = 0}
t_0,
\]
we obtain that \eqref{eq:Hadamard:deux} holds.
Equation \eqref{eq:Hadamard:un} also holds from  the above comment on the first derivative (on the real line). It thus remains to show that $\psi_x$ is linear. By definition $\psi_x$ is homogeneous of degree one. 

Let $z_1 , z_2 \in \HH$ and $t_1,t_2 \in \mathbb{R}$.  
As seen before we have
\[
f( x + t z_1 + t z_2 )
=
f(x) + 
t
\psi_x ( z_1 + z_2 )
+
r,
\]
with $|r| \leq c_2 \|f\|_{\HH_K} \|t z_1 + t {z_2}\|_{\HH}^2/2 =  c_2 \|f\|_{\HH_K}  \|z_1+z_2\|_\HH^2 t^2/2$. 

Let $u, v \in \HH$ such that  $\|u\|_\HH=1$, $\|v\|_\HH = 1$ and $\langle u,v \rangle_\HH = 0$ and let $a_1,b_1,a_2,b_2 \in \mathbb{R}$ such that $z_1 = a_1u + b_1 v$ and $z_2 = a_2u + b_2 v$. 

The function 
\[
(t_1,t_2) \in \mathbb{R}^2 
\mapsto
f( x + t_1 z_1 + t_2 z_2 ) 
\]
is obtained by linear change of inputs from the function
\[
(s_1,s_2) \in \mathbb{R} 
\mapsto 
f( x + s_1 u + s_2 v ) 
\]
that is in the RKHS of the kernel $ (s_1,s_2,s'_1,s'_2) \mapsto e^{-{(s_1-s'_1)^2 - (s_2 - s'_2)^2} } $ and thus is twice differentiable.

Applying a (two-dimensional) Taylor expansion to this function, we obtain  
\begin{align*}
 f( x + t z_1 + t z_2 )  
= &
f(x)
+
\left(
\frac{\partial}{ \partial t_1 } 
f(x + t_1 z_1) 
\right)_{t_1 = 0}
t
+
\left(
\frac{\partial}{ \partial t_2 } 
f(x + t_2 z_2) 
\right)_{t_2 = 0}
t
+
r'
\\
& =
f(x)
+
t \psi_x(z_1 )
+
t
\psi_x(z_2)
+r',
\end{align*}
where $r' = \cO(  t^2 )$ (for fixed $x , z_1 , z_2$). 

Hence by unicity of the first order expansion, we have
$ \psi_x(z_1  + z_2) = \psi_x (z_1) + \psi_x(z_2) $. Hence $\psi_x$ is linear.
To conclude the proof, it can be shown simply that if, for a fixed $x$, two linear continuous functions $\psi_x$ and $\psi'_x$ satisfy \eqref{eq:Hadamard:deux} for all $y \in \HH$, then they coincinde.
\end{proof}

\begin{lemma} \label{lemma:convex:gradient}
   Let $\ell_1 , \ldots , \ell_n$, $\tell_1 , \ldots , \tell_n$ be linear functions on $\HK$, recall $Y_1 , \ldots , Y_n \in \mathbb{R}$ and let $ \lambda >0 $. Let 
    \[
f = \Argmin{h \in \HK} 
\frac{1}{n} 
\sum_{i=1}^n \left( \ell_i(h) - Y_i \right)^2
+ \lambda \|h\|_{\HK}^2
    \]
    and 
    \[
g = \Argmin{h \in \HK} 
\frac{1}{n} 
\sum_{i=1}^n \left( \tell_i(h) - Y_i \right)^2
+ \lambda \|h\|_{\HK}^2.
    \]
Then 
\begin{align*}
    \|f -g \|_{\HK}^2 
    & \leq 
    \frac{1}{\lambda} 
    \left[ 
    \frac{1}{n} 
    \sum_{i=1}^n 
    \left\{ 
\tell_i(f-g) \tell_i(f) 
-
\ell_i(f-g) \ell_i(f) 
    \right\} 
    +
        \frac{1}{n} 
    \sum_{i=1}^n 
    Y_i
    \left\{ 
\ell_i(f-g) 
-
\tell_i(f-g) 
    \right\}
    \right].
    \end{align*}
\end{lemma}

\begin{proof}
    Let, for $t \geq 0$,
    \begin{align*}
    \tilde{R}(t)  = &
    \frac{1}{n} \sum_{i=1}^n 
    \left\{ 
\tell_i( g + t(f-g) ) - Y_i 
    \right\}^2 
    + \lambda \| g + t(f-g) \|^2_{\HK}  
    \\
     = &
    t^2 \left[ 
     \frac{1}{n} \sum_{i=1}^n 
     \tell_i^2(f-g) 
     + \lambda \|f -g\|_{\HK}^2
    \right]
    +
     t \left[ 
     \frac{1}{n} \sum_{i=1}^n 
     2 \tell_i(f-g) 
     \left\{ 
\tell_i(g) - Y_i
     \right\}   
     + 2  \lambda \langle g , f -g \rangle_{\HK}
    \right]
    \\
&    + 
     \frac{1}{n} \sum_{i=1}^n  
     \left( 
\tell_i(g) - Y_i
     \right)^2 
     + \lambda \|g\|_{\HK}^2.
    \end{align*}
Then by strong convexity, $\tilde{R}'(0) = 0$ and
\[
\tilde{R}'(1) 
\geq 
2 \lambda \|f - g\|_{\HK}^2.
\]

Let similarly 
    \[
 R(t)  = 
    \frac{1}{n} \sum_{i=1}^n 
    \left\{ 
\ell_i( g + t(f-g) ) - Y_i 
    \right\}^2 
    + \lambda \| g + t(f-g) \|^2_{\HK}.  
    \] 

Then $R'(1) = 0$. Hence 
\begin{align*}
2 \lambda \|f - g\|_{\HK}^2 
\leq &
\left( 
\tilde{R}'(1)  - R'(1)
\right) 
\\
= &
\frac{2}{n} 
\sum_{i=1}^n  
\left\{ 
\tell_i^2 (f - g) 
- 
\ell_i^2 (f - g) 
\right\} 
+
\frac{2}{n} 
\sum_{i=1}^n  
\left\{ 
\tell_i (f - g) 
\tell_i (g) 
- 
\ell_i (f - g) 
\ell_i (g) 
\right\} 
\\
& + 
\frac{2}{n} 
\sum_{i=1}^n  
\left\{ 
\ell_i (f - g) 
Y_i
- 
\tell_i (f - g) 
Y_i
\right\} 
\\
= &
 \frac{2}{n} 
    \sum_{i=1}^n 
    \left\{ 
\tell_i(f-g) \tell_i(f) 
-
\ell_i(f-g) \ell_i(f) 
    \right\} 
    +
        \frac{2}{n} 
    \sum_{i=1}^n 
    Y_i
    \left\{ 
\ell_i(f-g) 
-
\tell_i(f-g) 
    \right\}.
\end{align*}
This concludes the proof.
\end{proof}

\section{
Proofs for Section \ref{subsection:sharper:rates:different:norms}
}
\label{appendix:sharper:rates:different:norms}

\subsection{Preliminary lemma}

The proof of Theorem \ref{bound:Cinfty} relies on Lemma \ref{conditioCOuntinousTCL} below. 
For a (linear) operator $A$ on $\HEK$, its operator norm is written $\|A\|_{OP(\HEK,\HEK)}$ and defined as
$\|A\|_{OP(\HEK,\HEK)}=\sup_{\|f\|_{\HEK}\leq 1}\| A f\|_{\HEK} $. 
We say that an operator $A$ is bounded if $\|A\|_{OP(\HEK,\HEK)} < \infty$. 
We say that a sequence of bounded operators $(A_n)_{n \in \N}$ converges to an operator $A$ in operator norm if $\|A_n - A\|_{OP(\HEK,\HEK)}$ goes to zero as $n\to \infty$. 
Finally, for any $u \in \cE $, we let $K_{\cE,u} \in \HEK$ be defined by $K_{\cE,u}(v) = K(u,v)$ for $v \in \cE$.

\begin{lemma}\label{conditioCOuntinousTCL}
    The sequence of operators $\Theta_{n}:\HEK\to \HEK$, defined as
\[
 \Theta_{n}(\phi) = 
    \frac{1}{n}
    \sum_{i=1}^n
    2 \phi(x_i) 
    K_{\cE,x_i}
\]    
converges almost surely as $n \to \infty$ in operator norm 
 to a bounded injective operator $\Theta$. 
\end{lemma}

\begin{proof}
For $x \in \cE$, the operator $\Theta _x: \HEK \to \HEK $ defined by $\Theta_x \phi =  \phi(x) K_{
\cE,x}$ is easily seen to be self-adjoint and non-negative. It is also trace class: for an orthonormal basis $(e_k)_{k \in \N}$ of $\HEK$ we have
\begin{align*}
    \sum_{k=1}^{\infty}
    \langle \Theta _x e_k , e_k \rangle_{\HEK} 
    = & 
    \sum_{k=1}^{\infty} e_k(x) \langle e_k ,K_{
\cE,x}\rangle_{\HEK} \\
= & \sum_{k=1}^{\infty} 
\left( 
\langle e_k ,K_{
\cE,x}
\rangle_{\HEK}
\right)^2 
\\ 
= & \| K_{
\cE,x} \|^2_{\HEK}
\\ 
= & 1,
\end{align*}
from Parseval's identity. 

The linear space of trace class operators on the separable Hilbert space $\HEK$ is a Banach space with the trace norm
\citep[Cor. 4.2.2]{murphy2014c}.
Hence applying the strong law of large numbers on Banach spaces, see for instance \citet[Cor. 7.10]{ledoux1991probability}, we obtain that $\Theta_n$ converges almost surely in trace norm, and thus in operator norm \citep[Cor. 3.4.4]{pedersen2012analysis}, to the operator $\Theta$ defined by 
\begin{equation} \label{eq:Theta}
\Theta \phi
=
\int_{\cE} 
2
\phi(x)
K_{\cE,x} 
\dd \cL(x).
\end{equation}
This operator is injective because if $\Theta \phi = 0 \in \HEK $ then
\begin{align*}
    0 = &
    \langle 
    \Theta \phi , \phi 
    \rangle_{\HEK}
    \\ 
    = &
2    \int_{\cE} 
\phi(x) 
  \langle 
K_{\cE,x} 
,
\phi  
\rangle_{\HEK}
\dd \cL(x) 
\\ 
= &
 2   \int_{\cE} 
\phi(x)^2 
\dd \cL(x) 
\end{align*}
and thus $\phi$ is $\cL$-almost surely zero on $\cE$. Since $\cE$ is the probabilistic support of $\cL$ and since $\phi$ is continuous on $\cE$, then $\phi$ is identically zero on $\cE$. This concludes the proof. 
\end{proof}

\subsection{Proofs of Theorem \ref{bound:Cinfty}}

We consider the Banach space $\cC(\cE)$ of the continuous functions,
from the compact space $\cE$ to $\R$, endowed with the norm $\| \cdot \|_{\cE,\infty}$.
We say that a sequence $(X_{n,N})$ of random elements of $\cC(\cE)$ is tight if for any $\epsilon>0$, there exists  a compact set $A$ such that $\mathbb{P}( X_{n,N} \in A) \geq 1-\epsilon$, for all $n,N$.

Then, Prohorov's theorem \citep[Thm. 1.3.9]{vandervaart2013weak} states that any tight sequence of probability measures is relatively compact for the weak convergence, 
that is, every subsequence  has a further subsequence that converges to a
tight probability measure.
In our space, $\cC(\cE)$, the following condition implies tightness: for any $\tau,\mu>0$, there exists $\delta>0$ such that 
\begin{equation}
    \label{condition}
    \limsup_{
    \substack{
    n \to \infty 
    \\ 
    N \to \infty}
  }  
    \mathbb{P}\left(\sup_{
    \substack{
    x,x' \in \cE \\
    \|x-x'\|_{\HH}<\delta}
    }
    |X_{n,N}(x)-X_{n,N}(x')|>\mu \right)<\tau.
\end{equation}
This claim is direct consequence of \cite[Theorem 1.5.6]{vandervaart2013weak}\footnote{This reference states this condition  for the (strictly bigger) space of real-valued bounded functions on $\cE$, denoted by $\ell^{\infty}(\cE)$, and in terms of a finite partition of $\cE$.  Since the open sets for the form $O_x = 
\{x'\in \cE: \ \|x-x'\|_{\HH}<\delta\}$, $x \in \cE$,
conform a $\delta$-covering of the whole compact set $\cE$, there exists a finite $\delta$-sub-covering. As a consequence, \eqref{condition} implies tightness in  $\ell^{\infty}(\cE)$, so in $\cC(\cE)$.}.
The boundeness in probability of the norm 
$\|a_{n,N} (\hat{f}_{n,N}-\hat{f}_n)\|_{\HEK}$ yields, as a consequence, the tightness of  $a_{n,N} (\hat{f}_{n,N}-\hat{f}_n)$ in  $\cC(\cE)$.  
\begin{lemma}\label{Lemma:tight}
    The sequence $ (a_{n,N} (\hat{f}_{n,N}-\hat{f}_n))_{n,N}$ is tight in  $\cC(\cE)$, i.e., \eqref{condition} holds. 
\end{lemma}
\begin{proof}
    For the sake of readability, we denote 
    \[
    \omega_{n,N}(\delta)=\sup_{ \substack{     x,x' \in \cE \\     \|x-x'\|_{\HH}<\delta}} |a_{n,N} (\hat{f}_{n,N}-\hat{f}_n)(x)-a_{n,N} (\hat{f}_{n,N}-\hat{f}_n)(x')|.
    \]
    The reproducing property yields
    \[
    \omega_{n,N}(\delta)=\sup_{ \substack{     x,x' \in \cE \\     \|x-x'\|_{\HH}<\delta}} |\langle a_{n,N} (\hat{f}_{n,N}-\hat{f}_n), K_{\cE,x}-K_{\cE,x'}\rangle_{\HEK}|,
    \]
 which is easier to bound: 
 \begin{align*}
      \omega_{n,N}(\delta)
& \leq \sup_{ \substack{     x,x' \in \cE \\     \|x-x'\|_{\HH}<\delta}} \|  a_{n,N} (\hat{f}_{n,N}-\hat{f}_n)\|_{\HEK} \| K_{\cE,x}-K_{\cE,x'}\|_{\HEK}\\
    &  =\sup_{ \substack{     x,x' \in \cE \\     \|x-x'\|_{\HH}<\delta}} \|  a_{n,N} (\hat{f}_{n,N}-\hat{f}_n)\|_{\HEK} \sqrt{2-2K(x,x')}.
 \end{align*}

Via Lemma \ref{lemma:holder:square:exponential}, we have, 
for some constant $c_1$, 
$$ 
      \omega_{n,N}(\delta)\leq c_1 \delta \|a_{n,N} (\hat{f}_{n,N}-\hat{f}_n)\|_{\HEK},
$$
where the assumption $\|a_{n,N} (\hat{f}_{n,N}-\hat{f}_n)\|_{\HEK}=\cO_{\bP}(1)$ concludes the proof. 
\end{proof}

Therefore, the  sequence  $a_{n,N} (\hat{f}_{n,N}-\hat{f}_n)$ is tight, so we need to find the possible limits of its subsequences. 
To do so, we compute the gradient of $R_n:\HEK\to \R$ given in \eqref{eq:Rn}, at $f \in \HEK$. This gradient is denoted $R_n'(f)$ and defined as
\begin{equation*} 
    R_n'(f) 
    =
    \frac{1}{n}
    \sum_{i=1}^n
    2\left( 
f(x_i) - Y_i  
    \right)
 K_{\cE,x_i}
    +
2 \lambda f.
\end{equation*}

As $\hat{f}_n$ is the unique maximizer of $R_n$, we can check that $R_n'(\hat{f}_n)=0$. 
We can also simply check that, for all $ f , \psi \in \HEK$,
\[
R_n'(f+\psi) - R_n'(f) = \Theta_n \psi +
   2 \lambda \psi.
\]

As a consequence, taking $f=\hat{f}_n$ and $\psi=(\hat{f}_{n,N}-\hat{f}_n)$ we obtain
\begin{equation} \label{eq:proof:sharper:trois}
R_n'(\hat{f}_{n,N}) - R_n'(\hat{f}_n)
=
\Theta_n (  \hat{f}_{n,N}-\hat{f}_n )
+ 2 \lambda (  \hat{f}_{n,N}-\hat{f}_n ).
\end{equation}
 Let us define $R_{n,N}'$ from the expression of $R_{n,N}$ in \eqref{eq:Rn:empirical} similarly as $R_{n}'$, with 
\begin{equation*} 
    R_{n,N}'(f) 
    =
    \frac{1}{n}
    \sum_{i=1}^n
    2\left( 
f(x_{N,i}) - Y_i 
    \right)
 K_{\cE,x_{N,i}}
    +
2 \lambda f.
\end{equation*}
Since $R_{n,N}'(\hat{f}_{n,N})=R_n'(\hat{f}_n)=0$, replacing $R_n'(\hat{f}_n)$ by $R_{n,N}'(\hat{f}_{n,N}) $  in \eqref{eq:proof:sharper:trois}, we obtain 
\begin{equation} \label{eq:proof:sharper:trois.1}
 R_n'(\hat{f}_{n,N})
 -
 R_{n,N}'(\hat{f}_{n,N})
=
\Theta_n (  \hat{f}_{n,N}-\hat{f}_n )
+ 2 \lambda (  \hat{f}_{n,N}-\hat{f}_n ).
\end{equation}
Due to Lemma~\ref{conditioCOuntinousTCL},
$$ \|\Theta_n-\Theta\|_{OP(\HEK,\HEK)}\|\hat{f}_{n,N}-\hat{f}_n\|_{\HEK}=o_{\bP}(\|\hat{f}_{n,N}-\hat{f}_n\|_{\HEK}),$$
so that, because $\lambda \to 0$,
$$a_{n,N}
(
R_n'(\hat{f}_{n,N})
-
R_{n,N}'(\hat{f}_{n,N})
)
= \Theta (a_{n,N}(\hat{f}_{n,N}-\hat{f}_n))+o_{\bP}(a_{n,N}\|\hat{f}_{n,N}-\hat{f}_n \|_{\HEK}).$$
Here and in the sequel, for two sequences $(g_{n,N})$ and $(m_{n,N})$, with $g_{n,N} \in \HEK$ and $m_{n,N} \geq 0$, we write $g_{n,N} = o_{\bP}(m_{n,N})$ when $\| g_{n,N} \|_{\HEK} = o_{\bP}(m_{n,N})$ as $n,N \to \infty$. 
Then, by assumption, 
\begin{equation} \label{eq:in:proof:better:rate}
    a_{n,N}(
    R_n'(\hat{f}_{n,N})
    -
    R_{n,N}'(\hat{f}_{n,N})
    )
    = \Theta (a_{n,N}(\hat{f}_{n,N}-\hat{f}_n))+o_{\bP}(1).
\end{equation}
Then, for all fixed $\psi\in \HEK$, letting $D_{n,N} = \hat{f}_{n,N} - \hat{f}_{n}$,
\begin{align*}
   &
   \left \langle R_{n}'(\hat{f}_{n,N})
   -
    R_{n,N}'(\hat{f}_{n,N}) , \psi
    \right \rangle_{\HEK}
    \\
    = &
    \frac{2}{n}
    \sum_{i=1}^n
    \left(Y_i- \hat{f}_{n,N}(x_{N,i})\right)\psi(x_{N,i})
    -
    \frac{2}{n}
    \sum_{i=1}^n
    \left( 
Y_i - \hat{f}_{n,N}(x_{i})
    \right)\psi(x_{i})\\
    = &
    \frac{2}{n}
    \sum_{i=1}^n
Y_i \left( \psi( x_{N,i}) - \psi(x_i) \right)
    +
    \frac{2}{n}
    \sum_{i=1}^n
    \hat{f}_{n,N}(x_{i})
    \psi(x_i)
    -
    \hat{f}_{n,N}(x_{N,i})
    \psi(x_{N,i})
\\
= &
    \frac{2}{n}
    \sum_{i=1}^n
Y_i \left( \psi( x_{N,i}) - \psi(x_i) \right)
    +
    \frac{2}{n}
    \sum_{i=1}^n
    (\hat{f}_{n,N}(x_{i}) - \hat{f}_{n,N}(x_{{N,i}}))
    \psi(x_i)
    \\
& +
    \frac{2}{n}
    \sum_{i=1}^n
    \hat{f}_{n,N}(x_{N,i})
    \left(
    \psi(x_i) - \psi(x_{N,i})
    \right)
\\
= &
\underbrace{
    \frac{2}{n}
    \sum_{i=1}^n
Y_i \left( \psi( x_{N,i}) - \psi(x_i) \right)
 }_{ = T_{1,n,N}(\psi)}
    +
    \underbrace{
    \frac{2}{n}
    \sum_{i=1}^n
    (\hat{f}_{n}(x_{i}) - \hat{f}_{n}(x_{{N,i}}))
    \psi(x_i)
     }_{ = T_{2,n,N}(\psi)}
    \\
&     +
\underbrace{
    \frac{2}{n}
    \sum_{i=1}^n
    (D_{n,N}(x_{i}) - D_{n,N}(x_{{N,i}}))
    \psi(x_i)
     }_{ = T_{3,n,N}(\psi)}
     +
\underbrace{
    \frac{2}{n}
    \sum_{i=1}^n
    \hat{f}_{n}(x_{i}) 
    \left(
    \psi(x_i) - \psi(x_{N,i})
    \right)
     }_{ = T_{4,n,N}(\psi)}
    \\
& +
\underbrace{
    \frac{2}{n}
    \sum_{i=1}^n
\left( 
\hat{f}_{n}(x_{N,i})
-
\hat{f}_{n}(x_{i})
\right)
 \left(
    \psi(x_i) - \psi(x_{N,i})
    \right)
 }_{ = T_{5,n,N}(\psi)}
 +
    \underbrace{
    \frac{2}{n}
    \sum_{i=1}^n
    D_{n,N}(x_{N,i})
    \left(
    \psi(x_i) - \psi(x_{N,i})
    \right).
     }_{ = T_{6,n,N}(\psi)}
\end{align*}
Using similar arguments as in Section \ref{section:proof:main:error:bound}, we can show
\[
\bE \left[ \left| T_{1,n,N}(\psi) \right| \right] \leq  \bE[  Y_{\max,n}  ] 
\| \psi \|_{\HEK}
\cO
\left( 
\frac{1}{N} +  \frac{1}{\sqrt{nN}}  
\right), 
\]
\[
\bE \left[ \left| T_{2,n,N}(\psi) \right| \right] \leq  \bE[ c_n ]  
\| \psi \|_{\HEK}
\cO 
\left( 
\frac{1}{N} 
+  \frac{1}{\sqrt{nN}}  
\right),
\] 
\[
| T_{3,n,N}(\psi) | \leq \| D_{n,N} \|_{\HEK} \| \psi \|_{\HEK} \cO_{\bP} \left( \frac{1}{\sqrt{N}} \right), 
\]
\[
\bE \left[ \left| T_{4,n,N}(\psi) \right| \right] \leq  \bE[ c_n ]  
\| \psi \|_{\HEK}
\cO
\left( 
\frac{1}{N} 
+ 
\frac{1}{ \sqrt{nN} }
\right),
\]
\[
\bE \left[ \left| T_{5,n,N}(\psi) \right| \right] 
\leq 
\bE[ c_n ] 
\| \psi \|_{\HEK}
\cO \left( \frac{1}{N} \right),
\] 
and
\[
| | T_{6,n,N}(\psi) | \leq \| D_{n,N} \|_{\HEK} \| \psi \|_{\HEK} 
\cO_{\bP} \left( \frac{1}{\sqrt{N}} \right).
\]
By the assumptions on $a_{n,N}$, we thus have $a_{n,N}  T_{\ell,n,N}(\psi) = o_{\bP}(1)$, for $\ell=1,\ldots,6$.
Hence eventually we have the rate $ 
\langle a_{n,N}
R_{n}'(\hat{f}_{n,N}) -
a_{n,N} R_{n,N}'(\hat{f}_{n,N})
,
\psi
\rangle_{\HEK}
= o_{\bP}(1)$, for all fixed $\psi \in \HEK$.
Hence, from \eqref{eq:in:proof:better:rate}, we have 
$\langle \Theta (a_{n,N}(\hat{f}_{n,N}-\hat{f}_n)), \psi \rangle_{\HEK} =o_{\bP}(1)$, for any fixed $ \psi\in \HEK$.


Via Lemma~\ref{Lemma:tight}, we know that $a_{n,N}(\hat{f}_{n,N}-\hat{f}_n)$ is tight in $\cC(\cE)$. Therefore, for every subsequence,
there exists a random variable $X\in \cC(\cE)$ such that, along a further subsequence, for any bounded continuous function $g:\cC(\cE)\to \R $,
\begin{equation} \label{eq:along:subsequence}
    \bE [g(a_{n,N}(\hat{f}_{n,N}-\hat{f}_n))]\to \bE[g(X)]. 
    \end{equation} 

Note that, in Lemma \ref{conditioCOuntinousTCL}, $\Theta$ is defined from $\HEK$ to $\HEK$. However, the expression for $\Theta$ in \eqref{eq:Theta} also defines an operator $\Theta_{\cC}$ from $\cC(\cE)$ to $\HEK$ that coincides with $\Theta$ on $\HEK$. 

Let us show that if a sequence $(g_\ell)_{\ell \in \N}$, with $g_\ell \in \cC(\cE)$, satisfies $\|g_{\ell}\|_{\infty}\to 0$, then 
$$ \|\Theta_{\cC} (g_{\ell})\|_{\infty}\to 0.$$ 
To do so, since 
$$ (\Theta_{\cC} g_{\ell})(x)
=
\int_{\cE} 
2
g_{\ell}(y)
K(x,y) 
\dd \cL(y)$$
and $\| g_{\ell}\|_{\infty}\to 0$, it holds
$$ \vert (\Theta_{\cC} g_{\ell})(x) \vert 
\leq 2\| g_{\ell}\|_{\infty}
\int_{\cE} 
K(x,y) 
\dd \cL(y)\leq 2\| g_{\ell}\|_{\infty}.$$
Therefore, $ \|\Theta_{\cC} (g_{\ell})\|_{\infty}\to 0$, which means that $\Theta_{\cC} $ is a continuous operator on $ \cC(\cE)$, which remains the limit of $\Theta_n$ on $\HH_{\cE,K} \subset \cC(\cE)$. 
From \eqref{eq:along:subsequence},
for any bounded continuous $g:\cC(\cE)\to \R $,
it holds
\begin{equation} \label{eq:proof:sharper:two}
\bE [g(\Theta_{\cC}(a_{n,N}(\hat{f}_{n,N}-\hat{f}_n)))]\to \bE[g(\Theta_{\cC} X)]. 
\end{equation}
As seen before, we have 
$\langle \Theta_{\cC} (a_{n,N}(\hat{f}_{n,N}-\hat{f}_n)), \psi\rangle_{\HH_{\cE,K}}=o_{\bP}(1)$, for any $ \psi\in \HEK$, so that (set $\psi=K_{\cE,x}$) $(\Theta_{\cC} (a_{n,N}(\hat{f}_{n,N}-\hat{f}_n)))(x)=o_{\bP}(1)$. 
Therefore, 
$$ \bE [h(\{\Theta_{\cC}(a_{n,N}(\hat{f}_{n,N}-\hat{f}_n))\}(x))]\to \bE [h(0)], $$
for any bounded continuous 
 $h:\R \to \R $ and $x\in \cE$. 
  For any bounded continuous 
 $h:\R \to \R $ and $x\in \cE$, since $ \psi \mapsto h(\psi(x))$
 is continuous on $\cC(\cE)$, 
 it holds from \eqref{eq:proof:sharper:two} that
 $$ \bE[h((\Theta_{\cC} X )(x))]= \bE[h(0)] =  h(0).$$
This implies $(\Theta_{\cC} X )(x)=0$ almost surely, for all $x\in \cE$.
Therefore, since $\Theta_{\cC} X$ is continuous, $\Theta_{\cC} X=0$ almost surely.
Hence, almost surely,
%
\begin{align*}
0 = & 
\|  \Theta_{\cC} X \|_{\HEK}^2
\\
= &
4 \int_{\cE}
\int_{\cE}
X(u) X(v) 
K( u,v )
\dd \cL(u)
\dd \cL(v).
\end{align*}
Since $K$ takes strictly positive values, $(u,v) \mapsto X(u) X(v)$ is zero $\cL$-almost everywhere on the compact set $\cE$, almost surely. By continuity, $X$ is the zero function on $\cE$, almost surely.

Hence, we have proved that $a_{n,N}(\hat{f}_{n,N}-\hat{f}_n)$ is tight in $\cC(\cE)$ and any limit in distribution along subsequences is the degenerate random variable $0$. 
Therefore, $ a_{n,N}(\hat{f}_{n,N}-\hat{f}_n)$ tends to $0$ in probability, which means that $\|  a_{n,N}(\hat{f}_{n,N}-\hat{f}_n)\|_{\cE,\infty}=o_{\bP}(1)$. This concludes the proof.

\section{Proofs for Section \ref{subsection:convergence:rates:hilbert}} 

\subsection{Proofs of Theorem \ref{theorem:general:minimax}}

We have
\begin{align*}
\sqrt{
\int_{\HH}
\left( 
f^{\star}(x)
-
\hat{f}_{n,N}(x)
\right)^2
\dd \mathcal{L}(x)
}
\leq &
\sqrt{
\int_{\HH}
\left( 
f^{\star}(x)
-
\hat{f}_{n}(x)
\right)^2
\dd \mathcal{L}(x)
} 
\\
 & +
\sqrt{
\int_{\HH}
\left( 
\hat{f}_n(x)
-
\hat{f}_{n,N}(x)
\right)^2
\dd \mathcal{L}(x)
}.
\end{align*}

\citet[Thm. 1]{caponnetto2007optimal} shows that 
\[
\sqrt{
\int_{\HH}
\left( 
f^{\star}(x)
-
\hat{f}_{n}(x)
\right)^2
\dd \mathcal{L}(x)
}
=
\cO _{\bP} 
\left( 
n^{-\frac{bc}{2(bc+1)}}
\right).
\]
Note that with the setting of Theorem \ref{theorem:general:minimax}, it is straigthforward to check Hypotheses 1 and 2 in \cite{caponnetto2007optimal}. 
Hence, it just remains to show that 
\begin{equation*} 
E_{n,N}
=
\sqrt{
\int_{\HH}
\left( 
\hat{f}_n(x)
-
\hat{f}_{n,N}(x)
\right)^2
\dd \mathcal{L}(x)
}
=
\cO _{\bP} 
\left( 
n^{-\frac{bc}{2(bc+1)}}
\right).
\end{equation*}

Now, before applying Theorem \ref{theorem:error:bound}, we show that $c_n = \| \hat{f}_n \|_{\HK} $ there is bounded in probability. The developments in \citet[Sect. 9.1]{szabo2016learning} (using \cite{caponnetto2007optimal}), although written for a specific Hilbert space $\HH$ based on mean embeddings of distributions, can actually be seen to hold for a general $\HH$ as in the context of Theorem \ref{theorem:general:minimax}. These developments yield
\begin{align*}
c_n^2
=
\cO_{\bP}
\left( 
\frac{1}{\lambda^2 n^2}
+
\frac{\cN(\lambda)}{n \lambda}
+
\frac{\cB(\lambda)}{\lambda^2 n^2}
+
\frac{\cA(\lambda)}{\lambda^2 n}
+
\cB(\lambda)
+
1
\right).
\end{align*}
Above, the quantities $\cN(\lambda)$, $\cB(\lambda)$ and $\cA(\lambda)$ are defined in \citet[Sect. 5.2]{caponnetto2007optimal} and it is shown
in Proposition 3 there
that $\cN(\lambda) = \cO( \lambda^{-1/b} )$, $\cB(\lambda) = \cO( \lambda^{c-1} )$ and $\cA(\lambda) = \cO( \lambda^{c} )$. Hence, it is simple to check that $c_n = \cO_{\bP}(1)$.

Then, Theorem \ref{theorem:error:bound} yields, with $\cB_n$ the $\sigma$-algebra generated by $(\mu_i,Y_i)_{i=1}^n$, and with a constant $c_1$,
\begin{align*}
\bE [E_{n,N} | \cB_n]
& \leq
    \frac{c_1 (Y_{\max} + c_n)  }{ \lambda N} 
+
\frac{c_1 (Y_{\max} + c_n)}{\lambda \sqrt{n} \sqrt{N}} 
\\ 
& + 
\left(
1
+
\frac{\sqrt{N} }{\sqrt{n} } 
\right)^{-1}
\left(
 \frac{c_1  (Y_{\max} + c_n)}{\lambda n }
+
\frac{c_1  (Y_{\max} + c_n )}{\lambda^2 n \sqrt{N}}
\right).
\end{align*}

With Markov inequality, applied conditionally to $\cB_n$ this implies that 

\begin{align*}
E_{n,N}
& = \cO_{\bP}
\left(
    \frac{ Y_{\max} + c_n }{ \lambda N} 
+
\frac{ Y_{\max} + c_n}{\lambda \sqrt{n} \sqrt{N}} 
\right)
\\ 
& + 
\left(
1
+
\frac{\sqrt{N} }{\sqrt{n} } 
\right)^{-1}
\cO_{\bP}
\left(
 \frac{  Y_{\max} + c_n}{\lambda n }
+
\frac{  Y_{\max} + c_n  }{\lambda^2 n \sqrt{N}}
\right).
\end{align*}

Let us now use that $c_n = \cO_{\bP}(1)$ and that $\lambda n^{\frac{b}{bc+1}}$ is lower and upper bounded.
Furthermore, let $a$ be as in the theorem statement, so that
$N / n^{ a} $ is lower-bounded. This yields

\begin{align} \label{eq:bound:EnN:mean:emb}
E_{n,N}
& = \cO_{\bP}
\left(
n^{\frac{b}{bc+1} - a}
+
n^{\frac{b}{bc+1} - \frac{1}{2} - \frac{a}{2}}
\right)
\\ 
& + 
\min(
1
,
n^{\frac{1}{2} - \frac{a}{2}}
)
\cO_{\bP}
\left(
n^{\frac{b}{bc+1} - 1}
+
n^{\frac{2b}{bc+1} - 1 - \frac{a}{2}}
\right). \notag
\end{align}

We now study whether the bound in \eqref{eq:bound:EnN:mean:emb} can be of order $ \mathcal{O}_{\bP}
\left( 
n^{-\frac{bc}{2(bc+1)}}
\right)$ with $a \leq 1$. Necessary and sufficient conditions for this are
\begin{align*}
& \frac{b}{bc+1} - a
\leq 
-\frac{bc}{2(bc+1)},
~ ~ ~
\frac{b}{bc+1} - \frac{1}{2} - \frac{a}{2}
\leq 
-\frac{bc}{2(bc+1)}, \\
& 
\frac{b}{bc+1} - 1
\leq 
-\frac{bc}{2(bc+1)},
~ ~ ~
\frac{2b}{bc+1} - 1 - \frac{a}{2}
\leq 
-\frac{bc}{2(bc+1)}.
\end{align*}

Using that we aim for $a \leq 1$, so that the third condition is implied by the first one, 
the conditions are equivalent to

\begin{align*}
& a \geq 
\frac{b + \frac{bc}{2}}{bc+1},
~ ~ ~
a \geq 
\frac{2b -1}{bc+1}, 
~ ~ ~
a \geq 
\frac{
4b - bc - 2
}{
bc+1
}.
\end{align*}

The three lower bounds on $a$ above are smaller or equal to $1$ if and only if $b(1 - \frac{c}{2}) \leq \frac{3}{4}$. Hence, in this case $a = \max(
\frac{b + \frac{bc}{2}}{bc+1}
,
\frac{2b -1}{bc+1}
,
\frac{
4b - bc - 2
}{
bc+1
}
) \leq 1$ indeed yields 
$E_{n,N} = \mathcal{O}_{\bP}
\left( 
n^{-\frac{bc}{2(bc+1)}}
\right)$ with $N / n^a$ lower-bounded.

We now consider the case $b(1 - \frac{c}{2}) > \frac{3}{4}$, and we now study whether the bound in \eqref{eq:bound:EnN:mean:emb} can be of order $ \mathcal{O}_{\bP}
\left( 
n^{-\frac{bc}{2(bc+1)}}
\right)$ with $a > 1$. Necessary and sufficient conditions for this are
\begin{align*}
& \frac{b}{bc+1} - a
\leq 
-\frac{bc}{2(bc+1)},
~ ~ ~
\frac{b}{bc+1} - \frac{1}{2} - \frac{a}{2}
\leq 
-\frac{bc}{2(bc+1)}, \\
& 
\frac{1}{2} - \frac{a}{2}
+
\frac{b}{bc+1} - 1
\leq 
-\frac{bc}{2(bc+1)},
~ ~ ~
\frac{1}{2} - \frac{a}{2}
+
\frac{2b}{bc+1} - 1 - \frac{a}{2}
\leq 
-\frac{bc}{2(bc+1)}.
\end{align*}

These conditions are equivalent to 
\begin{align*}
& a \geq 
\frac{b + \frac{bc}{2}}{bc+1},
~ ~ ~
a \geq 
\frac{2b -\frac{1}{2}}{bc+1}. 
~ ~ ~
\end{align*}

Hence, when $b(1 - \frac{c}{2}) > \frac{3}{4}$, 
taking $a = \max( \frac{b + \frac{bc}{2}}{bc+1} , \frac{2b -\frac{1}{2}}{bc+1} )$,
we indeed have $E_{n,N} = \mathcal{O}_{\bP}
\left( 
n^{-\frac{bc}{2(bc+1)}}
\right)$ with $N / n^a$ lower-bounded. This concludes the proof.

\section{Proofs for Section \ref{subsection:application:sinkhorn} }

\subsection{Proof of Lemma \ref{lemma:near:unbias}}
\label{supplement:subsection:proof:sinkhorn}

For $p>0$, we let $\cC^p(\Omega)$ be the space of functions $f: \Omega\to \mathbb{R}$ that are $\lfloor p \rfloor$ times differentiable, with $\lfloor . \rfloor$ the integer part and with $\|f \|_{\cC^p(\Omega)} < \infty$, where
\begin{equation*}
    \|f \|_{\cC^p(\Omega)} = \sum_{\beta=0}^{\lfloor p \rfloor}\sum_{|\alpha|= \beta}\|D^{\alpha} f \|_{\infty}.
    \end{equation*}
    Above $\alpha = (\alpha_1, \ldots , \alpha_d) \in \mathbb{N}^d$ with $\sum_{\ell=1}^d \alpha_\ell = \beta$ and $D^{\alpha} = \partial^\beta / \partial _{x_1}^{\alpha_1} \cdots \partial _{x_d}^{\alpha_d}$. The space $\cC^p(\Omega)$ is endowed with the norm $\|\cdot \|_{\cC^p(\Omega)}$.

A distance between two measures $\mu , \nu \in \cPO$ can thus be defined as
\begin{equation*} 
    \| \mu - \nu \|_p =
    \sup_{f\in \cC^{p}(\Omega),\ \| f\|_{\cC^p(\Omega)}\leq 1}
    \int f(x) (\dd \mu(x) - \dd \nu(x)).
\end{equation*}

Fix $p > d$. It can be seen from the proof of \citet[Thm. 2.1]{gonzalez2022weak} 
and from \citet[Prop. 3.1]{carlierdifferential2020}
that we have, with a constant $C_{\Omega}$,
\[
||g^{\mu^N}  - g^{\mu}
- \mathcal{B}\mathcal{A}(\mu^N-\mu) ||_{\mathcal{L}^2(\mathcal{U})}\leq C_{\Omega} \| \mu^N-\mu  \|_p^2,
\]
with $\mathcal{B}: \mathcal{L}^2(\mathcal{U})\to \mathcal{L}^2(\mathcal{U})$ being a bounded linear operator---with unimportant shape for us---and $\mathcal{A}$ defined by, for $\eta_1, \eta_2 \in \cPO$, $u \in \Omega$,
\[
(\mathcal{A}(\eta_1 - \eta_2))(u) 
=
\int_{\Omega}
e^{\| u - v \|^2}
\dd (\eta_1 - \eta_2)(v).
\]
We call $a_N=\mathcal{B}\mathcal{A}(\mu^N-\mu)$ and $$b_N= g^{\mu^N} - g^{\mu} -\mathcal{B}\mathcal{A}(\mu^N-\mu).$$
As a consequence, 
$ g^{\mu^N} - g^{\mu}
=a_N+b_N$. 

First, the proof of \citet[Eq. 4.13]{Barrio2022AnIC} yields $\bE[ \| \mu^N-\mu  \|_p^{2s} ]\leq  \frac{C_{\Omega}}{N^s}$ (up to increasing the value of $C_{\Omega}$).
From this, we obtain 
$ \bE [ \|b_N\|_{\cL^2(\cU)}^s ] \leq  C_{\Omega}^s \bE [ \| \mu^N-\mu  \|_p^{2s} ] \leq  \frac{C_{\Omega}^{1+s}}{N^{s}}$.
Then, 
with a constant $c_1$, and letting $\| \cB \| = \sup_{\| h \|_{\cL^2(\cU)} \leq 1} \| \cB h \|_{\cL^2(\cU)}$ (the operator norm), we have
$ \bE [ \|a_N\|_{\cL^2(\cU)}^s ] \leq 
\|
\cB
\|^s
\bE [ 
\| \cA (\mu^N - \mu) \|_{\cL^2(\cU)}^s
]
\leq 
\|
\cB
\|^s
c_1
\bE [
\| \mu^N - \mu \|_p^s
]
\leq 
\frac{
\|
\cB
\|^s
c_1
\sqrt{C_{\Omega}}
}{
N^{s/2}
}.
$

Finally, the adjoint operator $\mathcal{B}^{\star}: \mathcal{L}^2(\mathcal{U}) \to  \mathcal{L}^2(\mathcal{U})$ satisfies, for any $h \in \cL^2(\cU)$, 
\[
 \langle h,
\mathcal{B}\mathcal{A}(\mu^N-\mu) 
\rangle_{\cL^2(\cU)} =   \langle \mathcal{B}^{\star}(h),
\mathcal{A}(\mu^N-\mu) \rangle_{\cL^2(\cU)}.
\]
Taking expectation first and then applying Fubini's theorem, we obtain 
\begin{multline*}
    \bE\left[ \langle h,
\mathcal{B}\mathcal{A}(\mu^N-\mu) \rangle_{\cL^2(\cU)} \right]=\bE\left[  \langle \mathcal{B}^{\star}(h),
\mathcal{A}(\mu^N-\mu) \rangle_{\cL^2(\cU)} \right]\\=
\int (\mathcal{B}^{\star}(h))(u) \bE \left[\int  e^{\| u-y\|^2} (\dd \mu^N - \dd \mu)(y)\right] \dd \mathcal{U}(u)=0. 
\end{multline*}
The proof is concluded.

\section{Proofs for Section \ref{subsection:application:mean:embedding}}

\subsection{Proof of Lemma \ref{lemma:near:unbiased:mean:embedding}}
\label{supplement:subsection:proof:unbiased:mean:embedding}

As in the statement of Condition \ref{condition:near:unbias}, the analysis here is conducted conditionally to $(\mu_i,Y_i)_{i=1}^n$ and we use the notation $\bE_n$ and $\bP_n$ to denote the conditional expectation and probability given $(\mu_i,Y_i)_{i=1}^n$.
The independence property of Condition~\ref{condition:near:unbias} clearly holds, together with the property on $(b_{N,i})_{i=1}^n$.
Let $x \in \HH_k$ be fixed. We have
\begin{align*}
\bE_n 
\left[
\left \langle x , x_{N,i} 
\right \rangle_{\Hk}
-
\left \langle x , x_{i}
\right \rangle_{\Hk}
\right]
= &
\bE_n 
\left[
\left \langle
x , \frac{1}{N} \sum_{j=1}^N k( X_{i,j} , \cdot ) 
\right \rangle_{\Hk}
-
\left \langle
x ,\int_{\Omega} k( t , \cdot ) \dd \mu_i(t) 
\right \rangle_{\Hk}
\right]
\\
= &
\bE_n 
\left[
\frac{1}{N} \sum_{j=1}^N x(X_{i,j})
-
\int_{\Omega} x(t) \dd \mu_i(t) 
\right]
\\= & 0.
\end{align*}

It remains to show the moment bound on $a_{N,i}$.
For $t \in \Omega$, we let $k_t = k( t , \cdot ) \in \HH_k$.
Note that
\[
a_{N,i}
=
\frac{1}{N}
\sum_{j=1}^N 
\left(
k_{X_{i,j}} 
-
\int_{\Omega} k_t \dd \mu_i(t)
\right)
\]
is an average of i.i.d. random centered elements of $\HH_k$. These random elements have norm bounded by $ B = 2 \sup_{u \in \Omega} \sqrt{k(u,u)}$. Hence, using \citet[Prop. 2 p. 345]{caponnetto2007optimal}, as in \citet[Sect. 7.3.1]{szabo2016learning}, we obtain, for $
\eta >0$,
\[ 
\bP_n 
\left( 
\|a_{N,i} \|_{\Hk}
\geq 2 
\left(
\frac{2 B}{n}
+ \frac{B}{\sqrt{n}}
\right)
\log( 2 / \eta )
\right)
\leq \eta.
\]
From there, it is simple to show that for any $s >0$, there is a constant $c_s$ such that
$\bE_n [ \|a_{N,i} \|_{\Hk}^s ] \leq c_s N^{-s/2}$.
This completes the proof.

\section{Proofs for Section \ref{subsection:application:sliced:wasserstein}}
	
\subsection{ Proof of Lemma \ref{lemma:example:sliced}} 

We write the proof only for the case $d \geq 2$, since the proof for $d =  1$ uses similar arguments and is simpler.
Fix $\delta$ with $0 < \delta < \epsilon$.
Let us fix $i$ and a realization of $\mu_i$ for which the almost sure statements in the lemma hold, and let us work conditionally to this realization of $\mu_i$.
Let us fix $\theta \in \cS^{d-1}$. 
We let $\cK$ be the support of $\mu_i$. 
We let $ \theta , v_2 , \ldots , v_d $ be an orthonormal basis of $\mathbb{R}^d$. We let $g : \mathbb{R}^d \to \mathbb{R}$ be the density of $\mu_i$
(which is zero outside of $\cK \subseteq \Omega$)
 and we let $h : \mathbb{R}^d \to \mathbb{R}$ be the function defined by, for $x \in \mathbb{R}^d$,
\[
g(x)
=
h(  x^\top  \theta , x^\top  v_2, \ldots ,
x^\top  v_d ).
\]
We let similarly 
$ I_{\cK}: \mathbb{R}^d \to \mathbb{R}$ be the function defined by, for $x \in \mathbb{R}^d$,
\[
 \mathrm{1}_{ x \in \cK}
=
I_{\cK}(  x^\top  \theta , x^\top  v_2, \ldots ,
x^\top  v_d ).
\]
Note that $\inf\{ h(x_1 , \ldots , x_d) ; I_{\cK}(x_1 , \ldots , x_d) = 1 \}
=
\inf \{ g(x) ; x \in \cK \}$.
Similarly $$\sup \{ h(x_1 , \ldots , x_d) ; I_{\cK}(x_1 , \ldots , x_d) = 1 \}
=
\sup \{ g(x) ; x \in \cK \}.$$ 
Note also that, by assumption,
$I_{\cK}$ is zero outside of $[-L,L]^d$, where $L = \kappa \sqrt{d}$.
Consider the set
\begin{align*}
S_1 =	
& \Big\{\ 
 x_1 \in \mathbb{R};
~ \text{the set} ~ \{   x_2 , \ldots , x_d \in \mathbb{R} ~ \text{s.t.} ~ h(x_1 , \ldots , x_d) >0 \} 
\\
& 
~ \text{has non-zero Lebesgue masure in $
\mathbb{R}^{d-1}$} 
   \Big\}.
\end{align*}

Exploiting the convexity of the support of $\mu_i$, this set $S_1$ is a segment of the form $[x_{\inf} , x_{\sup}]$, $(x_{\inf} , x_{\sup}]$, $[x_{\inf} , x_{\sup})$ or $(x_{\inf} , x_{\sup})$, with $[- \tau , \tau ] \subseteq (x_{\inf} , x_{\sup})$. Then for $x_1 \in (x_{\inf} , x_{\sup})$, the density of $ \theta^\top X_{i,1} $ at $x_1$, written $g_1(x_1)$, is
\begin{equation} \label{eq:gunxun}
\int_{[-L,L]^{d-1}}
I_{\cK}( x_1 , x_2, \ldots , x_d )
h(x_1 , x_2 , \ldots , x_d)
\dd x_2 \cdots \dd x_d.
\end{equation}
By the definition of the set $S_1$, we see that the density $g_1$ is strictly positive on $(x_{\inf} , x_{\sup})$, which shows that $F_{\mu_{i,\theta}}$ is bijective from $(x_{\inf} , x_{\sup})$ to $(0,1)$. We write $F_{i,\theta}^{-1}$ for its inverse function.

If $F_{i,\theta}^{-1}(\delta) \leq 0 $, consider the following.
Since the function $h I_{\cK}$ is bounded by $T$, there is a deterministic constant $c_1 >0$ such that there are $d-1$ segments
$[r_2, s_2], \ldots , [r_d , s_d] \subset \mathbb{R}$ of length at least $c_1$ such that for some $\bar{x}_1 \in [x_{\inf}, F_{i,\theta}^{-1}(\delta/2)] $, $\bar{x}_1 \theta + [r_2, s_2] v_2 + \cdots + [r_d , s_d] v_d  \subseteq \cK$.
Considering the convex hull of the union of $\bar{x}_1 \theta + [r_2, s_2] v_2 + \cdots + [r_d , s_d] v_d$ and $B(0,\tau)$, that belongs to $\cK$ which is convex, since $h$ is lower-bounded when $I_{\cK}$ is non-zero in \eqref{eq:gunxun}, the density $g_1$ is lower-bounded on $[ F_{i,\theta}^{-1}(\delta) , 0] $, by a deterministic constant.

If $F_{i,\theta}^{-1}(1-\delta) \geq 0 $, by the same reasoning, this density is also lower-bounded on $[ 0 ,  F_{i,\theta}^{-1}(1-\delta) ]$, by a deterministic constant.
In the end, in all cases, the density $g_1$ is lower-bounded on $[ F_{i,\theta}^{-1}(\delta) , F_{i,\theta}^{-1}(1-\delta)] $, by a deterministic constant. 

From  \eqref{eq:gunxun}, 
and since $h$ is upper bounded by $T$, then also $g_1$ is upper bounded by $c_2$, with a deterministic constant $c_2$.

Let us now show that $g_1$
is differentiable on $(x_{\inf},x_{\sup})$, with derivative bounded by a deterministic constant $c_3$. For this, we will use that $h$ is
differentiable on $\{  
x ; I_{\cK}(x) = 1 \}$, with gradient bounded in Euclidean norm by $T$.
Consider the set $S_{x_1}$
of the $ x_2 , \ldots , x_{d} $ such that the value of $x'_1 \mapsto I_{\cK}(x'_1,x_2,\ldots,x_d)$ is not locally constant at $x_1$.
It is sufficient so show that for $x_1 \in (x_{\inf},x_{\sup})$, 
$S_{x_1}$ has zero Lebesgue measure. This enables to conclude with the dominated convergence theorem applied to \eqref{eq:gunxun}.

 Since $x_1 \in (x_{\inf} , x_{\sup})$, by convexity, there exists $\tilde{x}_2 , \ldots, \tilde{x}_d$ such that $(x_1 , \tilde{x}_2 , \ldots, \tilde{x}_d)$ is in the interior of $\{ 
x ; I_{\cK}(x) = 1 \}$. 
Consider also $(\bar{x}_2 , \ldots, \bar{x}_d) $ in the interior of the convex set $C_{x_1}$ defined by $C_{x_1} = \{ 
x_2 , \ldots,  x_d ; 
I_{\cK}(x_1 , \ldots , x_d) = 1 \}$. Then by convexity, $(\bar{x}_2 , \ldots, \bar{x}_d) \not \in S_{x_1}$. Consider then $(\bar{x}_2 , \ldots, \bar{x}_d) \not \in C_{x_1}$. Since $\{  
x ; I_{\cK}(x) = 1 \}$ is closed, as $\cK$ is a probabilistic support, then $(\bar{x}_2 , \ldots, \bar{x}_d) \not \in S_{x_1}$.

Hence, we have shown that $S_{x_1}$ is included in the boundary of the closed convex set $ C_{x_1}$. This set has non-zero $d-1$-Lebesgue measure because  $x_1 \in (x_{\inf} , x_{\sup})$. Hence, it is well-known that this boundary of $C_{x_1}$ has zero $d-1$-Lebesgue measure. Hence, indeed $g_1$
is differentiable on $(x_{\inf},x_{\sup})$, with derivative bounded by a deterministic constant $c_3$.

Finally, we can exploit that the derivative of $F_{\mu_{i,\theta}}$ is $g_1$ to conclude that $F_{i,\theta}^{-1}$ is twice differentiable on $(\delta, 1- \delta)$ with first and second derivatives bounded in absolute value, by deterministic constants.

\subsection{Proof of Lemma \ref{lemma:near:unbiased:sliced:wasserstein}}
Let us fix $i$ and a realization of $\mu_i$ for which the almost sure statements in Condition \ref{cond:sliced:wasserstein} hold, and let us work conditionally to this realization of $\mu_i$.
In particular, in the proof, the notations $\bP$ and $\bE$ denote the conditional probability and expectation given $\mu_i$.
Fix $\theta \in \cS^{d-1}$ and $t \in (\epsilon,1-\epsilon)$.
For $j = 1 , \ldots, N$, we let $V_j = \theta^\top X_{i,j}$. 
We let $G$ and $G^{(N)}$ be the c.d.f. and empirical c.d.f. of $V_1,\ldots,V_N$. 
We also let $U_j = G(V_j)$, $j=1,\ldots,N$ and we remark that $U_1,\ldots,U_N$ are i.i.d random variables uniformly distributed on $[0,1]$, since $G$ is assumed to be bijective from $(a_i(\theta) , b_i(\theta)) $ to $(0,1)$.

We let $V_{(1)}  \leq \cdots \leq V_{(N)}$ be the order statistics of $V_1,\ldots,V_N$ and we define 
$U_{1}, \cdots, U_{N}$ 
with $U_j = G(V_j)$. We let
$U_{(1)}  \leq \cdots \leq U_{(N)}$, breaking ties in the same way as for  $V_{(1)}  \leq \cdots \leq V_{(N)}$. 
We let $\ell_N(t) \in \{1 , \ldots , N\}$ be the smallest integer larger or equal to $N t$. For convenience, we may write $\ell = \ell_N(t)$. 

If $\epsilon >0$,
we write $A_N$ for the event
\[
A_N 
= 
\{ 
V_{(\ell)} 
\in [ G^{-1}( \delta )  ,  G^{-1}( 1 - \delta ) ] 
\}
\]
and we let $\mathbf{1}_{A_N} $ be its indicator function and $A_N^c$ be its complement event.
If $\epsilon = 0$, we let $\mathbf{1}_{A_N}  = 1$ by convention. 
If $\epsilon >0$, note that $V_{(\ell)} \leq G^{-1}(\delta)$ implies that $G^{(N)}(G^{-1}(\delta)) \geq \ell/n$ and thus $G^{(N)}(G^{-1}(\delta)) \geq t \geq \epsilon$.
Hence, using Hoeffding inequality, one can see that there are deterministic constants $c_1$ and $c_2$ (not depending on $N,t,\theta$) with $c_2 >0$ such that 
$\mathbb{P}( V_{(\ell)} \leq G^{-1}(\delta) ) 	\leq 
c_1 e^{- c_2 N}$.
Reasoning similarly on the event $V_{(\ell)} \geq G^{-1}(1 - \delta)$, up to changing $c_1$ and $c_2$, we have
\begin{equation} \label{eq:after:hoeffding}
	\bP \left(
A_N^c	
	\right)
	\leq 
	c_1 e^{- c_2 N}.
\end{equation}
Of course, \eqref{eq:after:hoeffding} also holds for $\epsilon = 0$.
Then, writing $G^{(N)-1} = (G^{(N)})^{-1}$, 
\begin{align*}
\bE 
\left[ 
G^{(N)-1}( t )
-
G^{-1}(t)
\right] 
 = &
 \underbrace{
 \bE 
 \left[ 
 \mathbf{1}_{A_N^c} 
 G^{(N)-1}( t )
 -
  \mathbf{1}_{A_N^c} 
 G^{-1}(t)
 \right]
}_{= B_1} 
 \\
 & +
  \underbrace{
  \bE 
 \left[ 
 \mathbf{1}_{A_N} 
 G^{(N)-1}( t )
 -
 \mathbf{1}_{A_N} 
 G^{-1}(t)
 \right] 
}_{= B_2}.
\end{align*}
From \eqref{eq:after:hoeffding}, and because the values of $ G^{(N)-1}( t )$ and  $G^{-1}(t)$ are bounded by $\max \{ ||x|| ; x \in \Omega \}$, we obtain 
\begin{equation} \label{eq:bias:quantile:Bun}
|B_1| \leq 
\frac{c_3}{N},
\end{equation}
for a constant $c_3$. Next,
\[
B_2
= 
\bE 
\left[ 
\mathbf{1}_{A_N} 
V_{(\ell)}
-
\mathbf{1}_{A_N} 
G^{-1}(t)
\right]
=
\bE 
\left[ 
\mathbf{1}_{A_N} 
G^{-1} (U_{(\ell)})
-
\mathbf{1}_{A_N} 
G^{-1}(t)
\right].
\]
By a Taylor expansion, exploiting the fact that the event $A_N$ holds in the right-most expectation above, we obtain, with a random $\xi \in ( \delta,  1 - \delta )  $, writing $G^{-1'}$ and  $G^{-1''}$ for the first and second derivatives of  $G^{-1}$ on $( \delta,  1 - \delta )$,
\begin{align*}
B_2 
= &
\bE 
\left[ 
\mathbf{1}_{A_N} 
G^{-1'} (t )
\left( 
U_{(\ell)}
-
t
\right)
+
\frac{
\mathbf{1}_{A_N} 
}{
2
}
G^{-1''} (\xi )
\left( 
U_{(\ell)}
-
t
\right)^2
\right]
\\
= &
G^{-1'} (t )
\bE 
\left[ 
U_{(\ell)}
-
t
\right]
-
G^{-1'} (t )
\bE 
\left[ 
\mathbf{1}_{A_N^c} 
\left( 
U_{(\ell)}
-
t
\right)
\right]
+
\bE 
\left[
\frac{
	\mathbf{1}_{A_N} 
}{
	2
}
G^{-1''} (\xi )
\left( 
U_{(\ell)}
-
t
\right)^2
\right].
\end{align*}
Above, $U_{(\ell)}$ follows the $\cB(\ell,N+1-\ell)$ distribution, where $\cB$ stands for the Beta distribution. 
Hence, in the above display, the first expectation is of order $1/N$ since $| Nt - \ell  | \leq 1$.
The second expectation is of order  $1/N$  from the same arguments as before \eqref{eq:bias:quantile:Bun}. 
The quantities $G^{-1'} (t )$ and $G^{-1''} (\xi )$ are bounded by $c^{(2)}$ from Condition \ref{cond:sliced:wasserstein} because $t$ and $\xi$ are in $(\delta , 1- \delta )$. Hence, 
the third expectation above is of order $1/N$ also because  $U_{(\ell)}$ follows the $\cB(\ell,N+1-\ell)$ distribution. Thus, using \eqref{eq:bias:quantile:Bun}, we have
\begin{equation} \label{eq:sliced:bNi}
\left| 
\bE 
\left[ 
G^{(N)-1}( t )
-
G^{-1}(t)
\right] 
\right| 
\leq \frac{c_4}{N},
\end{equation}
for a constant $c_4$. 

Then for $\theta \in \cS^{d-1}$ and $t \in (\epsilon,1-\epsilon)$,
we let $F^{(N)-1}_{\mu_{i,\theta}} = ( F^{(N)}_{\mu_{i,\theta}} )^{-1}$, where $F^{(N)}_{\mu_{i,\theta}}$ is the empirical c.d.f. of $\theta^\top X_{i,1} , \ldots , \theta^\top X_{i,N}$, 
and we define  
\[
a_{N,i}(\theta,t)
=
F^{(N)-1}_{\mu_{i,\theta}}
(t)
-
\bE 
\left[
F^{(N)-1}_{\mu_{i,\theta}}
(t)
\right]
\]
and 
\[
b_{N,i}(\theta,t)
=
\bE 
\left[
F^{(N)-1}_{\mu_{i,\theta}}
(t)
\right]
-
F_{\mu_{i,\theta}}^{-1}(t).
\]
Hence $a_{N,i}(\theta,t) + b_{N,i}(\theta,t) = F^{(N)-1}_{\mu_{i,\theta}}
(t) - F_{\mu_{i,\theta}}^{-1}(t) = x_{N,i}(\theta,t) - x_i(\theta,t)$.
Also, \eqref{eq:general:cond:aNi:unbiased}  is simple to show.
Since $b_{N,i}$ is deterministic, \eqref{eq:sliced:bNi} implies \eqref{eq:general:cond:bNi}. It thus remains to prove \eqref{eq:general:cond:aNi}.

Let us fix $s \geq 1$.
We have, using Jensen's inequality twice,
\begin{align} \label{eq:sliced:Ani}
	\bE [ \|  a_{N,i} \|_{\HH}^s ]
= &
	\bE
	\left[ 
 \left(
	\frac{1}{1-2 \epsilon}
	\int_{\cS^{d-1}} 
	\int_{\epsilon}^{1- \epsilon}
	\left(
	F^{(N)-1}_{\mu_{i,\theta}}
	(t)
	-
	\bE 
	\left[
	F^{(N)-1}_{\mu_{i,\theta}}
	(t)
	\right]
	\right)^2
	\dd \Lambda (\theta)
	\dd t 	
	\right)^{s/2}
 \right]
	\notag
\\ 
& \leq 
\sqrt{
	\bE 
		\left[ 
  \left(
	\frac{1}{1-2 \epsilon}
	\int_{\cS^{d-1}} 
	\int_{\epsilon}^{1- \epsilon}
	\left(
	F^{(N)-1}_{\mu_{i,\theta}}
	(t)
	-
	\bE 
	\left[
	F^{(N)-1}_{\mu_{i,\theta}}
	(t)
	\right]
	\right)^2
	\dd \Lambda (\theta)
	\dd t 	
\right)^s
\right]
}
\notag
\\ 
& \leq 
\sqrt{
	\bE 
 \left[
	\frac{1}{1-2 \epsilon}
	\int_{\cS^{d-1}} 
	\int_{\epsilon}^{1- \epsilon}
	\left(
	F^{(N)-1}_{\mu_{i,\theta}}
	(t)
	-
	\bE 
	\left[
	F^{(N)-1}_{\mu_{i,\theta}}
	(t)
	\right]
	\right)^{2s}
	\dd \Lambda (\theta)
	\dd t 	
 \right]
}.
\end{align}
We now fix $t \in (\epsilon,1- \epsilon)$ and $\theta \in \cS^{d-1}$.
We use the same notation as above: $G$, $V_1,\ldots,V_N$, $U_1,\ldots,U_N$, $A_N$. We then study the above integrand
\begin{align*}
&
F^{(N)-1}_{\mu_{i,\theta}}
	(t)
	-
	\bE 
	\left[	F^{(N)-1}_{\mu_{i,\theta}}
	(t)
	\right] \\
= &	G^{(N)-1}
(t)
-
\bE 
\left[
G^{(N)-1}
(t)
\right]
\\
 = &
 V_{(\ell)}
 -
 \bE 
  V_{(\ell)} 
  \\
  = &
  G^{-1} (U_{(\ell)})
  -
  \bE 
  G^{-1} (U_{(\ell)}) 
  \\
  = &
  \underbrace{
  \mathrm{1}_{A_N }
    G^{-1} (U_{(\ell)})
    -
      \bE [
   \mathrm{1}_{A_N }   G^{-1} (U_{(\ell)})]
}_{C_2}
   +
     \underbrace{
     \mathrm{1}_{A_N^c}
    G^{-1} (U_{(\ell)})
    -
    \bE [
    \mathrm{1}_{A_N^c}   G^{-1} (U_{(\ell)})
    ]
}_{C_1}.
\end{align*}
From the same arguments as before \eqref{eq:bias:quantile:Bun},  we have 
\begin{equation} \label{eq:sliced:Cun}
\bE 
[
|C_1|^{2s}
]
\leq 
\frac{c_{s,5}}{N^s}, 
\end{equation}
for a constant $c_{s,5}$ (not depending on $N$, $\theta$, $t$). 

To study $C_2$, we use similarly as before that under the event $A_N$ there is a random $\xi \in (\delta,  1- \delta)$ such that 
\begin{align*}
C_2 
= &
\mathrm{1}_{A_N }
G^{-1} \left( \frac{\ell}{N+1} \right)
+ 
\mathrm{1}_{A_N }
G^{-1'} \left(\xi \right)
\left(
U_{(\ell)}
-
 \frac{\ell}{N+1} 
\right)
\\
& -
\bE 
\left[
\mathrm{1}_{A_N }
G^{-1} \left( \frac{\ell}{N+1} \right) 
\right]
- 
\bE 
\left[
\mathrm{1}_{A_N }
G^{-1'} \left( \xi \right)
\left(
U_{(\ell)}
-
\frac{\ell}{N+1} 
\right)
\right].
\end{align*}
Above, $|G^{-1'} \left( \xi \right) |$ is bounded by $c^{(2)}$ from Condition \ref{cond:sliced:wasserstein} since $A_N$ holds. Using \eqref{eq:after:hoeffding} as above, we can show, for a constant $c_{s,6}$,
\[
\bE [ |C_2|^{2s} ]
\leq 
c_{s,6} 
\bE 
\left[
\left| 
U_{(\ell)}
-
\frac{\ell}{N+1} 
\right|^{2s}
\right]
+
\frac{c_{s,6}}{N^s}. 
\]
We can finally use \citet[Thm. 3]{skorski2023bernstein}, together with a simple induction, to obtain
\[
\bE [ |C_2|^{2s} ]
\leq 
\frac{c_{s,7}}{N^s}
\]
for a constant $c_{s,7}$. Combined with \eqref{eq:sliced:Ani} and \eqref{eq:sliced:Cun}, we thus obtain that \eqref{eq:general:cond:aNi} holds for $s \geq 1$. From Jensen's inequality,  \eqref{eq:general:cond:aNi} also holds for $s \leq 1$ and the proof is concluded.

 \end{document}